\documentclass{article}

\usepackage{amsmath,amsthm}
\usepackage{latexsym}
\usepackage{amssymb}
\usepackage{colonequals}
\usepackage{todonotes}
\newtheorem{thm}{Theorem}[section]
\newtheorem{la}[thm]{Lemma}
\newtheorem{Defn}[thm]{Definition}
\newtheorem{Remark}[thm]{Remark}
\newtheorem{Conj}[thm]{Conjecture}
\newtheorem{prop}[thm]{Proposition}
\newtheorem{cor}[thm]{Corollary}
\newtheorem{Example}[thm]{Example}
\newenvironment{defn}{\begin{Defn}\rm}{\end{Defn}}

\newenvironment{rem}{\begin{Remark}\rm}{\end{Remark}}

\newtheorem{Number}[thm]{\!\!}
\newenvironment{numba}{\begin{Number}\rm}{\end{Number}}

\newcommand{\cO}{{\mathcal O}}

\newcommand{\cE}{{\mathcal E}}
\newcommand{\cL}{{\mathcal L}}
\newcommand{\cC}{{\mathcal C}}

\newcommand{\cF}{{\mathcal F}}
\newcommand{\cB}{{\mathcal B}}
\newcommand{\cS}{{\mathcal S}}

\newcommand{\cD}{{\mathcal D}}

\newcommand{\ve}{\varepsilon}
\newcommand{\R}{{\mathbb R}}
\newcommand{\N}{{\mathbb N}}

\newcommand{\Sph}{{\mathbb S}}
\newcommand{\mto}{\mapsto}
\newcommand{\sub}{\subseteq}
\newcommand{\wt}{\widetilde}

\newcommand{\wb}{\overline}
\newcommand{\bL}{{\mathcal L}}
\DeclareMathOperator{\id}{id}

\DeclareMathOperator{\pr}{pr}

\DeclareMathOperator{\Evol}{Evol}
\DeclareMathOperator{\Lf}{\mathbf{L}}

\newcommand{\coloneq}{\colonequals}
\newcommand{\cg}{{\mathfrak g}}

\newcommand{\Frechet}{Fr\'{e}chet }
\DeclareMathOperator{\Mes}{\mathcal{M}} %Measures
\DeclareMathOperator{\VC}{VC} %Vector charges
\DeclareMathOperator{\BC}{BC} %Vector charges with bounded variation
\DeclareMathOperator{\Var}{Var} %Variation
\DeclareMathOperator{\Md}{\Mes_{\mathrm{RN}}}
\DeclareMathOperator{\Mdna}{\Mes_{\mathrm{RN,na}}} %measures with density wrt. non-atomic measure
\DeclareMathOperator{\BV}{BV} %Functions of bounded variation
\DeclareMathOperator{\op}{op}
\DeclareMathOperator{\rc}{rc}
\DeclareMathOperator{\sta}{st}
\DeclareMathOperator{\AC}{AC}
\DeclareMathOperator{\evol}{evol}
\DeclareMathOperator{\diam}{diam}
\newcommand{\st}{{\hspace*{.3mm}\sta}}
\newcommand{\bI}{{\mathbb{I}}}

\begin{document}
$\;$\\[-27mm]
\begin{center}
{\bf\Large Manifolds of continuous BV-functions
and\\[2.3mm]
vector measure regularity
of Banach-Lie groups}\\[6mm]
{\bf Helge Gl\"{o}ckner,\footnote{Universit\"{a}t Paderborn, Institut für Mathematik, Warburger Str.\ 100,
33098 Paderborn,
Germany; glockner@math.uni-paderborn.de, asuri@math.uni-paderborn.de}
Alexander Schmeding,\footnote{NTNU, Department of Mathematical Sciences,
1338 Sentralbygg 2 Gl\o{}shaugen, Trondheim;
alexander.schmeding@ntnu.no,
ORCID: 0000-0001-9463-3674.}
and Ali Suri\hspace*{.1mm}${}^1$ }\vspace{2mm}
\end{center}
\begin{abstract}
\hspace*{-4.8mm}We construct a smooth Banach manifold $\BV([a,b],M)$
whose elements are suitably-defined
functions $f\colon [a,b]\to M$
of bounded variation with values in a smooth Banach manifold~$M$
which admits a local addition.
If the target manifold is a Banach--Lie group $G$, with Lie algebra $\cg$,
we obtain a Banach--Lie group $\BV([a,b],G)$
with Lie algebra $\BV([a,b],\cg)$.
Strengthening known regularity properties of Banach--Lie groups,
we construct a smooth evolution map
from a Banach space of $\cg$-valued vector measures
on $[0,1]$ to $\BV([0,1],G)$.\vspace{3mm}
\end{abstract}
% mfd of maps in nonlinear FA
(secondary)\\[2.3mm]
\vspace{-1mm}
\tableofcontents\vspace{4mm}

\section{Introduction and statement of main results}
By the classical theory of
differential equations
on Banach manifolds, each Banach--Lie group $G$
is a regular Lie group \cite{Mil,Nee},
and it even is $L^1$-regular in the sense
that each Bochner-integrable function $\gamma\colon [0,1]\to \cg$ to its Lie algebra has an absolutely continuous
evolution
$\eta\colon [0,1]\to G$
which depends smoothly on $[\gamma]\in L^1([0,1],\cg)$
(see \cite{Mea}).
We show that Banach--Lie groups have an even stronger
regularity property, which we call \emph{vector measure
regularity}. Suitable $\cg$-valued vector measures~$\mu$
on $[0,1]$ give rise to evolutions $\eta\colon [0,1]\to G$
which are functions of bounded variation,
with smooth dependence on~$\mu$.
To achieve this goal,
topics are addressed which are
irrespective of Lie theory:
\begin{itemize}
\item
For real numbers $a<b$,
we define and study $\BV$-functions $[a,b]\to M$
with values in a smooth Banach manifold~$M$.
\item
For suitable $M$,
we construct a smooth manifold structure
on the set $\BV([a,b],M)$
of all $M$-valued $\BV$-functions.
\item
We define and study differential
equations whose solutions are $\BV$-functions
with values in a Banach space, or in a Banach
manifold.
\end{itemize}
Let us become more specific and provide background and context.\\[2.3mm]
{\bf Manifolds of mappings.}
Manifolds of mappings are important tools in global analysis
and geometry.
It is a classical fact
that the set $C^k(N,M)$ of $C^k$-maps
$f\colon N\to M$
can be made a smooth manifold
for each compact manifold~$N$ (which may have a boundary),
finite-dimensional smooth manifold~$M$, and $k\in \N_0\cup\{\infty\}$.
More generally, $M$ can be a smooth manifold modeled on locally convex spaces
which admits a local addition
(see \cite{Eel,Mic,KaM,AGS}
and the references therein).
Analogous constructions are possible for mappings
of other classes of regularity, like $H^s$-maps
(cf.\ \cite{EaM,Pal}).\\[2.3mm]
For a smooth manifold~$M$
(modeled on a locally convex space) and $p\in [1,\infty]$,
the set $\AC_{L^p}([0,1],M)$ of $M$-valued absolutely
continuous functions with $L^p$-derivatives in local charts
can be made a smooth manifold, if~$M$
admits a local addition (see \cite{Pin};
for Hilbert manifolds, see \cite{Sme};
for $G$ a Lie group, see \cite{Mea,Nik}).
The subset of $f\in \AC_{L^p}([0,1],M)$
with $f(0)=f(1)$ is a submanifold which can be identified with
$\AC_{L^p}(\Sph_1,M)$.
For $M$ a compact Riemannian manifold, $\AC_{L^2}(\Sph_1,M)$
has first been used in \cite{FaK,Kli}
to show the existence of closed geodesics.
It is also used in string topology (see \cite{Ste}
and
its references).\\[2.3mm]
{\bf Manifolds and Lie groups of BV-functions.}
Beyond absolutely continuous functions,
scalar-valued
functions of bounded variation are a classical
topic of analysis with specific applications;
$\BV$-functions with values in Banach spaces have been treated in~\cite{MaV}.
We specialize their approach to continuous functions.
Moreover, for $M$ a manifold modeled on a Banach space~$E$,
we define
a class of $M$-valued functions
$f\colon [0,1]\to M$ of bounded variation
($\BV$-functions, in short).
The functions in the class are continuous;
in local charts,
they are distribution functions of suitable
non-atomic $E$-valued vector measures.
We show:
\begin{thm}\label{fithm}
$\BV([0,1],M)$ can be turned into a smooth
Banach manifold for each smooth Banach manifold~$M$
admitting a local addition,
using modeling spaces and local charts
as described in Proposition~{\rm\ref{BVmfd-struct}}.
\end{thm}
\noindent
Each Banach--Lie group admits a smooth local
addition (see \cite[C2]{Schm}), whence Theorem~\ref{fithm} applies.
We show (see Theorem \ref{prop:Lie_grp}):
\begin{thm}\label{thm-is-Lie}
For each Banach--Lie group~$G$
with Lie algebra $\cg$,
the pointwise
group operations make
$\BV([0,1],G)$ a Banach--Lie group
whose Lie algebra is isomorphic to
$\BV([0,1],\cg)$ with the pointwise
Lie bracket as a topological Lie algebra.
\end{thm}
\noindent
Among others, we have applications in mind in the regularity
theory of infinite-dimensional Lie groups.\\[2.3mm]
{\bf Regularity properties of Lie groups.}
For a Lie group~$G$ modeled on a locally convex space,
consider the right action
\begin{equation}\label{leftact}
\sigma\colon TG\times G \to TG,\quad (g,v)\mto v.g:=TR_g(v)
\end{equation}
of $G$ on its tangent bundle, where $R_g\colon G\to G$, $x\mto xg$.
Let $e$ be the neutral element of~$G$ and
$\cg:=\Lf(G):=T_eG$ be its Lie algebra.
For $k\in \N_0\cup\{\infty\}$,
a Lie group~$G$ modeled on a locally
convex space is called \emph{$C^k$-regular}
if the initial value problem
\begin{equation}\label{iniva}
\dot{\eta}(t)=\gamma(t).\eta(t),\quad \eta(0)=e
\end{equation}
has a (necessarily unique) $C^{k+1}$-solution $\Evol(\gamma):=\eta\colon [0,1]\to G$
for each $C^k$-curve $\gamma\colon [0,1]\to \cg$
and the map $\Evol\colon C^k([0,1],\cg)\to C^{k+1}([0,1],G)$
is smooth\footnote{Which is equivalent
to smoothness of $\Evol\colon C^k([0,1],\cg)\to C([0,1],G)$
and to smoothness of the time-$1$-map
$\evol\colon C^k([0,1],\cg)\to G$,
$\gamma\mto \Evol(\gamma)(1)$, cf.\ \cite{Sem,GaN}.}
(see \cite{Sem,GaN}).
Then $G$ is $C^\ell$-regular for all $\ell\geq k$.
The $C^\infty$-regular Lie groups
are simply called \emph{regular};
for Lie groups with sequentially complete
modeling spaces, they were first considered by Milnor~\cite{Mil}.
Regularity is an important tool in the theory of infinite-dimensional
Lie groups (see \cite{Mil,Nee,GaN}, cf.\ \cite{KaM}).\\[2.3mm]
For $p\in [1,\infty]$
and $G$ a Lie group modeled on a sequentially
complete locally convex space, $\AC_{L^p}([0,1],G)$
is a Lie group (see \cite{Nik}, cf.\ \cite{Mea}).
The Lie group~$G$ is called \emph{$L^p$-regular}
if (\ref{iniva}) has a Carath\'{e}odory
solution $\Evol([\gamma]):=\eta\in \AC_{L^p}([0,1],G)$
for all $[\gamma]\in L^p([0,1],\cg)$
and $\Evol\colon L^p([0,1],\cg)\to \AC_{L^p}([0,1],G)$
is smooth (see \cite{Mea,Nik} for details).\footnote{It is equivalent to require
that $\Evol\colon L^p([0,1],\cg)\to C([0,1],G)$ is smooth,
see \cite{Mea,Nik}.}
Each $L^p$-regular Lie group is $L^q$-regular for all
$q\geq p$ and it is $C^0$-regular.
See \cite{Mea,Nik,GaH} for further information.
By the preceding,
there is a hierarchy of regularity properties;
for a Lie group $G$ modeled on a locally convex space~$E$,
we have\\[2mm]
\hspace*{10mm}$L^1$-reg.\ $\Rightarrow$ $L^p$-reg.\ $\Rightarrow$
$L^\infty$-reg.\
$\Rightarrow$ $C^0$-reg.\ $\Rightarrow$ $C^k$-reg.\ $\Rightarrow$
regular,\\[2mm]
assuming $E$ sequentially complete for the measurable
regularity properties.\pagebreak

$\;$\\[-12mm]
Strengthened regularity properties
are useful. For example,
the Trotter Product Formula and Commutator Formula
hold for each $L^\infty$-regular
Lie group (cf.\ \cite{Mea});
this remains valid for $C^0$-regular
Lie groups (cf.\ \cite{Hn2}).\\[2.3mm]
Each Banach--Lie group is $L^1$-regular~\cite{Mea}.
We'd like to get beyond this
and connect functions
$\eta\in \BV([0,1],G)$ to vector measures $\mu$
in a suitable Banach space $\Mdna([0,1],\cg)$
of $\cg$-valued vector measures.\\[2.3mm]
{\bf Vector measure regularity for Banach--Lie groups.}
Let $G$ be a Banach--Lie group with Lie algebra
$\cg$. The smooth right action $\sigma$ from
(\ref{leftact}) restricts to a smooth mapping
\begin{equation}\label{partial-action}
f\colon \cg\times G\to TG,\quad (v,g)\mto v.g=T_eR_g(v)
\end{equation}
such that $f(\cdot,g)\colon \cg\to T_gG$
is linear for each $g\in G$.
If $\gamma\in L^p([0,1],\cg)$
is replaced with suitable $\cg$-valued
vector measures $\mu$ on $[0,1]$,
we are able to give a sense to an initial
value problem
\begin{equation}\label{new-iniva}
\dot{\eta}=f_*(\mu,\eta),\quad \eta(0)=e
\end{equation}
that replaces~(\ref{iniva});
see \ref{situ-m-2}.
Its solutions are functions
$\eta\colon [0,1]\to G$ of bounded variation.
\begin{thm}\label{Banach-Lie-VM-reg}
For each Banach--Lie group $G$,
the initial value problem {\rm(\ref{new-iniva})}
has a necessarily unique $\BV$-solution
$\Evol(\mu):=\eta\colon [0,1]\to G$
for each $\mu\in \Mdna([0,1],\cg)$ and the map
$\Evol\colon \Mdna([0,1],\cg)\to \BV([0,1],G)$ is smooth.
\end{thm}
\noindent
To each $\BV$-function $\eta\colon [0,1]\to G$,
a right logarithmic derivative
$\delta^r(\eta)\in \Mdna([0,1],\cg)$
can be associated
such that $\delta^r\!\colon \!\!\BV([0,1],G)\! \to \! \Mdna([0,1],\cg)$
is a left inverse of
$\Evol$
(cf.\ Proposition~\ref{props-logder}).
We deduce that
the map
\[
\Mdna([0,1],\cg)\times G\to \BV([0,1],G),
\;\;
(\mu,g)\mto \Evol(\mu)g
\]
is a $C^\infty$-diffeomorphism, for each Banach--Lie group $G$ (see Proposition~\ref{semid-again}).\\[2.3mm]
\noindent
{\bf Applications in rough path theory.}
Paths of bounded variation with values in Banach spaces also feature prominently in the theory of rough paths (cf.\ \cite{FaV10} or the more recent \cite{FaH14}, which however uses the dual picture of H\"older continuous maps). Rough path theory analyzes differential equations driven by irregular signals, in particular for paths of finite $p$-variation. Recall that $\BV$-paths are $1$-variation paths, whence they belong to the most regular class of paths considered in the theory. While the $\BV$-paths are too simple to necessitate the full power of rough path theory, they are interesting for analysis due to their connection with the signature. Associating to a $\BV$-path the family of iterated integrals of the path against itself, one obtains a path in the tensor algebra (see \cite[Definition 7.2]{FaV10}) and it turns out that up to a natural equivalence relation the $\BV$-path is uniquely determined by the signature  \cite{HaL10}. Moreover, the signature takes values in an infinite-dimensional subgroup of the tensor algebra (which unfortunately is at best a \Frechet algebra but not a Banach algebra, see, e.g., \cite[Chapter 8]{Schm} or \cite{GaNaS22} for the Banach case). Truncating the signature at the $N$-th tensor level yields the truncated signature of a path and this path is a $\BV$-path taking values in $\mathcal{G}^N(E)$, the free nilpotent Lie group of step $N$ over the Banach space $E$, \cite[Theorem 7.30]{FaV10}. Hence the truncated signature maps $E$-valued $\BV$-paths to the Lie group $\BV([0,1],\mathcal{G}^N(E))$ constructed in Theorem~\ref{thm-is-Lie}.
%\ref{prop:Lie_grp} below.
We remark that there is an intricate relation between the signature and the Lie group of $\BV$-paths and its regularity (see again \cite[Chapter 8]{Schm} for more information).
Signatures have many applications (see, e.g., \cite{Sig1,Sig2,Sig3} and the references therein) and it would be interesting to investigate the implications of our results for the analysis of signatures. Note, however, that we only treat $\BV$-paths but not $p$-variation paths for $p>1$. Our results are therefore too restrictive
to tackle finer
theoretical questions arising at the interface of rough path theory and (infinite-dimensional) differential geometry and Lie theory.\\[2.3mm]
{\bf Structure of the article.}
We first fix general notation and describe
the class of vector measures and
vector-valued $\BV$-functions
which are suitable for our goals (Section~\ref{secprels}).
Basically, the concepts are as in~\cite{MaV},
where a Chain Rule was established for
the composition $f\circ g$ of a scalar-valued continuously
Fr\'{e}chet differentiable function $f$ and
a vector-valued $\BV$-function~$g$.
But we restrict attention to non-atomic vector measures
and continuous $\BV$-functions,
whence the chain rule of~\cite{MaV} simplifies to a form
which carries over to $f\circ g$ for vector-valued~$f$
(Proposition~\ref{chainrule}, Corollary~\ref{chainrule2}). This is essential, as
it enables $\BV$-functions $[a,b]\to M$
to Banach manifolds to be defined.
In Section~\ref{sec-bilin-meas},
we give a streamlined exposition of vector measures
with a density,
with a view towards our
%intended
applications
(see \cite{Din} for a comprehensive treatment).
Our notation for $f\mu$ is $f\odot_\beta \mu$,
if $\beta\colon E_1\times E_2\to F$ is a continuous bilinear
map between Banach spaces, $\mu\colon \cS\to E_2$
a vector measure of bounded variation
on a measurable space $(X,\cS)$
and $f\colon X\to E_1$
a totally measurable function in the sense of
\cite[\S6.2, Definition~2]{Din}
(called an $\cL^\infty_{\rc}$-function in \ref{L-infty-concept}).\footnote{The concepts coincide
in this situation by \cite[Remark~2 in \S9.1]{Din}.}
A special case, denoted $f_*(\mu,\gamma)$,
is introduced and studied in Section~\ref{sec-the-rhs};
such measures
appear later as the right-hand sides
of differential equations.
Continuity of superposition operators and composition
operators on spaces of vector-valued $\BV$-functions
(and open subsets) are studied in Section~\ref{sec-superpo},
and corresponding operators between Banach spaces $\Gamma_f
\sub \BV([a,b],TM)$ for real numbers $a<b$
and $f\in \BV([a,b],M)$
(Section~\ref{sec-model}).
They enable a smooth manifold structure
to be defined on the set $\BV([a,b],M)$
of $M$-valued $\BV$-functions, for each Banach manifold~$M$ admitting a local addition (see Section~\ref{sect:BV-mfd}).
Its modeling space at $f\in \BV([a,b],M)$
is $\Gamma_f$. Some technical aspects
concerning the tangent bundle
of $\BV([a,b],M)$ have been relegated to
an appendix (Appendix~\ref{app:details_tangent}).
Superposition operators and other relevant maps
between manifolds of $\BV$-functions are studied in Section~\ref{sec-canonical}.
Notably, they allow $\BV([a,b],G)$ to be considered
as a Banach--Lie group for each Banach--Lie group~$G$.
In Section~\ref{sec-BV-ode}, we introduce
differential equations for $\BV$-function both locally in
open subsets of a Banach space and globally on a Banach manifold.
Local existence, local uniqueness, and parameter dependence
of solutions are addressed. This enables us to
define and establish
vector measure regularity for Banach--Lie groups
(Section~\ref{sec-vm-reg}).\\[2.3mm]
\textbf{Acknowledgements.}
The second author would like to thank the mathematical\linebreak
institute at the University of Paderborn for its hospitality while conducting the work presented in this article.
He was supported by Norwegian Research council, IS-DAAD-Forskerutveksl. Norge-Tyskland project number 318974.
The first and third authors were supported
by Deutsche Forschungsgemeinschaft (DFG),
project~517512794.
\section{Preliminaries and notation}\label{secprels}
\noindent
We now introduce concepts for later use and
collect basic facts.
\subsection*{Basic notation}
\noindent
We write $\N:=\{1,2,\ldots\}$
and $\N_0:=\N\cup\{0\}$.
We abbreviate
``Hausdorff locally convex topological
$\R$-vector space''
as ``locally convex space.''
If $(E,\|\cdot\|_E)$
is a normed space, $x\in E$ and $r>0$,
we write $B^E_r(x):=\{y\in E\colon \|y-x\|_E<r\}$
and $\wb{B}^E_r(x):=\{y\in E\colon \|y-x\|_E\leq r\}$
for balls.
\subsection*{Infinite-dimensional calculus}
\noindent
We work in the setting of differential calculus
going back to Andr\'{e}e Bastiani~\cite{Bas}
(see \cite{Res, GaN, Ham, Mic, Mil, Nee, Schm}
for discussions in
varying generality),
also known as Keller's $C^k_c$-theory~\cite{Kel}.
\begin{numba}\label{defn-Ck}
Consider locally convex spaces $E$, $F$
and a map $f\colon U\to F$
on an open subset $U\sub E$.
Write
\[
(D_yf)(x):=df(x,y):=\frac{d}{dt}\Big|_{t=0}f(x+ty)
\]
for the directional derivative of~$f$ at $x\in U$
in the direction $y\in E$, if it exists.
Let $k\in \N_0\cup\{\infty\}$.
If $f$ is continuous, the iterated directional derivatives
\[
d^jf(x,y_1,\ldots, y_j):=(D_{y_j}\ldots D_{y_1}f)(x)
\]
exist for all $j\in\N_0$ such that
$j\leq k$, $x\in U$ and $y_1,\ldots, y_j\in E$,
and the maps $d^jf\colon U\times E^j\to F$
are continuous, then $f$ is called~$C^k$.
If $U$ may not be open,
but has dense interior~$U^0$
and is locally convex in the sense
that each $x\in U$ has a convex neighborhood
in~$U$, following~\cite{GaN}
a map $f\colon U\to F$ is called $C^k$
if it is continuous,
$f|_{U^0}$ is $C^k$
and $d^j(f|_{U^0})$ has a continuous
extension $d^jf\colon U\times E^j\to F$
for all $j\in \N_0$ with $j\leq k$.
The $C^\infty$-maps are also called~\emph{smooth}.
\end{numba}
\noindent
Sometimes, we shall use mappings between open subsets of Banach spaces which are differentiable in the \Frechet sense. To distinguish these from the Bastiani setting just recalled, we write $\text{FC}^k$ for a mapping which is $k$times continuously \Frechet differentiable. It is well known that each $C^{k+1}$-map on a Banach space is $\text{FC}^k$ and each $\text{FC}^k$ map is $C^k$ (see, e.g., \cite{Schm}).
\begin{numba}\label{conventions-mfd}
All manifolds and Lie groups considered
in the article are modeled on locally
convex spaces which may be infinite-dimensional,
unless the contrary is stated. All manifolds are assumed to be smooth and all diffeomorphisms are $C^\infty$-diffeomorphisms, unless the contrary is stated.
If $M$ is a manifold modeled on a Banach space~$E$
and $F\sub E$ is a closed vector subspace,
then a subset $N\sub M$ is called a \emph{submanifold}
of~$M$ modeled on~$F$ if, for each $p\in N$,
there is a $C^\infty$-diffeomorphism (chart)
$\phi\colon U_\phi\to V_\phi$ from an open $p$-neighborhood
$U_\phi\sub M$ onto an open subset $V_\phi\sub E$ such that
$\phi(U_\phi\cap N)=V_\phi\cap F$.
Submanifolds of Lie groups which are subgroups
are called \emph{Lie subgroups}.
\end{numba}
\begin{numba}\label{defn-df}
If $U$ is an open subset
of a locally convex space~$E$
(or a locally convex subset with dense interior),
we identify its tangent bundle $TU$ with $U\times E$,
as usual,
with bundle projection $(x,y)\mto x$.
If $M$ is a $C^k$-manifold
and $f\colon M\to U$
a $C^k$-map with
$k\geq 1$, we write $df$ for the second component
of $Tf\colon TM\to TU=U\times E$.
Thus $Tf=(f\circ \pi_{TM},df)$,
using the bundle projection $\pi_{TM}\colon TM\to M$.
\end{numba}
\subsection*{Vector measures}
\noindent
We follow \cite{Sche} and recall some basic notation for vector measures. In the following, let $(X,\cS)$ denote a measurable space (i.e.\ a set $X$ with a $\sigma$-algebra $\cS$). If $Y$ is a topological space, we write $\cB(Y)$
for the $\sigma$-algebra of Borel sets.
\begin{numba}
For a Banach space $(E,\lVert \cdot\rVert_E)$, we let $\VC (X,E)$ denote the set of \emph{vector charges}, i.e., finitely additive mappings $\mu \colon \cS \rightarrow E$.
If $E=\R$, we abbreviate $\VC(X)\coloneq \VC(X,\R)$.
If $\mu \in \VC(X)$ takes its values in $[0,\infty[$, we call $\mu$ a \emph{positive charge}. The set of positive charges is denoted by $\VC_+(X)$.
\end{numba}

\begin{numba}
 For a vector charge $\mu \in \VC(X,E)$, its \emph{variation} is the function
 \[
 \Var(\mu) \colon \cS \rightarrow [0,\infty],\;\;
 \Var(\mu) (A):=\sup \big(
 \|\mu(S_1)\|_E+\cdots + \|\mu (S_n)\|_E\big),
 \]
 for $n\in \N$ and all partitions
 of $A$ into disjoint measurable
 sets $S_1,\ldots, S_n$.
 If $\Var(\mu)(X)<\infty$, the charge $\mu$ has \emph{bounded variation}. Then $\Var (\mu)$ is a positive charge (see \cite[29.6 a]{Sche}).
 For a charge $\mu$ of bounded variation,
 $$\lVert \mu \rVert \coloneq \Var (\mu)(X)$$
 is called its \emph{total variation}.
 By \cite[29.6(c)]{Sche}, $\lVert \cdot \rVert$ is a norm on the vector space $\BC(X,E)$ of all charges
 of bounded variation
 and makes it a Banach space;
 we call it the \emph{variation norm}.
For each $S\in \cS$, the map
\begin{equation}\label{bounded-op}
\BC(X,E)\to E,\quad \mu\mto\mu(S)
\end{equation}
is linear and continuous with operator norm $\leq\|\mu\|$.
\end{numba}

\begin{numba}
Denote by $\Mes (X,E) \subseteq \BC(X,E)$ the set of \emph{$E$-valued measures of bounded variation}, i.e.\ charges of bounded variation
which are $\sigma$-additive mappings. The Nikodym Convergence Theorem \cite[29.8]{Sche}
implies that the measures form a closed vector subspace of $\BC(X,E)$;
thus $(\Mes(X,E),\lVert \cdot \rVert)$ is a Banach space.\\[2.3mm]
Further, we let $\Mes(X)\coloneq \Mes(X,\R)$ and
$\Mes_+(X)\coloneq \Mes(X)\cap \VC_+(X)$.
Given $\mu \in \Mes_+(X)$,
a function $f \colon X \rightarrow E$
is called \emph{integrable}
if it is a measurable
mapping from $(X,\cS)$ to $(E,\cB(E))$,
the image $f(X)$ is separable,
and $\|f\|_{\cL^1}:=\int_X\|f(x)\|_E\,{\rm d}\mu(x)<\infty$
(see \cite[22.28 and 21.4]{Sche}).
We write $\cL^1(\mu,E)$ for the vector space of
integrable functions and
$L^1(\mu,E)$
for the Banach space of equivalence classes
$[f]$ of such functions (modulo functions vanishing
almost everywhere), with norm $\|\,[f]\,\|_{L^1}:=\|f\|_{\bL^1}$.
Each integrable function $f$ admits a \emph{Bochner integral}
$\int_Xf\,{\rm d}\mu\in E$, see \cite[23.16]{Sche}.
\end{numba}
\noindent
Given $A\in \cS$, let us write $1_A\colon X\to \{0,1\}\sub\R$
for its characteristic function (indicator function).
\begin{numba}\label{with-density} (Vector measures with density).
If $\mu\in \Mes(X)_+$ and $f\in \cL^1(\mu,E)$,
then
\[
\cS\to E, \quad A\mto\int_X 1_Af\, {\rm d}\mu
\]
defines an element of $\Mes(X,E)$,
denoted $f\,{\rm d}\mu$ (see \cite[29.10]{Sche}).
Its total variation is given by\footnote{In fact, the formula is clear for measurable
functions with finite image (simple functions).
The general case follows since both sides
of the equation are continuous in $[f]\in L^1(\mu,E)$
and simple functions are dense in $L^1(\mu,E)$
(see \cite[22.30 b]{Sche}).}
\begin{equation}\label{tovar}
\|f \, {\rm d}\mu\|=\|f\|_{\cL^1}.
\end{equation}
\end{numba}
\begin{rem}\label{newrem}
By (\ref{tovar}), the linear map $L^1(\mu,E)\to \Mes(X,E)$,
$f\mto f\, {\rm d}\mu$
is an isometric embedding.
\end{rem}
\begin{numba}\label{defn-meas-d}
Define the space
$$\Md (X,E) \coloneq \{\mu \in \Mes (X,E)\mid \exists \nu \in \Mes_+(X), \exists \rho \in \cL^1(\nu,E) \text{ with }\mu = \rho \,\mathrm{d}\nu\}.$$
In other words, we consider the space of vector-valued measures which admit a Radon--Nikodym derivative with respect to some finite positive measure~$\nu$.
Then $\Var(\mu)=\|\rho(x)\|_E\, {\rm d}\nu(x)$
(cf.\ (\ref{tovar})).
Setting $h(x):=\rho(x)/\|\rho(x)\|_E$ if $\rho(x)\not=0$
and $h(x):=0$ otherwise, we obtain a function
$h\in \cL^1(\Var(\mu),E)$
such that $\mu=h\, {\rm d}\Var(\mu)$,
using the Change of Variables Formula \cite[29.12 a]{Sche}.
Thus~$\nu$ can always be chosen as
$\Var(\mu)$, if it exists.
\end{numba}
\noindent
If $\mu\in \Mes(X,E)$ and $\nu\colon \cS\to [0,\infty[$
is a finite positive measure, we say that $\mu$ is
\emph{absolutely continuous} with respect to $\nu$ (and we write
$\mu \ll \nu$) if $\mu(A)=0$
for each $A\in \cS$ such that $\nu(A)=0$.
Then also $\Var(\mu)\ll \nu$.
\begin{rem}
For a Banach space~$E$,
the following properties are
equivalent:\footnote{If $\mu\in \Mes(X,E)$, then $\mu\ll \Var(\mu)$.
Thus $\mu=\rho\,{\rm d}\Var(\mu)$ for some $\rho\in\cL^1(\Var(\mu),E)$
and hence $\mu\in \Md(X,E)$,
if $E$ has the (RNP).
If (b) holds and $\mu$, $\nu$
are as in~(a), then $\mu=\rho\,{\rm d}\Var(\mu)$
for some $\rho\in \cL^1(\Var(\mu),E)$.
Since $\mu\ll\nu$, also $\Var(\mu)\ll \nu$,
whence $\Var(\mu)=h\, {\rm d}\nu$
for some $h\in \cL^1(\nu,\R)$
by the Radon--Nikodym Theorem.
Then $\mu=\rho h\, {\rm d}\nu$
by \cite[29.12 b]{Sche}.}
\begin{itemize}
\item[(a)]
$E$ has
the Radon--Nikodym property (RNP),
i.e., for each finite measure space $(X,\cS,\nu)$
and $\mu\in \Mes(X,E)$ with $\mu\ll \nu$,
there exists $\rho\in \cL^1(\nu,E)$
such that $\mu=\rho\,{\rm d}\nu$
(see \cite[Definition 29.21]{Sche}).
\item[(b)]
$\Md(X,E)=\Mes(X,E)$ for each measurable space
$(X,\Sigma)$.
\end{itemize}
\noindent
By Phillip's Theorem,
each reflexive
Banach spaces has the (RNP) (see \cite[29.26]{Sche});
notably, any Hilbert space has the (RNP).
By the Theorem of Dunford--Pettis,
the dual $E:=F'$
of a Banach space $F$ has the (RNP)
%whenever
if $E$ is separable.
\end{rem}
\begin{la}
 The set $\Md(X,E)$ is a closed vector subspace of $\Mes (X,E)$,
 and hence a Banach space.
\end{la}

\begin{proof}
Assume that $\mu_n \in \Md(X,E)$ for $n \in \N$ converges in $\Mes (X,E)$ to $\mu$ with respect to the total variation norm. Write $\mu_n = \rho_n \,\mathrm{d}\nu_n$ with $\nu_n \in \Mes_+(X)$ and $\rho_n \in \cL^1(\nu_n,E)$. Since all the $\nu_n$ are finite positive measures, also
$$
\nu \coloneq \sum_{n=1}^\infty
\frac{\nu_n}{2^n(\nu_n(X)+1)}
$$
is a finite positive measure.
By construction, $\nu_n \ll \nu$, whence
$\nu_n = h_n \, {\rm d}\nu$ for some $h_n \in \cL^1 (\nu,\R)$, cf.\ \cite[Theorem 29.20]{Sche}. By the Change of Variables Formula \cite[29.12 b]{Sche},
we have $\mu_n=\rho_nh_n\,{\rm d}\nu$.
By Remark~\ref{newrem}, $([\rho_n h_n])_{n\in \N}$
is a Cauchy sequence in $L^1(\nu,E)$
and thus convergent to some $[\rho]\in L^1(\nu,E)$.
Then $\mu_n=\rho_nh_n\,{\rm d}\nu\to \rho\, {\rm d}\,\nu$
as $n\to\infty$, whence $\mu=\rho\, {\rm d}\,\nu\in \Md(X,E)$.
\end{proof}
\noindent
For our purposes, it will be important to work with continuous functions, whence we focus on non-atomic measures.
We recall: If $(X,\cS)$ is a measurable space
and $\mu\colon \cS\to E$ a vector measure,
then a set $A\in \cS$
is called an \emph{atom} if $\mu (A)\neq 0$ and for each $B \in \cS$ with $B \subseteq A$,
either $\mu(B) =\mu(A)$ or $\mu(B)=0$.  
\begin{numba}
A vector measure without atoms is called
\emph{non-atomic}. For $(X,\cS)$ a measurable space and $E$ a Banach space, we define
\[
\Mdna (X,E) \coloneq\{\mu \in \Md (X,E) \colon \mu \text{ is non-atomic}\}.
\]
\end{numba}
\begin{rem}\label{easy-nonatomic}
Let $\cS$ be as before and $\mu\in \Mes(X,E)$.
\begin{itemize}
\item[(a)]
If $A$ is an atom for $\mu$ and $A=A_1\cup\cdots\cup A_n$
with disjoint measurable sets, then there exists $k\in \{1,\ldots,n\}$ such that $\mu(A_k)=\mu(A)$ and $\mu(A_j)=0$
for all $j\in \{1,\ldots, n\}\setminus\{k\}$.\smallskip

\noindent
[In fact, if $\ell$ is the number of indices $j\in \{1,\ldots,n\}$ such that $\mu(A_j)\not=0$ and thus $\mu(A_j)=\mu(A)$, then $\mu(A)=\ell\mu(A)$ and thus $\ell=1$.]
\item[(b)]
If $A$ is an atom for $\Var(\mu)$,
then $\mu(B)\not=0$ for some $B\in \cS$
with $B\sub A$. Any such~$B$ is an atom for~$\mu$.\smallskip

\noindent
[For each $C\in \cS$ with $C\sub B$
and $\mu(C)\not=0$, we have $\Var(\mu)(C)\not=0$
and hence hence $\Var(\mu)(C)=\Var(\mu)(A)$.
Now $\lVert \mu(B\setminus C)\rVert_E\leq\Var(\mu)(B\setminus C)\leq \Var(\mu)(A\setminus C)=0$,
whence $\mu(C)=\mu(B)$.]
\item[(c)]
Consider finite positive measures $\mu$ and $\nu$
on $(X,\cS)$.
If $\nu$ is non-atomic and $\mu$ is absolutely
continuous with respect to~$\nu$,
then also $\mu$ is non-atomic.\smallskip

\noindent
This well-known fact is easy to check,
using that $\mu$ has a density with respect to $\nu$
by the Radon--Nikodym Theorem.
\end{itemize}
\end{rem}
\noindent
The following result is standard and repeated here for the reader's convenience.

\begin{la}\label{la:nonatomic}
A vector measure $\mu \in \Mes (X,E)$ is non-atomic if and
only if the positive measure $\Var(\mu)$ is non-atomic.
\end{la}

\begin{proof}
Let $\Var(\mu)$ be non-atomic and $A$ be a measurable set
with $\mu (A) \neq 0$. If $A$ was an atom for~$\mu$, then for each partition $A =B \cup C$ into disjoint measurable sets,
we must have
$\mu (B)=\mu(A)$ or $\mu (C)=\mu(A)$. By
Remark~\ref{easy-nonatomic}\,(a), this implies
$\Var(\mu)(C)=0$ or $\Var(\mu)(B)=0$.
Hence $A$ would be an atom for $\Var(\mu)$, contradiction.
Conversely: If $\mu$ is non-atomic, then also
$\Var(\mu)$, by Remark~\ref{easy-nonatomic}\,(b).
\end{proof}

\begin{la}\label{na-closed}
\,\,The set $\Mdna (X,E)$ is a vector subspace and closed
in\linebreak
$(\Md (X,E),\lVert \cdot \rVert)$,
hence a Banach space with respect to the variation norm. 
\end{la}

\begin{proof}
If $\nu\in \Mdna(X,E)$ and $r\in \R$,
then $\Var(r\nu)=|r|\Var(\nu)$ is non-atomic,
whence $r\nu\in \Mdna(X,E)$.
Let $(\mu_n)_{n\in \N}$ be a sequence in $\Mdna (X,E)$ such that $\sum_{n=1}^\infty \lVert \mu_n\rVert<\infty$.
Then the limit
$\mu := \sum_{n=1}^\infty \mu_n$
exists in $\Md(X,E)$, as $\Md(X,E)$ is a Banach space.
We claim that~$\mu$ is non-atomic.
If this always holds, choosing $\mu_n=0$ for $n\geq 3$
shows that $\Mdna(X,E)$ is closed under addition and hence
a vector subspace of $\Md(X,E)$.
Returning to the general case,
the claim shows that each absolutely convergent
sequence in the normed space $\Mdna(X,E)$ converges,
whence $\Mdna(X,E)$ is a Banach space and hence
closed in $\Md(X,E)$.
In view of Lemma \ref{la:nonatomic}, the claim
will hold if we can show that
$\Var (\mu)$ is non-atomic. Note that
\begin{equation}\label{hence-absoc}
\Var (\mu) = \Var \left(\sum_{n=1}^\infty \mu_n\right) \leq
\sum_{n=1}^\infty \Var (\mu_n)
\end{equation}
pointwise.
Both $\Var(\mu)$ and $\sum_{n=1}^\infty \Var(\mu_n)$
are finite positive measures (of total
mass $\leq \sum_{n=1}^\infty \lVert\mu_n\rVert$),
and the former is absolutely continuous with respect to the
latter, by~(\ref{hence-absoc}).
If we can show that $\sum_{n=1}^\infty \Var(\mu_n)$ is non-atomic, then also $\Var(\mu)$ is non-atomic, by Remark~\ref{easy-nonatomic}\,(c).
Let $A \subseteq X$ be a measurable set
with $\sum_{n=1}^\infty \Var (\mu_n)  (A)\neq 0$.
There exists $m \in \N$ with $\Var (\mu_m) (A) \neq 0$.
Since $\mu_m$ and also $\Var(\mu_m)$ is non-atomic, we find
a measurable set $B \subseteq A$ such that $0< \Var (\mu_m)(B)<\Var(\mu_m)(A)$. Then 
 $$0 < \sum_{n=1}^\infty  \Var(\mu_n)(B)
 < \sum_{n=1}^\infty \Var (\mu_n)(A).$$
Thus $A$ is not an atom for $\sum_{n=1}^\infty \Var(\mu_n)$
and $\sum_{n=1}^\infty \Var (\mu_n)$ is non-atomic.
\end{proof}
\noindent
Our interest lies in the special case when $X=[a,b]$ is a compact interval,
which we always endow with the
$\sigma$-algebra $\cB([a,b])$
of Borel sets.
\begin{la}\label{char-na}
For a Banach space $E$ and $\mu\in\Mes([a,b],E)$,
the following conditions are equivalent:
\begin{itemize}
\item[\rm(a)]
$\mu$ is non-atomic.
\item[\rm(b)]
$\mu(\{x\})=0$ for all $x\in [a,b]$.
\item[\rm(c)]
The function $f\colon [a,b]\to E$, $x\mto\mu([a,x])$
is continuous and $f(a)=0$.
\end{itemize}
\end{la}
\begin{proof}
(a)$\Rightarrow$(b): If $\mu(\{x\})\not=0$, then $\{x\}$
is an atom.\smallskip

\noindent
(b)$\Rightarrow$(a): If $A_0\in \cB([a,b])$
is a $\mu$-atom, we find a
sequence $A_0\supseteq A_1\supseteq A_2\supseteq\cdots$
of Borel sets such that $\mu(A_n)=\mu(A_0)$
for all $n\in \N_0$ and $A_n$ has diameter
$\diam(A_n)\leq 2^{-n}(b-a)$. We have $A_0$.
If $A_n$ has been found, write
$A_n=B_1\cup\cdots\cup B_{2^{n+1}}$ with
$B_j:=A_n\cap [a+(j-1)(b-a)2^{-n-1},a+j(b-a)2^{-n-1}]$
for $j\in \{1,\ldots, 2^{n+1}\}$.
Then $\mu(B_j)\not=0$ (and hence
$\mu(B_j)=\mu(A_0)$) for some $j$,
and we let $A_{n+1}:=B_j$.
Then $A:=\bigcap_{n\in \N_0}A_n$
is not empty as $\mu(A)=\lim_{n\to\infty}\mu(A_n)=\mu(A_0)$.
Since $\diam(A)=0$, we have $A=\{x\}$
for some $x\in [a,b]$. Then $\mu(\{x\})=\mu(A)=\mu(A_0)
\not=0$.\smallskip

\noindent
(b)$\Rightarrow$(c): We have $\Var(\mu)(\{x\})=\lVert\mu(\{x\})\rVert_E=0$ for all $x\in [a,b]$. Thus $f(a)=0$.
We now use that $\Var(\mu)$
is outer regular by \cite[Theorem 2.18]{Rud}.
Hence, given $\ve>0$, we find $\delta>0$
such that $\Var(\mu)(U)<\ve$ holds for the
neighborhood $U:=[x-\delta,x+\delta]\cap [a,b]$ of $x$
in $[a,b]$. For any $y\in U$, let $\alpha:=\min\{x,y\}$
and $\beta:=\max\{x,y\}$.
Then $\lVert f(y)-f(x)\rVert_E=\lVert\mu(]\alpha,\beta])\rVert_E\leq \Var(\mu)(U)<\ve$.\smallskip

\noindent
(c)$\Rightarrow$(b): We have $\mu(\{a\})=f(a)=0$.
For $x\in \;]a,b]$, we have $\mu(\{x\})=\lim_{y\to x_-}
\mu(]y,x])=\lim_{y\to x_-}(f(x)-f(y))=0$.
\end{proof}
\subsection*{Vector-valued functions of bounded variation}
\begin{numba}
We say that a function $f\colon [a,b]\to E$
to a Banach space~$E$
is \emph{of bounded variation} (or also a \emph{$\BV$-function}, or
a \emph{$\BV$-map})
if there exists a vector measure $\mu \in \Mdna([a,b],E)$
such that
\begin{equation}\label{stara}
f(x)=f(a)+\mu([a,x])\quad\mbox{for all $\,x\in [a,b]$.}
\end{equation}
Then $f$ is continuous, by Lemma~\ref{char-na}.
We write $\BV([a,b],E)$ for the vector space of all $E$-valued $\BV$-functions on $[a,b]$.
\end{numba}

\begin{numba}\label{def-f-prime}
Applying continuous linear functionals, we deduce from
\cite[29.34]{Sche}
that $\mu\in\Mdna([a,b],E)$ with (\ref{stara})
is determined by
$f\in\BV([a,b],E)$; we write
\[
f'\; :=\; \mu.
\]
\end{numba}
\noindent
By the preceding, the map
\begin{equation}\label{isometric}
\BV([a,b],E)\to E\times \Mdna([a,b],E),\quad f\mto (f(a),f')
\end{equation}
is a bijection. We give $\BV([a,b],E)$
the locally convex vector topology turning the map
into an isomorphism of topological vector spaces,
using the norm $(y,\mu)\mto \|y\|_E+ \lVert \mu\rVert$
on the right-hand side. Then the norm
\begin{align}\label{BV:NORM}
 \lVert f\rVert_{BV} \coloneq \lVert f(a)\rVert_E +\lVert \mu \rVert 
\end{align}
defines the topology on $\BV([a,b],E)$
and turns the map~(\ref{isometric})
into an isometry, whence $\BV([a,b],E)$
is a Banach space.
Since
\[
\|f(x)\|_E=\|f(a)+\mu([a,x])\|_E
\leq \|f(a)\|_E+\|\mu([a,x])\|_E\leq \|f(a)\|_E+
\lVert\mu\rVert
\]
for all $x\in [a,b]$,
we have $\|f\|_\infty\leq \|f\|_{BV}$.
Conversely, $\|f(a)\|_E\leq \|f\|_\infty$.
We deduce that $\|\cdot\|_{BV}$
is equivalent to the norm on $\BV([a,b],E)$
given by
\begin{align}\label{BV:NORM2}
 \lVert f \rVert_{BV}^\st \coloneq \|f\|_\infty+
 \lVert \mu \rVert,
\end{align}
which is the standard choice of norm in the
scalar-valued case.
\begin{rem} \label{rem:cont:inclusion}
In the following, it will be essential that the inclusion map
\[
I \colon \BV([a,b],E) \rightarrow C([a,b],E)
\]
is continuous. In fact, the map is linear and $\|f\|_\infty\leq\|f\|^{\st}_{BV}$
 for all $f\in \BV([a,b],E)$. The assertion follows since
 the norm $\lVert\cdot\rVert_{BV}^\st$ is equivalent to $\lVert\cdot\rVert_{BV}$.
\end{rem}
\begin{rem}
For a discussion of $E$-valued functions
$[a,b]\to E$ of bounded variation
which need not be continuous,
the reader is referred to \cite[22.19]{Sche},
where a norm analogous to $\lVert\cdot\rVert_{BC}$
in (\ref{BV:NORM}) is used.
The starting point there is
the variation of a function $f\colon [a,b]\to X$
to a metric space (see \cite[19.21]{Sche}).
For scalar-valued functions, cf.\ also
\cite[29.34]{Sche}.
\end{rem}
\noindent
We record three simple facts.
\begin{la}\label{into-prod}
If $a=t_0<\cdots<t_n=b$,
then the map
\begin{equation}\label{thethem}
\BV([a,b],E)\to \prod_{j=1}^n\BV([t_{j-1},t_j],E),\;\;
f\mto (f|_{[t_{j-1},t_j]})_{j=1}^n
\end{equation}
is linear and a topological embedding with closed image.
The image consists of all $(f_j)_{j=1}^n$
such that $f_j(t_j)=f_{j+1}(t_j)$
for all $j\in \{0,\ldots, n-1\}$.
\end{la}
\begin{proof}
The map in (\ref{thethem}) is linear.
For $(f_j)_{j=1}^n$ as above, define
$f\colon [a,b]\to E$ via $f(x):=f_j(x)$ if $x\in [t_{j-1},t_j]$. Let $\mu(A):=
\sum_{j=1}^n f_j'(A\cap [t_{j-1},t_j])$
for $A\in \cS$. Then
$f(x)=f(a)+\mu([a,x])$ for $x\in [a,b]$.
Also, $\mu\in \Mdna([a,b],E)$.
Thus $f\in \BV([a,b],E)$.
Hence, the image of the map in
(\ref{thethem}) is as asserted.
It is therefore closed,
using Remark~\ref{rem:cont:inclusion}
and the continuity of the point evaluation
$\ve_x\colon C(K,E)\to E$, $f\mto f(x)$
on $(C(K,E),\|\cdot\|_\infty)$
for each compact interval~$K$
and $x\in K$.
Since $\|f|_{[t_{j-1},t_j]}\|^{\st}_{BV}\leq \|f\|^{\st}_{BV}$, the injective linear map in (\ref{thethem})
is continuous. Its inverse is continuous
as $\|f\|^{\st}_{BV}\leq\sum_{j=1}^n
\|f|_{[t_{j-1},t_j]}\|^{\st}_{BV}$.
\end{proof}
\noindent
If $\mu\!\in \!\Mdna([a,b],E)$,
then $\mu([\alpha,\gamma])\!=\!\mu([\alpha,\beta])+\mu([\beta,\gamma])$ for $\alpha\!\leq\! \beta\!\leq\! \gamma$.\ Hence:
\begin{la}\label{bv-basepoint}
Let $\mu\in \Mdna([a,b],E)$,
$f\colon [a,b]\to E$ be a function
and $t_0\in [a,b]$.
Then the following conditions are equivalent.
\begin{itemize}
\item[\rm(a)]
$f$ is a $\BV$-function and $\mu=f'$.
\item[\rm(b)]
$f(x)=f(t_0)+\mu([t_0,x])$ for all $x\in [t_0,b]$
and  $f(x)=f(t_0)-\mu([x,t_0])$ for all
$x\in [a,t_0]$. \qed
\end{itemize}
\end{la}
\noindent
Affine reparametrizations lead to isomorphic Banach spaces of $\BV$-functions.

\begin{la}\label{BV:aff_repara}
Let $\alpha \colon [c,d] \rightarrow [a,b], s \mapsto a+ (b-a)\cdot (s-c)/(d-c)$. Then
$\BV(\alpha, E)\colon$
$\BV([a,b],E) \rightarrow \BV([c,d],E)$,
$f \mapsto f \circ \alpha$ is an isomorphism of Banach spaces.
\end{la}

\begin{proof}
 Let $f \in \BV ([a,b],E)$ with associated measure $\mu = \rho \, {\rm d} \nu$. Since $\alpha$ is a homeomorphism,
 the map $\cB([c,d])\to \cB([a,b])$, $A\mto \alpha(A)$
 between the Borel $\sigma$-algebras is a bijection.
 Consider the image measures $(\alpha^{-1})_*(\mu)$ and $(\alpha^{-1})_*(\nu)$ on $[c,d]$. By construction, $\nu$ is positive. Since $\mu$ is non-atomic, $(\alpha^{-1})_*(\mu)$ is non-atomic. Finally, for a Borel set $B \subseteq [c,d]$, we have $$
 (\alpha^{-1})_*(\mu) (B) = \mu (\alpha(B)) = \int_{\alpha(B)} \rho \, {\rm d}\nu = \int_B \rho \circ \alpha \, {\rm d} (\alpha^{-1})_* (\nu),$$
 by the Change of Variables Formula,
 \cite[29.12 b]{Sche}. So $(\alpha^{-1})_*(\mu) \in \Mdna ([c,d],E)$ and we have $(f \circ \alpha) (x) = f(a)+\mu([a,\alpha(x)])
 =f(\alpha(c))+(\alpha^{-1})_*(\mu) ([c,x])$.
 Notably, $f \circ \alpha \in \BV([c,d],E)$. The pullback $f\mto f\circ \alpha$ is linear and preserves the norms \eqref{BV:NORM} and \eqref{BV:NORM2}, whence it is an isometry of Banach spaces.
\end{proof}
\noindent
To construct manifold structures on spaces of functions of bounded variation, we need a chain rule. We adapt
the chain rule by Moreau and Valadier \cite[Theorem 3]{MaV},
which is stated there only for scalar-valued~$\psi$.
\begin{prop}[Chain Rule] \label{chainrule}
Let $E$ and $F$ be Banach spaces, $\Omega\sub E$ be
a convex open subset,
$\psi\colon \Omega\to F$ a continuously Fr\'{e}chet
differentiable mapping and $f \in \BV ([a,b],E)$ such that $f([a,b])\sub \Omega$.
Let $\mu:=f'\in \Mdna([a,b],E)$ and $\nu$ be a positive
measure on $[a,b]$ such that $\mu=\rho\,d\nu$
for some $[\rho]\in L^1(\nu,E)$.
Then $\psi\circ f\in \BV([a,b],F)$ and
 \begin{align}\label{chainrule:1}
(\psi\circ f) (t)=(\psi\circ f)(a)+\int_a^t \psi' (f(s))(\rho(s))\,\mathrm{d}\nu(s)
\;\;\mbox{for all $\,t\in [a,b]$.}
\end{align}
\end{prop}

\begin{proof}
Using a translation, we may assume that $f(a)=0$.
It suffices to establish \eqref{chainrule:1} to prove the proposition. To this end, let $\alpha \colon F \rightarrow \R$ be continuous linear. We recall that $f$ is continuous
as $\nu$ is non-atomic.
This simplifies the chain rule in
\cite[Theorem~3]{MaV} which we use for the second identity
in the following calculation:
 \begin{align*}
  \alpha((\psi \circ f)(t))&= (\alpha \circ \psi) (f(t))
  = \int_a^t (\alpha \circ \psi)'(f(s))(\rho(s))\,\mathrm{d}s\\
  &= \int_a^t \alpha \left(\psi' (f(s))(\rho(s))\right)\,\mathrm{d}\nu(s)\\
  &= \alpha \left(\int_a^t \psi'(f(s))(\rho(s))\,\mathrm{d}\nu(s)\right).
 \end{align*}
We used twice that $\alpha$ is continuous linear, to pass from the first line to the second and then to the third. Note that $\alpha$ can be taken out of the Bochner integral. As continuous linear
functionals separate points on~$F$,
\eqref{chainrule:1} follows.
\end{proof}
\noindent
Convexity of~$\Omega$ is irrelevant for the chain rule to be valid.
\begin{cor}\label{chainrule2}
The conclusions of Proposition~\emph{\ref{chainrule}}
remain valid if $\Omega\sub E$ is open,
but not necessarily convex.
\end{cor}
\begin{proof}
As $f([a,b])$ is compact and $\Omega$ is open,
we find convex open subsets $\Omega_1,\ldots,\Omega_n\sub \Omega$
such that $f([a,b])\sub \Omega_1\cup\cdots\cup\Omega_n$.
Let $\delta>0$ be a Lebesgue number
for the open cover $f^{-1}(\Omega_1),\ldots,f^{-1}(\Omega_n)$
of $[a,b]$ and
$a=t_0<\cdots<t_m$
be a subdivision with $t_i-t_{i-1}\leq \delta$
for all $i\in \{1,\ldots, m\}$.
For each $i\in \{1,\ldots,m\}$,
we then find $j(i)\in \{1,\ldots,n\}$
such that $f([t_{i-1},t_i])\sub \Omega_{j(i)}$.
by Proposition~\ref{chainrule},
we have $f|_{[t_{i-},t_i]}\in \BV([t_{i-1},t_i],F)$
for all $i\in \{1,\ldots,m\}$
and (\ref{chainrule:1}) holds for $f|_{[t_{i-1},t_i]}$
in place of~$f$.
Using Lemma~\ref{into-prod}, the conclusion follows.
\end{proof}
\section{The vector measures {\boldmath$f\odot_\beta\mu$}}\label{sec-bilin-meas}
We construct new vector measures from
given ones, which will be essential
in the following.
Certain function spaces will be used,
which we now describe.
\begin{numba}\label{L-infty-concept}
Given a Banach space $(E,\|\cdot\|_E)$
and a measurable space $(X,\cS)$,
we write $\cF(X,E)$
for the vector space of all $E$-valued simple functions,
i.e., measurable
functions $f\colon (X,\cS)\to (E,\cB(E))$
with finite image.
We let $\cL^\infty(X,E)$
be the vector space of all
measurable functions
$f\colon X \to E$
such that $f(X)$ is bounded
and has a countable dense subset.
The supremum norm $\|\cdot\|_\infty$ makes $\cL^\infty(X,E)$
a Banach space; see, e.g.,
\cite[Lemma~1.19]{Mea} (applied with counting measure)
or
\cite[22.27 b]{Sche}.
We let $\cL^\infty_{\rc}(X,E)$
be the set of all measurable functions
$f\colon X\to E$ whose image $f(X)$
is relatively compact in~$E$.
Then $(\cL^\infty_{\rc}(X,E),\|\cdot\|_\infty)$
is complete (see \cite[Proposition~3.21]{CAN},
again using the counting measure).
Hence $\cL^\infty_{\rc}(X,E)$
is a closed vector subspace of $\cL^\infty(X,E)$.
Moreover, $\cF(X,E)$ is dense in $\cL^\infty_{\rc}(X,E)$
(see \cite[Proposition 3.18]{CAN} or \cite[21.4\,(E)]{Sche}).
\end{numba}
\begin{rem}
If $X$ is compact topological space and $\cS=\cB(X)$,
then $C(X,E)$\linebreak
$\sub \cL^\infty_{rc}(X,E)$,
i.e., each continuous function $f\colon X\to E$
is in $\cL^\infty_{rc}(X,E)$.
\end{rem}
\noindent
We consider the following setting:
\begin{numba}
Let $(E_1,\|\cdot\|_1)$, $(E_2,\|\cdot\|_2)$
and $(F,\|\cdot\|_F)$ be
Banach spaces and
\[
\beta\colon E_1\times E_2\to F
\]
be a continuous bilinear map.
Let $(X,\cS)$ be a measurable space.
Our goal is to define a vector measure
\[
f\odot_\beta \mu\in \Mes(X,F)
\]
for $f\in \cL^\infty_{\rc}(X,E_1)$
and $\mu\in \Mes(X,E_2)$.
This measure is a generalization
of the familiar vector measures
$f\,{\rm d\mu}$ with density recalled in \ref{with-density}
for a finite positive measure $\mu$
and $f\in \cL^1(\mu,E)$ (as we shall
see in \ref{ordinary-density}).
\end{numba}
\begin{numba}
Given $f\in \cF(X,E_1)$, we write $f=\sum_{i=1}^n{\bf 1}_{A_i}v_j$
with disjoint measurable
subsets $A_1,\ldots, A_n\sub X$ and $v_1,\ldots, v_n\in E_1$,
using characteristic functions.
If $\mu\in \Mes(X,E_2)$ is a vector measure
of bounded variation,
we define a function $\mu_f\colon \cS\to F$ via
\[
\mu_f(A):=\sum_{i=1}^n\beta(v_i,\mu(A\cap A_i))
\]
for $A\in\cS$. Then $\mu_f(A)$ is well defined.
In fact, assume that also $f=\sum_{j=1}^m w_j{\bf 1}_{B_j}$
with disjoint sets $B_1,\ldots, B_m\in\cS$
and $w_1,\ldots, w_m\in E_1$.
Adding the set $A_{n+1}:=X\setminus (A_1\cup\cdots\cup A_n)$
and the vector $v_{n+1}:=0$ if necessary,
we may assume that $A\cup\cdots\cup A_n=X$.
Likewise, we may assume that $B_1\cup\cdots\cup B_m=X$.
For each $i\in\{1,\ldots, n\}$, the set
$A_i$ is the disjoint union of the intersections
$A_i\cap B_j$ for $j\in\{1,\ldots, m\}$.
If $A_i\cap B_j\not=\emptyset$,
we pick an element $x\in A_i\cap B_j$ and see that
\[
v_i=f(x)=w_j.
\]
If $A_i\cap B_j=\emptyset$, then $\mu(A\cap A_i\cap B_j)=0$.
Thus
\begin{eqnarray*}
\sum_{i=1}^n\beta(v_i,\mu(A\cap A_i))
&=&
\sum_{i=1}^n\beta\left(v_i,\mu\left(\bigcup_{j=1}^m A\cap A_i\cap B_j\right)\right)\\
&= & \sum_{i=1}^n\sum_{j=1}^m \underbrace{\beta(v_i,\mu(A\cap A_i\cap B_j))}_{=
\beta(w_j,\mu(A\cap A_i\cap B_j))}
=\sum_{j=1}^m\beta(w_j,\mu(A\cap B_j)).
\end{eqnarray*}
For each sequence $(C_k)_{k\in\N}$
of disjoint measurable subsets of~$X$, we have
\begin{eqnarray*}
\mu_f\left(\bigcup_{k\in\N}C_k\right)
&=&\sum_{j=1}^n\beta\left(v_j,\mu\left(\bigcup_{k\in\N}C_k\cap A_j\right)\right)
=\sum_{j=1}^n\sum_{k=1}^\infty \beta(v_j,\mu(C_k\cap A_j))\\
& =& \sum_{k=1}^\infty
\mu_f(C_k),
\end{eqnarray*}
whence $\mu_f$ is a vector measure.
Assuming that $A_i$ is non-empty for each $i$, we get
$\|v_i\|_1\leq \|f\|_\infty$ for each $i\in\{1,\ldots, n\}$.
For each $A\in \cS$, we have
\begin{eqnarray*}
\|\mu_f(A)\|_F &\leq &\sum_{i=1}^n\|\beta(v_i,\mu(A\cap A_i))\|_F
\leq\sum_{i=1}^n\|\beta\|_{\op}\|v_i\|_1\|\mu(A\cap A_i)\|_2\\
&\leq & \|\beta\|_{\op}\|f\|_\infty \sum_{i=1}^n\|\mu(A\cap A_i)\|_2
\leq\|\beta\|_{\op}\|f\|_\infty \Var(\mu)(A).
\end{eqnarray*}
For all disjoint measurable sets $S_1,\ldots, S_m$
with union $A$, we deduce that
\[
\sum_{k=1}^m\|\mu_f(S_k)\|_F\leq \|\beta\|_{\op}\|f\|_\infty \sum_{k=1}^m\Var(\mu)(S_k)
=\|\beta\|_{\op}\|f\|_\infty \Var(\mu)(A).
\]
Hence
\[
\Var(\mu_f)\leq \|\beta\|_{\op}\|f\|_\infty\Var(\mu).
\]
Notably,
\begin{equation}\label{bil-cts}
\|\mu_f\|\leq \|\beta\|_{\op}\|f\|_\infty \|\mu\|\,.
\end{equation}
Direct calculation shows that the map
\[
b\colon \cF(X,E_1)\times \Mes(X,E_2)\to \Mes(X,F),\quad (f,\mu)\mto \mu_f
\]
is bilinear. By (\ref{bil-cts}), the bilinear map $b$ is continuous
with $\|b\|_{\op}\leq \|\beta\|_{\op}$.
It therefore has a unique continuous bilinear extension
\begin{equation}\label{the-b-bar}
\wb{b}\colon \cL^\infty_{\rc}(X,E_1)\times \Mes(X,E_2)\to \Mes(X,F),
\end{equation}
and the latter satisfies $\|\wb{b}\|_{\op}\leq \|\beta\|_{\op}$.
We define
\[
f\odot_\beta \mu \; :=\; \wb{b}(f,\mu)
\]
for $f\in\cL^\infty_{\rc}(X,E_1)$ and $\mu\in \Mes(X,E_2)$.
By construction,
$f\odot_\beta\mu\in \Mes(X,F)$
and
\begin{equation}\label{size-dense}
\|f\odot_\beta\mu\|
\leq \|\beta\|_{\op}\|f\|_\infty\|\mu\|\,.
\end{equation}
\end{numba}
\begin{numba}\label{ordinary-density}
Let $(X,\cS,\nu)$ be a finite measure space
and $f\in \cL^\infty_{\rc}(X,E)$.
Consider the map $\sigma \colon E\times \R\to E$, $(v,t)\mto tv$
given by scalar multiplication.
Then $f$ is in $\cL^1(\nu,E)$
and
\[
(f\odot_\sigma \nu)(A)=\int_A f(x)\, {\rm d}\nu(x),
\]
the Bochner integral, for each $A\in\cS$.
In fact, equality is obvious if $f\in \cF(X,E)$.
In the general case, let $(f_k)_{k\in\N}$
be a sequence in $\cF(X,E)$ with $\|f-f_k\|_\infty\to 0$
as $k\to\infty$.
Then $\|f-f_k\|_{\cL^1}\to 0$ as well and thus
\[
(f\odot_\sigma \nu)(A)=\lim_{k\to\infty}(f_k\odot_\sigma \nu)(A)
=\lim_{k\to\infty}\int_Af_k\, d\nu=\int_A f\, d\nu\,.
\]
Thus, in more traditional notation,
\[
f\odot_\sigma \nu=f\, d\nu=f(t)\, d\nu(t)\, .
\]
\end{numba}
\begin{la}\label{superpo-vm}
Let $\alpha\colon F_1\to F_2$
be a continuous linear map between Banach spaces
$(F_j,\|\cdot\|_j)$ for $j\in\{1,2\}$
and $(X,\cS)$ be a measurable space.
Then $\alpha\circ\mu\in \Mes(X,F_2)$
for each $\mu\in \Mes(X,F_1)$
and $\|\alpha\circ \mu\|\leq \|\alpha\|_{\op}\|\mu\|$.
Hence
\[
\Mes(X,\alpha)\colon \Mes(X,F_1)\to \Mes(X,F_2),\quad\mu\mto\alpha\circ\mu
\]
is a continuous linear map of operator norm $\leq \|\alpha\|_{\op}$.
\end{la}
\begin{proof}
The map $\alpha\circ\mu\colon \cS\to F_2$
is $\sigma$-additive and hence a vector measure. For
each $A\in\cS$ and all disjoint sets $S_1,\ldots, S_n\in\cS$
with union~$A$, we have
\[
\sum_{i=1}^n\|\alpha(\mu(S_i))\|_2\leq\|\alpha\|_{\op}
\sum_{i=1}^n\|\mu(S_i)\|_1\leq \|\alpha\|_{\op}\Var(\mu)(A);
\]
thus $\Var(\alpha\circ \mu)(A)\leq \|\alpha\|_{\op}\Var(\mu)(A)$.
Notably, $\|\alpha\circ\mu\|\leq\|\alpha\|_{\op}\|\mu\|<\infty$.
\end{proof}
\begin{la}\label{compo-lin}
Let $E_1$, $E_2$, $F_1$ and $F_2$ be Banach spaces,
$\beta\colon E_1\times E_2\to F_1$ be a continuous bilinear map
and $\lambda\colon F_1\to F_2$
%a map which is
be a continuous linear map. Then
\[
\lambda \circ (f\odot_\beta \mu)=f\odot_{\lambda \circ \beta}\mu
\]
for all $f\in\cL^\infty_{\rc}(X,E_1)$
and $\mu\in \Mes(X,E_2)$.
\end{la}
\begin{proof}
Using Lemma~\ref{superpo-vm},
we obtain continuous functions $h_1,h_2\colon \cL^\infty_{\rc}(X,E_1)\times
\Mes(X,E_2)\to \Mes(X,F_2)$ via
\[
h_1(f,\mu):=\lambda\circ (f\odot_\beta \mu)\qquad\mbox{and}\qquad
h_2(f,\mu):=f\odot_{\lambda\circ\beta}\mu.
\]
For each $f=\sum_{i=1}^nv_i{\bf 1}_{A_i}\in\cF(X,F_1)$
and each $\mu\in \Mes(X,F_1)$, we have
\[
h_1(f,\mu)(A)=\lambda\left(\sum_{i=1}^n\beta(v_i,\mu(A\cap A_i))\right)
\!=\! \sum_{i=1}^n(\lambda\circ\beta)(v_i,\mu(A\cap A_i))=h_2(f,\mu)(A)
\]
for all $A\in \cS$ and hence $h_1(f,\mu)=h_2(f,\mu)$.
As $\cF(X,F_1)$ is dense in $\cL^\infty_{\rc}(X,E_1)$
and both $h_1$ and $h_2$ are continuous,
$h_1=h_2$ follows.
\end{proof}
\begin{la}\label{odot-density}
Let $(E_1,\|\cdot\|_1)$, $(E_2,\|\cdot\|_2)$, and $(F,\|\cdot\|_F)$ be Banach spaces
and $\beta\colon E_1\times E_2\to F$ be a continuous
bilinear map. Let $\nu\in \Mes_+(X)$.
Then
\[
f\odot_\beta (\rho\, d\nu)=
\beta(f(t),\rho(t))\, d\nu(t)
\]
for all $f\in\cL^\infty_{\rc}(X,E_1)$
and $\rho\in \cL^1(\nu,E_2)$.
\end{la}
\begin{proof}
The map
\[
\theta\colon L^1(\nu,E_2)\to \Mes(X,E_2),\quad [\rho]\mto \rho\, {\rm d}\nu
\]
is linear and continuous, being actually an isometric embedding
(see Remark~\ref{newrem}). Likewise,
the corresponding map $\Theta\colon L^1(\nu,F)\to \Mes(X,F)$
is continuous.
For all $f\in \cL^\infty(X,E_1)$
and $\rho\in \cL^1(\nu,E_2)$,
the map $\beta\circ (f,\rho)\colon X\to F$
is measurable (cf.\ \cite[Lemma~2.7]{CAN})
with separable image and
\[
\int_X \underbrace{\|\beta(f(t),\rho(t))\|_F}_{\leq\|\beta\|_{\op}\|f(t)\|_1\|\rho(t)\|_2}\,{\rm d}\nu(t)\leq\|\beta\|_{\op}\|f\|_\infty \int_X\|\rho(t)\|_2\, {\rm d}\nu(t)=
\|\beta\|_{\op}\|f\|_\infty\|\rho\|_{\cL^1},
\]
whence $\beta\circ (f,\rho)\in \cL^1(\nu,F)$
with
\begin{equation}\label{again-ct}
\|\beta\circ (f,\rho)\|_{\cL^1}\leq \|\beta\|_{\op}\|f\|_\infty\|\rho\|_{\cL^1}.
\end{equation}
Hence
\[
b\colon \cL^\infty(X,E_1)\times L^1(\nu,E_2)\to L^1(\nu,F),\quad
(f,[\rho])\mto [\beta\circ (f,\rho)]
\]
makes sense. A direct calculation shows that $b$ is bilinear;
by (\ref{again-ct}), the bilinear map
$b$ is continuous with $\|b\|_{\op}\leq \|\beta\|_{\op}$.
Hence both of the maps
$h_1,h_2\colon \cL^\infty(X,E_1)\times L^1(\nu,E_2)\to L^1(\nu,F)$,
\[
h_1(f,[\rho]):=f\odot_\beta(\rho\, d\nu)=f\odot_\beta \theta(\rho)\;\,\mbox{and}\;\,
h_2(f,[\rho])=(\beta\circ (f,\rho))\,d\nu=
\Theta(b(f,\rho))
\]
are continuous.
If $f=\sum_{i=1}^nv_i{\bf 1}_{A_i}\in\cF(X,E_1)$
and $\rho=\sum_{j=1}^mw_j{\bf 1}_{B_j}\in\cF(X,E_2)$
with disjoint sets $A_1,\ldots, A_n\in\cS$
and $B_1,\ldots,B_m\in\cS$ with union~$X$, then
\[
\beta\circ (f,\rho)=\sum_{i=1}^n\sum_{j=1}^m \beta(v_i,w_j){\bf 1}_{A_i\cap B_j}
\]
and we get for all $A\in \cS$
\begin{eqnarray*}
h_1(f,\rho)(A) &=&\sum_{i=1}^n
\beta(v_i,\rho\,{\rm d}\nu (A\cap A_i))\\
&=&\sum_{i=1}^n\sum_{j=1}^m\beta(v_i,w_j)\nu(A\cap A_i\cap B_j)
=((\beta\circ (f,\rho))\, {\rm d}\nu)(A)\\
&=& h_2(f,\rho)(A).
\end{eqnarray*}
Thus $h_1(f,\rho)=h_2(f,\rho)$.
As $\cF(X,E_1)$ is dense in $\cL^\infty_{\rc}(X,E_1)$
and $\cF(X,E_2)$ is dense in $L^1(\nu,E_2)$
and both $h_1$ and $h_2$ are continuous,
$h_1=h_2$ follows.
\end{proof}
\noindent
We immediately deduce:
\begin{la}\label{odot-RN}
Let $E_1$, $E_2$, and $F$ be Banach spaces,
$\beta\colon E_1\times E_2\to F$ be a continuous bilinear
map and $(X,\cS)$ be a measurable space.
Then $f\odot_\beta \mu\in \Md(X,F)$
for all $f\in\cL^\infty_{\rc}(X,E_1)$
and $\mu\in \Md(X,E_2)$.
If $\mu\in \Mdna(X,E_2)$,
then also $f\odot_\beta\mu\in \Mdna(X,E_2)$.
\end{la}
\begin{proof}
Assume that $\mu=\rho\, {\rm d}\nu$ for a finite measure
$\nu$ on $(X,\cS)$ and $\rho\in\cL^1(\nu,E_2)$.
By Lemma~\ref{odot-density}, we have
$f\odot_\beta\mu=r\,{\rm d}\nu$
with $r:=\beta\circ (f,\rho)$, whence $f\odot_\beta\mu\in \Md(X,E_2)$.\\[2.3mm]
If $\mu$ is non-atomic, then $\Var(\mu)$
is non-atomic (by Lemma~\ref{la:nonatomic})
and we may assume that $\nu=\Var(\mu)$
has been chosen (see \ref{defn-meas-d}).
Then $\Var(f\odot_\beta\mu)
=\Var(\beta(f(x),\rho(x))\,{\rm d}\nu(x))
=\|\beta(f(x),\rho(x))\|_F\, {\rm d}\nu(x)$
(see \ref{defn-meas-d}),
which is non-atomic by Remark~\ref{easy-nonatomic}\,(c).
Hence also $f\odot_\beta\mu$ is non-atomic, by
Lemma~\ref{la:nonatomic}.
\end{proof}
\section{The vector measures {\boldmath$f_*(\mu,\gamma)$}}\label{sec-the-rhs}
With a view towards right-hand sides
of differential equations,
we introduce notation for
a special case of the construction in Section~\ref{sec-bilin-meas}.
\begin{numba}\label{situ-rhs}
Let $(E_1,\|\cdot\|_1)$,
$(E_2,\|\cdot\|_2)$
and $(F,\|\cdot\|_F)$
be Banach spaces, $U\sub E_2$ be an open subset
and $f\colon E_1\times U\to F$
be a mapping such that
\[
\wt{f}(y):=f(\cdot,y)\colon E_1\to F
\]
is continuous linear for each $y\in U$ and
the map
\[
\wt{f}\colon U\to (\cL(E_1,F),\|\cdot\|_{\op})
\]
is continuous.
Let $\ve\colon \cL(E_1,F)\times E_1\to F$,
$(\alpha,x)\mto\alpha(x)$ be the evaluation map,
which is bilinear and continuous.
If $a<b$ are real numbers,
we write
\[
\cL^\infty_{\rc}([a,b],U)
\]
for the set of all $\gamma\in \cL^\infty_{\rc}([a,b],E_2)$
such that $\wb{\gamma([a,b])}\sub U$.
Then $\cL^\infty_{\rc}([a,b],U)$
is an open subset of $\cL^\infty_{\rc}([a,b],E_2)$
(cf.\ \cite[Lemma~3.14\,(a)]{Mea}).
For each $\gamma\in \cL^\infty_{\rc}([a,b], U)$, we have
\[
\wt{f}\circ \gamma\in \cL^\infty_{\rc}([a,b],\cL(E_1,F)).
\]
In fact, the composition is measurable
and the continuous map $\wt{f}$
takes the compact set $\wb{\gamma([a,b])}$
to a compact set.
\end{numba}
\noindent
The following abbreviation is useful.
\begin{defn}
In the situation of~\ref{situ-rhs}, we define
\[
f_*(\mu,\gamma)\; :=\; (\wt{f}\circ \gamma)\odot_\ve \mu
\in \Mes([a,b],F)
\]
for each $\mu\in\Mes([a,b],E_1)$ and $\gamma\in \cL^\infty_{\rc}([a,b],U)
\sub \cL^\infty([a,b], E_2)$.
\end{defn}
If $\mu\in \Mdna([a,b],E_1)$, then
$f_*(\mu,\gamma)\in \Mdna([a,b],F)$,
by Lemma~\ref{odot-RN}.\\[2.3mm]
As before, let $(E_1,\|\cdot\|_1)$, $(E_2,\|\cdot\|_2)$,
and $(F,\|\cdot\|)$ be Banach spaces
and
$U\sub E_2$ be an open subset.
\begin{la}\label{diff-pfwd}
Let $k\in \N_0\cup\{\infty\}$
and $f\colon E_1\times U\to F$
be a mapping such that $\wt{f}(y):=f(\cdot,y)\in \cL(E_1,F)$
for all $y\in U$
and the map
\[
\wt{f}\colon U\to (\cL(E_1,F),\|\cdot\|_{\op})
\]
is $C^k$. Let $a< b$ be real numbers.
Then also the mapping
\[
f_*\colon
\Mes([a,b],E_1)\times\cL^\infty_{\rc}([a,b],U)\to\Mes([a,b],F),\;\,
(\mu,\gamma)\mto f_*(\mu,\gamma)
\]
is $C^k$ and also its restriction to a map
\[
f_*\colon \Mdna([a,b],E_1)\times\cL^\infty_{\rc}([a,b],U)\to\Mdna([a,b],F).
\]
\end{la}
\begin{proof}
It suffices to consider
the first mapping as it takes $\Mdna([a,b],E_1)$
into $\Mdna([a,b],F)$
and the latter is a closed vector subspace
of $\Mes([a,b],F)$ (see \cite[Lemma~1.25]{Schm}).
Using the continuous bilinear mapping
\[
b\colon \cL^\infty_{\rc}([a,b],\cL(E_1,F))\times \Mes([a,b],E_1)
\to\Mes([a,b],F),\;\,
(g,\mu)\mto g\odot_\ve \mu
\]
we have
\[
f_*(\mu,\gamma)=b(\wt{f}\circ\gamma,\mu).
\]
It therefore suffices to note
that the map
\[
\cL^\infty_{\rc}([a,b],\wt{f})\colon \cL^\infty_{\rc}([a,b], U)\to
\cL^\infty_{\rc}([a,b],\cL(E_1,F)),\;\,
\gamma \mto \wt{f}\circ \gamma
\]
is $C^k$, by \cite[Proposition~4.1]{Mea} (applied with
a singleton~$P$
as in the proof of \cite[Corollary~4.4]{Mea}).
\end{proof}
\begin{numba}
Let $(E,\lVert\cdot\rVert_E)$ and $(F,\lVert\cdot\rVert_F)$
be normed spaces and $f\colon U\to F$
be a map on a subset $U\sub E$.
The map $f$ is called \emph{Lipschitz}
if there exists a real number $L\geq 0$
(called a Lipschitz constant) such that
\[
\|f(y)-f(x)\|_F\;\leq \; L\, \|y-x\|_E\quad\mbox{for all $\,x,y\in U$.}
\]
\end{numba}
\begin{la}\label{star-Lipschitz}
Let $E_1$, $E_2$, $F$,
$f\colon E_1\times U\to F$, and $\wt{f}$
be as in {\rm \ref{situ-rhs}}.
If $\wt{f}\colon U\to (\cL(E_1,F),\|\cdot\|_{\op})$
is Lipschitz with constant~$L$,
then the mapping
\[
\cL^\infty_{\rc}([a,b],U)\to \Mdna([a,b],F),\quad
\gamma\mto f_*(\mu,\gamma)
\]
is Lipschitz with Lipschitz constant $\|\mu\|\,L$
for all real numbers $a<b$ and
each $\mu\in \Mdna([a,b],E_1)$.
\end{la}
\begin{proof}
Let $\gamma,\eta\in \cL^\infty_{\rc}([a,b],U)$.
For each $x\in [a,b]$, we then have
\[
\|\wt{f}(\gamma(x))-\wt{f}(\eta(x))\|_F
\leq L\, \|\gamma(x)-\eta(x)\|_2\leq L\, \|\gamma-\eta\|_\infty.
\]
Thus $\|(\wt{f}\circ \gamma)-(\wt{f}\circ\eta)\|_\infty\leq
L\, \|\gamma-\eta\|_\infty$. Since $\|\ve\|_{\op}\leq 1$
holds for the evaluation map $\ve\colon \cL(E_1,F)\times E_1\to F$,
we deduce that
\begin{eqnarray*}
\|f_*(\mu,\gamma)-f_*(\mu,\eta)\|
&=&\|(\wt{f}\circ \gamma)\odot_\ve\mu-(\wt{f}\circ\eta)\odot_\ve\mu\|\\
&=&
\|((\wt{f}\circ\gamma)-(\wt{f}\circ \eta))\odot_\ve\mu\|\\
&\leq & \|\ve\|_{\op}\|\mu\|\,
\|(\wt{f}\circ\gamma)-(\wt{f}\circ \eta)\|_\infty
\leq  \|\mu\|\,L\, \|\gamma-\eta\|_\infty,
\end{eqnarray*}
using (\ref{size-dense}) for the first estimate.
This completes the proof.
\end{proof}
\begin{rem}\label{from-axioms}
\begin{itemize}
\item[(a)]
If $f\colon E_1\times U\to F$
is a $C^{k+2}$-map with $k\in \N_0\cup\{\infty\}$
and $\wt{f}(y):=f(\cdot,y)$
is linear for each $y\in U$,
then $\wt{f}\colon U\to (\cL(E_1,F),\|\cdot\|_{\op})$
is a $C^{k+1}$-map
(see \cite[Proposition~1\,(b)]{Asp})
and hence $\text{FC}^k$.
In particular, $\wt{f}$
is $C^1$ and hence \emph{locally Lipschitz}
in the sense that $\wt{f}$ is Lipschitz
on some neighborhood of each point in its domain
(see \cite[Lemma~1.3.19]{GaN}).\vspace{2mm}
\end{itemize}
\noindent
Now assume that $E_1=E_2$ and write $E:=E_1$.
\begin{itemize}
\item[(b)] If $g\colon U\to F$
is an $\text{FC}^{k+1}$-map, then
\[
f\colon E\times U\to F,\;\,
(y,x)\mto dg(x,y)
\]
satisfies the hypotheses of Lemma~\ref{diff-pfwd}
as $\wt{f}=g'\colon U\to(\cL(E,F),\|\cdot\|_{\op})$
is an $\text{FC}^k$-map and hence $C^k$,
by definition of an $\text{FC}^{k+1}$-map.
\item[(c)] If $g\colon U\to F$
is an $C^{k+2}$-map, then
$g$ is $\text{FC}^{k+1}$,
whence (b) applies.
\end{itemize}
\end{rem}
\begin{rem}
We shall mainly use $f_*(\mu,\gamma)$
for continuous maps $\gamma\colon [a,b]\to E_2$
(rather than general $\gamma\in \cL^\infty_{\rc}([a,b],E_2)$).
For each Banach space $(E,\lVert\cdot\rVert_E)$,
the supremum norm $\|\cdot\|_\infty$ turns
the vector space $C([a,b],E)$ of continuous $E$-valued
functions into a Banach space which is a closed
vector subspace of $\cL^\infty_{\rc}([a,b], E)$.
\end{rem}
\begin{rem}\label{relax}
If $\alpha,\beta\in [a,b]$ with $\alpha<\beta$
in the situation of~\ref{situ-rhs}
and $\mu|_{[\alpha,\beta]}$
denotes the restriction of~$\mu$
to a map $\cB([\alpha,\beta])\to E_1$,
we relax notation for $\gamma\in \cL^\infty_{\rc}([\alpha,\beta],U)$ and write $f_*(\mu,\gamma)$
in place of $f_*(\mu|_{[\alpha,\beta]},\gamma)$.
\end{rem}
\begin{rem}\label{rem-explicit-1}
If $\mu=\rho\, {\rm d}\nu$
with $\nu\in \Mes([a,b])_+$
and $\rho\in \cL^1(\nu,E_1)$
in \ref{situ-rhs},
we get
\begin{equation}\label{f-star-explic}
f_*(\rho\, {\rm d}\nu,\gamma)=
(f\circ (\rho,\gamma))\, {\rm d}\nu .
\end{equation}
In fact,
$f_*(\rho\, {\rm d}\nu,\gamma)=
(\wt{f}\circ\gamma)\odot_\ve \rho\,{\rm d}\nu=
(\ve\circ (\wt{f}\circ \gamma,\rho))\, {\rm d}\nu$
by Lemma~\ref{odot-density}.
\end{rem}
\begin{rem}\label{swap-rem}
If $f\colon E_1\times U\to F$ is as in~\ref{situ-rhs},
consider $g\colon U\times E_1\to F$,
$(x,y)\mto f(y,x)$. It is sometimes convenient to swap
the order of the arguments and define
$g_*(\gamma,\mu):=f_*(\mu,\gamma)$
for $\mu\in \Mes([a,b],E_1)$ and $\gamma\in \cL^\infty_{\rc}([a,b],U)$.
\end{rem}
\begin{rem}\label{rem-explicit-2}
In the situation of Corollary~\ref{chainrule2},
we have
\[
(\psi\circ f)'=(d\psi)_*(f, f'),
\]
with notation as in Remark~\ref{swap-rem};
in fact, writing $f'=\rho\, {\rm d}\nu$, the corollary
yields
$(\psi\circ f)'=((d\psi)\circ(f,\rho))\, {\rm d}\nu$;
this equals $(d\psi)_*(f,f')$, by~(\ref{f-star-explic}).
\end{rem}
\section{Differentiability of certain mappings of BV-functions}\label{sec-superpo}
We study superposition operators
and composition operators
on open subsets of spaces of vector-valued
$\BV$-functions.
We let $a<b$ be real numbers and abbreviate
$\bI:=[a,b]$.\\[2.3mm]
We begin with superposition of continuous linear maps.
\begin{la}\label{linear-superpo}
Let $\alpha\colon E\to F$ be a continuous linear map
between Banach spaces.
Then $\alpha\circ f\in \BV(\bI,F)$
for each $f\in \BV(\bI,E)$
and the linear map
\[
\alpha_*:=\BV(\bI,\alpha)\colon \BV(\bI,E)\to\BV(\bI,F),\;\;
f\mto \alpha\circ f
\]
is continuous, with operator norm $\|\alpha\|_{\op}$
with respect to the standard norms,
and also with respect to the norms $\|\cdot\|_{BV}$.
\end{la}
\begin{proof}
We have $\|\alpha(f(x))\|_F\leq
\|\alpha\|_{\op}\|f(x)\|_E$
for all $x\in \bI$, whence $\|\alpha_*(f)\|_\infty\leq
\|\alpha\|_{\op}\|f\|_\infty$
and $\|\alpha_*(f)(a)\|_F\leq \|\alpha\|_{\op}\|f(a)\|_E$.
Moreover, $(\alpha\circ f)'=\alpha\circ f'\in  \Mdna(\bI,F)$ by the Chain Rule
with $\|\alpha\circ f'\|\leq \|\alpha\|_{\op}\|f'\|$,
by Lemma~\ref{superpo-vm}.
Thus $\|\alpha_*(f)\|_{BV}=\|\alpha_*(f)(a)\|_F
+\|(\alpha_*(f))'\|\leq\|\alpha\|_{\op}(\|f(a)\|_E+\|f'\|)=
\|\alpha\|_{\op}\|f\|_{BV}$
and $\|\alpha_*(f)\|_{BV}^{\st}=\|\alpha_*(f)\|_\infty
+\|(\alpha_*(f))'\|\leq \|\alpha\|_{\op}(\|f\|_\infty
+\|f'\|)=\|\alpha\|_{\op}\|f\|_{BV}^{\st}$.
\end{proof}
\begin{la}\label{prod:iso}
Let $E_1$ and $E_2$ be Banach spaces, $E:=E_1\times E_2$
and $\pr_j\colon E\to E_j$ be the projection onto
the $j$th component.
Then the map
\[
\Phi\colon
\BV(\bI,E)\to\BV(\bI,E_1)\times \BV(\bI,E_2),\;\,
f\mto (\pr_1\circ f,\pr_2\circ f)
\]
taking a function to its pair of components
is an isomorphism of topological
vector spaces.
\end{la}
\begin{proof}
Define continuous linear mappings $\lambda_1\colon E_1\to E$
and $\lambda_2\colon E_2\to E$
via $x_1\mto (x_1,0)$ and $x_2\mto (0,x_2)$,
respectively.
Then
\[
(\pr_j)_*:=\BV(\bI,\pr_j)\colon \BV(\bI,E)\to\BV(\bI,E_j),\;\,
f\mto \pr_j\circ f
\]
is a continuous linear mapping for $j\in \{1,2\}$
and also the linear map
$(\lambda_j)_*\:=\BV(\bI,\lambda_j)\colon \BV(\bI,E_j)\to\BV(\bI,E)$
is continuous.
Thus $\Phi=((\pr_1)_*,(\pr_2)_*)$
is continuous linear and also the map
\[
\Psi:=(\lambda_1)_*\circ \pi_1+(\lambda_2)_*\circ \pi_2,
\]
where $\pi_j\colon \BV(\bI,E_1)\times\BV(\bI,E_2)\to \BV(\bI,E_j)$
is the projection onto the $j$th component.
One readily checks that $\Psi\circ \Phi$
is the identity on $\BV(\bI,E)$
and $\Phi\circ \Psi$
is the identity on $\BV(\bI,E_1)\times\BV(\bI,E_2)$.
The assertion follows.
\end{proof}
\noindent
We shall frequently identify $\BV(\bI,E)$
with the product.
Keeping one argument of $\Psi$ fixed,
we obtain:
\begin{la}\label{lem:smooth_insertion}
If $f\in \BV(\bI,E_1)$
in the situation of Lemma~\emph{\ref{prod:iso}},
then
\[
\BV (\bI,E_2) \rightarrow \BV(\bI,E_1)\times \BV(\bI,E_2)
\cong \BV(\bI,E),\quad g \mapsto (f,g)
=(f,0)+(0,g)
\]
is a continuous affine map
$($and hence smooth$)$. $\,\square$
\end{la}
\begin{numba}
If $E$ is a Banach space and $U\sub E$ an open subset,
we shall write
\[
\BV(\bI,U):=\{f\in \BV(\bI,E)\colon f(\bI)\sub U\}
\]
in the following.
The latter is an open subset of $\BV(\bI,E)$:
In view of Remark~\ref{rem:cont:inclusion},
this follows from the openness
of $C(\bI,U)$ in $(C(\bI,E),\|\cdot\|_\infty)$.
\end{numba}
\begin{la}\label{improve-chn}
In the situation of Corollary~\emph{\ref{chainrule2}},
the map
\[
\BV(\bI,\psi)\colon \BV(\bI,\Omega)\to \BV(\bI,F),\quad
f\mto \psi\circ f
\]
is continuous, where $\bI:=[a,b]$.
Moreover, we have
\begin{equation}\label{better-chain}
(\psi\circ f)'=
(\psi'\circ f)\odot_\ve f'\quad\mbox{for each $\,f\in \BV(\bI,\Omega)$,}
\end{equation}
using the
continuous bilinear
evaluation
$\ve \colon \cL(E,F)\times E\to F$,
$(\alpha,x)\mto \alpha(x)$.
\end{la}
\begin{proof}
For $f$, $\mu$, $\nu$, and $\rho$ as in the corollary,
setting $y_0:=\psi(f(a))$ we have
\[
(\psi\circ f)(t)=y_0+\int_a^t (\psi'(f(s))(\rho(s))\,{\rm d}\nu(s)
=y_0+\int_a^t \ve\circ (\psi'\circ f,\rho)\, {\rm d}\nu
\]
and hence $(\psi\circ f)'=(\ve\circ(\psi'\circ f,\rho))\, {\rm d}\nu
=(\psi'\circ f)\odot_\ve(\rho\,{\rm d} \nu)
=(\psi'\circ f)\odot_\ve\mu$,
using Lemma~\ref{odot-density}.
This establishes~(\ref{better-chain}).
The topology on $\BV(\bI,F)$
is initial with respect to the inclusion map
\[
j_F\colon \BV(\bI,F)\to C(\bI,F)
\]
and the map $D_F\colon \BV(\bI,F)\to \Mdna (\bI,F)$,
$g\mto g'$.
Let $j_E$ and $D_E$ be analogous
with $E$ in place of~$F$. The map
\[
j_F\circ \BV(\bI,\psi)=C(\bI,\psi)\circ j_E|_{\BV(\bI,\Omega)}
\]
is continuous as $C(\bI,\psi)$
is continuous
(see, e.g., \cite[Lemma B.8]{Schm}).
For $\wb{b}$ as in (\ref{the-b-bar}) (with $X:=\bI$
and $\beta:=\ve$) we have
\[
D_F\circ \BV(\bI,\psi)=\wb{b}\circ (C(\bI,\psi'),D_E),
\]
which is a continuous function as well. Hence $\BV(\bI,\psi)$
is continuous.
\end{proof}
\begin{prop}\label{prop:star_smooth}
Let $k\in \N_0\cup\{\infty\}$
and $\varphi \colon [a,b] \times U \rightarrow F$ be a $C^{k+2}$-map,
where $E$ and $F$ are Banach spaces
and $U\sub E$ an open subset.
Then the following map is $C^k$:
\[
\varphi_\star \colon \BV([a,b],U) \rightarrow \BV([a,b],F),\quad f \mapsto \varphi \circ (\id_{[a,b]},f).
\]
Notably, $\varphi_\star$ is smooth whenever $\varphi$
is smooth.
\end{prop}
\begin{proof}
It suffices to prove the assertion for finite~$k$.
The proof is by induction.
Assume $k=0$ first.
Using the version
\cite[Theorem~3.1]{Han} of Seeley's extension theorem,
we get $C^2$-functions
$\varphi_+\colon \;]a,\infty[\,\times U\to F$
and $\varphi_-\colon \;]{-\infty},b[\,\times U\to F$
which extend $\varphi|_{]a,b]\times U}$
and $\varphi|_{[a,b[\,\times U}$,
respectively. Let $h\colon \R\to \R$
be a non-negative smooth function supported in
$[a+(b-a)/3,a+2(b-a)/3]$
with $\int_{-\infty}^\infty
h(s)\,{\rm d}s
=1$;\vspace{-.5mm} set $\theta(t):=\int_{-\infty}^th(s)\,
{\rm d}s$ for $t\in \R$.
Then
$\psi\colon \R\times U\to F$,
\[
(t,x)\mto\left\{
\begin{array}{cl}
\varphi_-(t,x) &\mbox{if $t<a+(b-a)/3$;}\\
\theta(t)\varphi_+(t,x)+(1-\theta(t))\varphi_-(t,x) &\mbox{if $t \in \,]a,b[$;}\\
\varphi_+(t,x) &\mbox{if $t>a+2(b-a)/3$}
\end{array}\right.
\]
is a $C^2$-map extending~$\varphi$,
with open domain $\Omega:=\R\times U$.
By Proposition~\ref{chainrule},
\[
\varphi_\star(f)=\psi\circ (\id_{[a,b]},f)=\BV(\bI,\psi)(\id_{[a,b]},f)
\]
is a continuous function of $f\in\BV([a,b],U)$.
We here used the identification from Lemma~\ref{prod:iso}.
Note that
$\id_{[a,b]}\in \BV([a,b],\R)$
with $(\id_{[a,b]})'=\lambda|_{[a,b]}$.\\[2.3mm]
Induction step. Let $k\geq 1$
and assume the proposition holds for
$k-1$ in place of~$k$. Let $\varphi$ be a $C^{k+2}$-map
as described in the proposition.
Then $\varphi_\star$
is $C^{k-1}$ by the inductive hypothesis
and hence continuous.
The function
\[
d_2\varphi\colon \bI\times U\times E\to F,\;\,
(t,x,y)\mto d\varphi((t,x),(0,y))
\]
with $\bI:=[a,b]$
is $C^{k+1}$, whence
\[
(d_2\varphi)_\star\colon \BV(\bI,U\times E)\to \BV(\bI,F),
\;\,
(f,g)\mto d_2\varphi \circ (f,g)
\]
is $C^{k-1}$, by the inductive hypothesis.
We identify the domain with the cartesian product
$\BV(\bI,U)\times \BV(\bI,E)$
here, as in Lemma~\ref{prod:iso}.
Given $f\in \BV(\bI,U)$ and $g\in \BV(\bI,E)$,
there exists $\ve>0$ such that $f(x)+tg(x)\in U$
for all $t\in [{-\ve},\ve]$
and $x\in \bI$, and thus $f+tg\in \BV(\bI,U)$.
Then
\[
h\colon [{-\ve},\ve]\times [0,1]\to\BV(\bI,F),\;\,
(t,s)\mto (d\varphi_2)_\ast(f+stg,g)
\]
is a continuous mapping.
Hence
\[
H\colon [{-\ve},\ve]\to\BV(\bI,F),\;\,
t\mto\int_0^1 h(t,s)\,{\rm d}s
\]
is continuous by the theorem
on continuous parameter-dependence
of integrals (see, e.g., \cite[Lemma 1.1.11]{GaN});
the weak integral exists since
$\BV(\bI,F)$ is a Banach space.
For non-zero $t\in [{-\ve},\ve]$,
consider the difference quotient
\[
\Delta(t):=\frac{\varphi_\star(f+tg)-\varphi_\star(f)}{t}.
\]
For each $x\in \bI$, the Mean Value Theorem
shows that
\begin{eqnarray*}
\Delta(t)(x) &=&\frac{\varphi(x,f(x)+tg(x))-\varphi(x,f(x))}{t}\\
&=& \int_0^1 d_2f(x,f(x)+stg(x),g(x))\, {\rm d}s\\
&= &\int_0^1 h(s,t)(x)\, {\rm d}s
=\left(\int_0^1 h(s,t)\, {\rm d}s\right)(x)\\[.5mm]
&=& H(t)(x);
\end{eqnarray*}
for the penultimate equality, we used that
the continuous linear
evaluation map $\BV(\bI,F)\to F$, $\gamma\mto\gamma(x)$
and integration can be interchanged
(see, e.g., \cite[Exercise 1.1.3\,(a)]{GaN}).
Thus
\[
\Delta(t)=H(t)\to H(0)=\int_0^1h(t,0)\,{\rm d}s=\int_0^1
(d_2\varphi)_\star(f,g)\, {\rm d}s
=(d_2\varphi)_\star(f,g)
\]
as $t\to 0$, whence the directional derivative
$d(\varphi_\star)(f,g)$ exists and is the element
\[
d(\varphi_\star)(f,g)=(d_2\varphi)_\star(f,g)
\]
of $\BV(\bI,F)$.
By the inductive hypothesis,
the right-hand side
is a $C^{k-1}$-function of $(f,g)$
and hence continuous.
Thus $\varphi_\star$ is $C^1$
and $d(\varphi_\star)$ a $C^{k-1}$-map,
entailing that $\varphi_\star$ is $C^k$.
\end{proof}
\noindent
Cf.\ \cite{MRO}
for
extension
results in the context of Banach manifolds with corners.
\begin{la}\label{lem:PF_smooth}
 Let $h \colon U \rightarrow F$ be a smooth mapping defined on an open subset $U \subseteq E$ of a Banach space.
 Then the pushforward
 $$\BV([a,b],h)\colon \BV([a,b],U) \rightarrow \BV([a,b],F),\quad f \mapsto h\circ f$$
 is smooth.
\end{la}

\begin{proof}
Apply Proposition \ref{prop:star_smooth} to the $C^\infty$-map $\varphi\colon \![a,b]\times  U \!\to\! F$,
$\!(t,u)\!\mto \! h(u)$.\!\!
\end{proof}
\section{The modeling spaces of bundle sections}\label{sec-model}
Let $M$ be a smooth manifold modeled on a Banach space $E$. We define functions of bounded variation with values in $M$ as follows:

\begin{defn}
A continuous function $f \colon [a,b] \rightarrow M$ is an \emph{$M$-valued function of bounded variation} if there exists a subdivision $a=t_0 < t_1 < \cdots < t_n=b$ and charts $(\varphi_i,U_i)$ of $M$ such that for each $1 \leq i\leq n$ we have $f([t_{i-1},t_i])\sub U_i$
and $\varphi_i \circ f|_{[t_i,t_{i-1}]}
\in \BV([t_{i-1},t_i],E)$. We write $\BV([a,b],M)$ for the
set of all $M$-valued $\BV$-maps on~$[a,b]$.
\end{defn}
\noindent
Note that the definition of manifold-valued $\BV$-maps is independent of any choices made for the subdivision of the interval and the charts. Indeed we can always refine a subdivision (cf.\ Lemma~\ref{into-prod}).
Due to the chain rule for $\BV$-maps, Corollary \ref{chainrule2}, a map $f$ which is of bounded variation if composed with one chart is also of bounded variation in any other chart which contains the image $f([t_{i-1},t_i])$.\smallskip

\noindent
To construct a manifold structure on the set $\BV([a,b],M)$, we require certain spaces of $\BV$-maps with values in the tangent bundle $TM$ of $M$. It will be convenient to discuss, more generally, spaces of $\BV$-sections for arbitrary vector bundles over $M$.

\begin{defn}\label{defn:gammaf}
 Let $\pi \colon B \rightarrow M$ be a smooth Banach vector bundle with typical fibre $F$. For $f \in \BV([a,b],M)$, we let 
 $$\Gamma_f(B) \coloneq \{\tau \in \BV([a,b],B) \mid \pi \circ \tau = f\}$$
 and call its elements \emph{$\BV$-maps over~$f$ with values in $B$}. If $B=TM$, we abbreviate $\Gamma_f \coloneq \Gamma_f (TM)$.
\end{defn}

\begin{numba}\label{resolve_Gammaf}
 For $f \in \BV([a,b],M)$, pick a subdivision $a=t_0 < t_1 < \cdots < t_n=b$ and local vector bundle trivializations $(\psi_i, U_i)$ of $B$ such that for each $1 \leq i\leq n$, we have $f([t_i,t_{i-1}]) \subseteq \pi(U_i)$.
 Thus $\psi_i\colon \pi^{-1}(U_i)\to U_i\times F$.
 We obtain an injective map
 \begin{align}
  \Psi\colon \Gamma_f(B) \rightarrow \prod_{1\leq i \leq n} \BV([t_{i-1},t_i],F),\ \tau \mapsto (\pr_2 \circ \,\psi_i \circ \tau|_{[t_{i-1},t_i]})_{i=1}^n. \label{thetamap}
 \end{align}
 Endow $\Gamma_f(B)$ with the initial topology with respect to the map $\Psi$.
 We shall refer to this topology as the \emph{$\BV$-topology.}
\end{numba}

\noindent
Employing Lemma \ref{into-prod}, the image of $\Psi$ is the closed vector subspace 
$$\left\{(g_i)_{1\leq i\leq n} \in {\textstyle\prod_{i=1}^n}
\BV([t_{i-1},t_i],F) \colon  g_{i-1}(t_i)=g_i(t_i)\;\,
\mbox{for all}\;\,  2\leq i\leq n\right\}.$$
A trivial calculation shows that with respect to addition in the fibres of the tangent bundle we have $\Psi(\tau+c\sigma)=\Psi(\tau)+c\Psi(\sigma)$ for any $\tau,\sigma \in \Gamma_f(B)$ and $c\in \R$. We deduce that the initial topology
turns $\Gamma_f(B)$ into a Banach space.

\begin{la}
 The $\BV$-topology on $\Gamma_f(B)$ does not depend on the choices in~\eqref{thetamap}.
\end{la}

\begin{proof}
 Let $\tilde{\Psi}$ be a map defined as in \eqref{thetamap} for different choices of $t_i$ and a different family of bundle trivializations $(\kappa_i,V_i)$. For $a<c<b$, the inclusion $\BV([a,b],F) \rightarrow \BV([a,c],F) \times \BV([c,b],F)$ is a topological embedding onto the closed vector subspace of pairs of mappings which take the same value at the bisection point $c$ (see Lemma~\ref{into-prod}).
  Hence, without loss of generality, we may pass to a common refinement of both subdivisions of $[a,b]$ and only have to show that the topology is independent of the families of trivializations chosen to construct it. We can check this independently for every $i$. Let $(\psi_i,U_i)$ and $(\kappa_i,V_i)$, $1\leq i \leq n$ be the families of bundle trivializations used to construct for $f$ the maps $\Psi$ and $\tilde{\Psi}$, respectively. Exploiting that $\psi_i \colon \pi^{-1}(U_i) \rightarrow U_i \times F$ we construct from $\kappa_i\circ \psi_i^{-1} \colon \pi^{-1}(U_i\cap V_i) \rightarrow U_i \cap V_i \times F$ a smooth map which is linear in the second component
  $$c_{i}^f \colon [t_{i-1},t_i] \times F \rightarrow F, \quad c_i^f (s, x)= \pr_2 \circ \kappa_i\circ \psi_i^{-1} (f(s),x).$$
  Then Lemma \ref{prop:star_smooth} implies that the change of trivialisations induces a continuous linear map  
 $(c_i^f)_\star \colon \BV([t_{i-1},t_i],F) \rightarrow \BV([t_{i-1},t_i],F)$. Its inverse is a mapping of the same type obtained by switching the roles of $\psi_i$ and $\kappa_i$, whence $(c_i^f)_\star$ is an isomorphism of Banach spaces. This shows that the $\BV$-topology does not depend on the choices of trivialisations in \eqref{thetamap}.
\end{proof}

\noindent
As the topology defined by the variation norm is finer than the compact-open topology, the same holds for the topology on $\Gamma_f(B)$ as the next result shows. In the following we consider vector bundles $B$ and denote by $B_x$ the fibre of $B$ above the point $x$.

\begin{la}\label{lem:co-top-weaker}
 The canonical inclusion of $\Gamma_f(B)$ into $C([a,b],B)$ is continuous.
\end{la}

\begin{proof}
 It suffices to show that $\lfloor K,O\rfloor =\{f \in C([a,b],B), f(K)\subseteq O\}$ with $K \subseteq [a,b]$ compact and $O \subseteq B$ open is an open neighborhood for each $\tau \in \Gamma_f(B) \cap \lfloor K,O\rfloor$. Picking such a $\tau$, we refine the partition $a=t_0 < \cdots t_k =b$ and pick bundle trivialisations $(\psi_i,U_i)$ such that for every $1 \leq i \leq k$ there are open sets $V^1_i,V_i^2$ such that $\psi_i \circ  \tau([t_{i-1},t_i] \cap K) \subseteq V_i^1 \times V_i^2 \subseteq \psi_i (\pi^{-1}(U_i \cap O))$. Hence for every $i$, $\pr_2\circ \psi_i \circ \tau$ is contained in the set $\Omega(K,i)\coloneq \lfloor K\cap[t_{i-1},t_i],  V_i^2\rfloor$. (We set $\Omega(K,i)=\BV([t_{i-1},t_i],F)$ for $K\cap [t_{i-1},t_i]=\emptyset$.) Note that for each~$i$, the set $\Omega(K,i)$ is open in $\BV([t_{i-1},t_i],F)$ as the Banach space has a finer topology than the compact-open topology by Remark \ref{rem:cont:inclusion}. If now $\Psi$ is the map \eqref{thetamap} constructed from the data we chose, then $\tau \in \Psi^{-1}(\prod_i \Omega(K,i))$ and by construction, each element in this open set is also contained in $\lfloor K, O\rfloor$. This finishes the proof.
\end{proof}
\noindent
As a consequence, we obtain continuity of the evaluation map.

\begin{la}\label{lem:sect_eval}
 Let $f \in \BV([a,b],M)$, then the evaluation map 
 $$\mathrm{ev} \colon \Gamma_f(B) \times [a,b] \rightarrow B,\quad (\tau,s) \mapsto \tau (s)$$
 is continuous. Moreover, for each $x \in [a,b]$ the point evaluation
 $$\varepsilon_x \colon \Gamma_f(B) \rightarrow B_{f(x)},\quad \varepsilon_x (\tau) = \tau (x)$$
 is continuous linear and hence smooth.
\end{la}

\begin{proof}
 Due to Lemma \ref{lem:co-top-weaker}, the inclusion of $\Gamma_f(B)$ into $C([a,b],B)$ is continuous. Hence continuity of $\mathrm{ev}$ follows from the continuity of the evaluation map for continuous functions (cf., e.g., \cite[Lemma B.10]{Schm}).
 Inserting $x \in [a,b]$, we see that $\varepsilon_x = \mathrm{ev} (\cdot ,x)$ is continuous and takes its values in the fibre $B_{f(x)}$. As the subspace topology coincides with the Banach space topology of the fibre $B_{f(x)}$, we deduce that $\varepsilon_x$ is continuous as claimed. Linearity is obvious and continuous linear mappings are smooth.
\end{proof}

\begin{la}\label{lem:bundle-calc}
Let $\pi_1 \colon B_1 \rightarrow M$ and $\pi_2 \colon B_2 \rightarrow M$ be smooth Banach vector bundles over a Banach manifold $M$. Let $f \in \BV([a,b],M)$. Then we have:
\begin{enumerate}
\item If $\psi \colon B_1 \rightarrow B_2$ is a smooth vector bundle map over the identity, then $\psi \circ \tau \in \Gamma_f (B_2)$ for each $\tau \in \Gamma_f (B_1)$ and $$\Gamma_f (\psi) \colon \Gamma_f (B_1) \rightarrow \Gamma_f (B_2), \quad \tau \mapsto \psi \circ \tau$$
is a continuous linear map.
\item For the Whitney sum $B_1 \oplus B_2$, the Banach space $\Gamma_f (B_1 \oplus B_2)$ is canonically isomorphic to $\Gamma_f (B_1) \times \Gamma_f (B_2)$.
\end{enumerate}
\end{la}

\begin{proof}
(a) If $\tau  \in \Gamma_f (B_1)$, then $\psi \circ \tau \colon [a,b] \rightarrow B_2$ satisfies $\pi_2 \circ \psi\circ \tau =f$. To see that the map is $\BV$, we pick a subdivision $a=t_0 < \cdots < t_n =b$ of $[a,b]$ such that $f([t_{i-1},t_i])$ is contained in $\pi_1^{-1}(U_i) \cap \pi_2^{-1}(V_i)$ for some bundle trivializations $(\kappa_i,U_i)$ of $B_1$ and $(\lambda_i,V_i)$ of $B_2$.
Let $W_i:=U_i\cap V_i$. Shrinking $W_i$ we may assume that there exists a manifold chart of $M$ for $W_i$ (whence we may and will in the following identify it with an open subset of a Banach space). Localizing $\Gamma_f (B_1)$ and $\Gamma_f (B_2)$ using mappings $\Psi_1,\Psi_2$ as in \eqref{thetamap} with respect to the trivializations $(\kappa_i)_i$ and $(\lambda_i)_i$, the map $\Gamma_f (\psi)$ is conjugate to $\theta \colon \prod_{i=1}^n \BV ([t_{i-1},t_i],F) \rightarrow \prod_{i=1}^n \BV ([t_{i-1},t_i],F), \theta=((\omega_i)_\star)_i$. 
 Here $\omega_i$ is the smooth map $\omega_i \colon [t_{i-1},t_i] \times F \rightarrow F, \omega_i(t,x)= (\text{pr}_2 \circ 
 \lambda_j \circ \psi \circ \kappa^{-1}_j)(f(t),v)$. From Proposition \ref{prop:star_smooth}, we deduce that 
 $(\omega_i)_\star$ is smooth and hence
 continuous. We deduce that $\theta$ and thus $\Gamma_f (\psi)$ are continuous. Evaluating at points $\Gamma_f(\psi)$ is easily seen to be linear.

 (b) For $j\in \{1,2\}$ let $\rho_j \colon B_1 \oplus B_2 \rightarrow B_i$ be the map taking $(v_1,v_2)$ to $v_j$ and $\iota_j \colon B_j \rightarrow B_1 \oplus B_2$ be the maps taking $v_j \in B_j$ to $(v_1,0)$ and $(0,v_2)$, respectively. Then by part (a)
 $$(\Gamma_f(\rho_1),\Gamma_f(\rho_2)) \colon \Gamma_f (B_1 \oplus B_2) \rightarrow \Gamma_f (B_1) \times \Gamma_f(B_2)$$ is continuous linear. Its inverse is $(\sigma,\tau) \mapsto \Gamma_f(\iota_1)(\sigma) + \Gamma_f (\iota_2)(\tau)$, whence it is an isomorphism of Banach spaces.
\end{proof}

\section{The Banach manifold {\boldmath$\BV([a,b],M)$}}\label{sect:BV-mfd}
In this section, we construct a manifold structure on the set of $\BV$-functions with values in a smooth Banach manifold $M$.
We shall always assume that $M$
admits a local addition,
and we shall assume that~$M$ is a
pure manifold, modeled
on a single Banach space~$E$
(see Remark~\ref{non-pure}
for the general case).
By contrast, $\BV([a,b],M)$
need not be pure;
we shall model it on a set of
Banach spaces.
\begin{defn}\label{def-loa}
Let $M$ be a smooth manifold. A \emph{local addition} is a smooth map
\[
\Sigma \colon U \to M,
\]
defined on an open neighborhood $U \subseteq TM$ of the zero-section
$0_M:=\{0_p\in T_pM\colon p\in M\}$
such that $\Sigma(0_p)=p$ for all $p\in M$,
\[
U':=\{(\pi_{TM}(v),\Sigma(v))\colon v\in U\}
\]
is open in $M\times M$ (where $\pi_{TM}\colon TM\to M$ is the bundle
projection) and the map
\[
\theta:=(\pi_{TM},\Sigma)\colon U \to U'
\]
is a $C^\infty$-diffeomorphism. If
\begin{equation}\label{bettersigma}
T_{0_p}(\Sigma|_{T_pM})=\id_{T_pM}\;\,
\mbox{for all $p\in M$,}
\end{equation}
we say that the local addition $\Sigma$ is \emph{normalized}.
\end{defn}

\begin{numba}
We mention that a Banach manifold
admits a spray if
it is \emph{smoothly paracompact}
in the sense that
for each
open cover of~$M$, there exists a smooth
partition of unity subordinate to it
(cf.\ Theorem~3.1 in \cite[IV, \S3]{Lang}).
A local addition can then be constructed as a restriction of
the exponential map associated to the spray
(see \cite[IV. \S 4]{Lang} for definitions and proofs).
Irrespective of sprays, each Lie group $G$ admits a local addition. The proof uses that $G$ has a trivial tangent bundle (see, e.g. \cite[C.2]{Schm}).
We mention that each Banach manifold admitting a local addition admits a normalized local addition (see \cite[Lemma A.14]{AGS}).
\end{numba}

\noindent
For the rest of this section, we stipulate the following convention.

\begin{numba}
We denote by $M$ a smooth Banach manifold with a local addition\linebreak
$\Sigma \colon TM \supseteq U \rightarrow M$, such that $\theta\coloneq (\pi_{TM} , \Sigma) \colon U \rightarrow U' \subseteq M \times M$ is a diffeomorphism onto an open subset $U'$.
\end{numba}

\begin{numba}{(Charts for $\BV([a,b],M)$)}
Pick $f \in \BV([a,b],M)$ and let $\Gamma_f$ be the Banach space constructed in Definition \ref{defn:gammaf}. We
consider the set
\[
O_f \coloneq \Gamma_f
\cap \lfloor  [a,b], U\rfloor,
\]
which is an open subset of $\Gamma_f$
by Lemma \ref{lem:co-top-weaker}, and the set
\[
O_f' \coloneq\{g\in \BV([a,b],M) \mid (f,g)([a,b])\subseteq U'\}.
\]
Then $\phi_f \colon O_f \rightarrow O_f'$, $\tau \mapsto \Sigma \circ \tau$ is a bijection with inverse $\phi_f^{-1}(g)=\theta^{-1} \circ (f,g)$. As the zero-section gets mapped by $\phi_f$ to $f$, the sets $O_f'$ cover $\BV([a,b],M)$ and we endow $\BV([a,b],M)$
with the final topology with respect to the mappings $(\phi_f)_{f \in \BV([a,b],M)}$.
We shall see that the maps $(\phi_f)^{-1}\colon O_f'\to O_f$
can be used as the charts for a smooth manifold structure on $\BV([a,b],M)$.
Once the manifold structure is constructed,
we shall refer to these charts as the \emph{canonical charts}.
\end{numba}
\noindent
The final topology on $\BV([a,b],M)$
is the finest topology making the maps $\phi_f$ continuous.
The resulting topology is finer than the compact-open topology, i.e.~the topology
induced by the inclusion $\BV([a,b],M) \subseteq C([a,b],M)$. To see this, view $\phi_f$ as a restriction of $C([a,b],\Sigma) \colon C([a,b],U) \rightarrow C([a,b],M)$ to $\Gamma_f \cap \lfloor [a,b],U\rfloor$. This map is continuous if we endow both sides with the compact-open topology (see e.g.~\cite[Lemma B.8]{Schm}).
Since the $\BV$-topology on $\Gamma_f$ is finer than the compact-open topology (see Lemma \ref{lem:co-top-weaker}), this establishes continuity also with respect to the $\BV$-topology. In particular, the sets $O_f'$ are open in the compact-open topology and thus also in the topological space $\BV([a,b],M)$.

\begin{prop}\label{BVmfd-struct}
For $f \in \BV([a,b], M)$, the maps $\phi_f^{-1} \colon O_f'\rightarrow O_f$ form a smooth manifold atlas for $\BV([a,b],M)$ turning it into a Banach manifold. 
The manifold structure does not depend on the choice of local addition used in its construction.
\end{prop}

\begin{proof}
We have to prove that the transition maps
$\phi_g^{-1}\circ \phi_f$ are smooth for all
$f,g \in \BV ([a,b],M)$. We first note that the preimage
$f^{-1}(O_g')$ is open in $\Gamma_f$ as $O_g'$ is open in the final topology.
Notably, the domain of the transition map will be open if $O_f'\cap O_g'
\not=\emptyset$. 
Inserting the definitions yields 
 \begin{align}\label{formula:chch}
  \phi_f^{-1} \circ \phi_g (\tau) = \theta^{-1}(f,\Sigma \circ \tau).
 \end{align}
Use charts of $M$ to localize \eqref{formula:chch}.
To keep the notation simple, we shall for the rest of the proof assume the following: Restrict to small intervals $[r,s] \subseteq [a,b]$, such that each of the compact subsets $f([r,s]),g([r,s])$ and $\Sigma(\tau([r,s]))$ is contained in a single chart domain (we do not assume that the same chart is used simultaneously). Covering $[a,b]$ with such intervals we can work locally. The local choices restrict us to an open subset of the domain of the change of charts (but note that we can choose for every $\tau$ in the domain such a neighborhood containing $\tau$). 
 Summing up, we restrict to the interval $[r,s]$ to rewrite \eqref{formula:chch} as
 \begin{align}\label{chch:locform}
  \theta^{-1}(f|_{[r,s]},\Sigma \circ \tau|_{[r,s]}) = \theta^{-1}\circ (\varphi^{-1} \times \kappa^{-1}) \circ (\varphi\circ f|_{[r,s]},\kappa\circ \Sigma \circ \tau|_{[r,s]}) 
  \end{align}
  for suitable charts $(\varphi, U_\varphi)$ and $(\kappa, U_\kappa)$ of $M$. Since also $g$ maps $[r,s]$ into a chart domain of say $(\psi,U_\psi)$, we may replace $\Gamma_g$ and $\Gamma_f$ for the following argument with the space $\BV([r,s],E)$, due to \ref{resolve_Gammaf}. So up to a harmless identification in charts, we rewrite the
  transition map \eqref{formula:chch} (locally on a subinterval) as the composition of two mappings
  \begin{align} \label{eq:map1}
   \delta_{f,r,s} \colon \BV([r,s],E) \supseteq D &\rightarrow \BV([r,s],E\times E), \\ \tilde{\tau}|_{[r,s]} &\mapsto (\varphi \circ f, \kappa\circ \Sigma \circ \tau|_{[r,s]}) \notag \\
   \gamma_{r,s} \colon \BV([r,s],E\times E)&\supseteq \BV([r,s],\varphi(U_\varphi)\times \kappa(U_\kappa)) \rightarrow \BV([r,s],E) \label{eq:map2}\\
    (h_1,h_2)& \mapsto \theta^{-1}\circ (\varphi^{-1},\kappa^{-1}) \circ (h_1,h_2). \notag
  \end{align}
 where $D$ in \eqref{eq:map1} is an open set and $\tilde{\tau}$ the mapping which gets identified with $\tau_{[r,s]}$ by \ref{resolve_Gammaf}. We can further dissect $\delta_{f,r,s}$ as a pushforward of a $\BV$-function by a smooth map (which is smooth in the $\BV$-function by Lemma \ref{lem:PF_smooth}) and the mapping $\BV([r,s],E) \rightarrow \BV([r,s],E\times E)$,
$h \mapsto (\varphi \circ f, h)$. Also this mapping is smooth in $h$ due to Lemma \ref{lem:smooth_insertion}. We conclude that $\delta_{f,r,s}$ is smooth in $\tau$.
For the map $\gamma_{r,s}$, \eqref{eq:map2} shows that it is given by a pushforward of
$\BV$-functions with a smooth function, whence it is smooth by Lemma \ref{lem:PF_smooth}. We conclude that the transition map is smooth,
being a composition of smooth mappings. Summing up, we obtain for each local addition $\Sigma$ a smooth atlas $(\phi_f^\Sigma, O_f')_{f \in \BV([a,b],M}$. 

\noindent
 To see that the manifold structure does not depend on the choice of local addition, consider another local addition $\tilde{\Sigma}$ with associated open sets $\tilde{U}, \tilde{U}'$ and diffeomorphism $\tilde{\theta}$. Recall that the topology on the modeling spaces is finer than the compact-open topology. Hence, an argument as above shows that the domain of transition
maps for canonical charts with respect to different local additions is an open set. Hence it suffices to prove that the transition map induced by arbitrary pairs $\phi_f^\Sigma, \phi_g^{\tilde{\Sigma}}$ is smooth. Again we argue with a local formula as in \eqref{chch:locform}. The key insight is that we are always composing with smooth functions. Since the mappings induced by $\tilde{\Sigma}$ are also smooth, we can argue as above to see that also the transition maps between canonical charts with respect to different local additions are smooth. 
Summing up, the construction of the manifold structure does not depend on the choice of the local addition. 
\end{proof}
\noindent
% As usual,
We would like to identify the tangent bundle of the manifold $\BV([a,b],M)$ with the manifold of $\BV$-functions with values in the tangent manifold $TM$. To this end,
we recall from \cite[Lemma A.11]{AGS} that $TM$ admits a local addition if $M$ admits a local addition. In fact, if $\Sigma\colon U \rightarrow M$ is the local addition on $M$, then $$\Sigma_{TM} \coloneq T\Sigma \circ \kappa \colon \kappa (TU) \rightarrow TM$$
defines a local addition, where $\kappa \colon T^2 M \rightarrow T^2M$ denotes the canonical flip, i.e.\ the involution of the bundle given locally as $(x,y,z,w)\mapsto (x,z,y,w)$.
In particular, we can apply Proposition \ref{BVmfd-struct} to construct a Banach manifold structure for $\BV ([a,b], TM)$.
Essentially as in \cite[Theorem A.12]{AGS} one now establishes the following result:

\begin{prop}\label{prop:tangent_ident}
 Let $M$ be a Banach manifold admitting a local addition $\Sigma \colon U \rightarrow M$. Then 
 $$\BV ([a,b],\pi_M) \colon \BV ([a,b],TM) \rightarrow \BV ([a,b],M), \quad F \mapsto \pi_{M} \circ F$$ is a smooth vector bundle with fibre $\Gamma_f$ over $f \in \BV([a,b],M)$. For each $v \in T\BV([a,b], M)$, we have $\Phi (v)\coloneq (T\varepsilon_x (v)) \in \BV([a,b],TM)$ and the map 
 $$\Phi \colon T\BV([a,b],M) \rightarrow \BV([a,b],TM),\quad v \mapsto \Phi(v)$$
 is an isomorphism of smooth vector bundles (over the identity). We shall write $\Phi_M$ instead of $\Phi$ if we wish to emphasize the manifold $M$.
\end{prop}
\noindent
For the reader's convenience, a detailed proof is recorded in Appendix \ref{app:details_tangent}. The smoothness of the mapping $\BV([a,b],\pi_M)$ followed a posteriori from the bundle identification in the proof of Proposition \ref{prop:tangent_ident}. However, smoothness for superposition operators will be studied systematically in the next section. As it fits thematically in the present
section, we state the next result here. Its proof is also postponed to Appendix \ref{app:details_tangent} as we need to treat superposition operators first.

\begin{la}\label{lem:product_mfd}
Let $M_1,M_2$ be Banach manifolds which admit local additions. Denote by $\text{pr}_i \colon M_1 \times M_2 \rightarrow M_i$ the canonical projection for $i  \in \{1,2\}$. Then
\begin{align*}
\BV([a,b],(\pr_1,\pr_2)) \colon \BV ([a,b],M_1\times M_2) &\rightarrow \BV([a,b],M_1) \times \BV([a,b],M_2),\\ f &\mapsto (\pr_1 \circ f, \pr_2 \circ f)
\end{align*}
is a diffeomorphism.
\end{la}

\noindent
Finally, we have an analog of Lemma \ref{into-prod} for the manifold of mappings.

\begin{la}\label{mfd_splitting_intervall}
    For $a=t_0 < t_1 < \cdots < t_n=b$, the mapping
    \begin{align}\label{splitting:interval_mfd}
    \BV([a,b],M)\rightarrow \prod_{i=1}^n \BV([t_{i-1},t_i],M), \quad f \mapsto (f|_{[t_{i-1},t_i]})_{i=1}^n
    \end{align}
is an embedding of $C^\infty$-manifolds
onto a closed submanifold.
\end{la}

\begin{proof}
 Consider $f \in \BV([a,b],M)$ and pick a refinement $a=t_0=s_0< s_1 < \cdots < s_m = t_n =b$ of the subdivision $(t_i)_{i=0}^n$ such that the image of each $f|_{[s_{j-1},j]}$, $j=1,\ldots , m$, is contained in the
 domain $U_j$ of a single chart $(\psi_j,U_j)$ of $M$.
 Let $f_i \coloneq f|_{[t_{i-1},t_i]}$ for $i\in\{1,\ldots, n\}$. 
 We use the canonical manifold charts $\phi_f^{-1}$ and $\prod_{i=1}^n \phi_{f_i}^{-1}$ to conjugate the mapping \eqref{splitting:interval_mfd} to a map between bundle sections. Resolving the spaces of bundle sections using the refined subdivision shows that, up to a reordering of factors, this map is the map  
 \begin{align}
  \Gamma_f(TM) \rightarrow \prod_{j=1}^m \BV([s_{j-1},s_j], F), \quad \tau \mapsto (\pr_2 \circ \, T\psi_j \circ \tau|_{[s_{j-1},s_j]})_j
 \end{align}
 from \eqref{thetamap}. This is an embedding whose image is a closed vector subspace (cf.\ \ref{resolve_Gammaf}). Hence the canonical charts conjugate \eqref{splitting:interval_mfd} to a smooth embedding. Its image is a closed submanifold (where the canonical charts form submanifold charts). This finishes the proof.
\end{proof}
\begin{rem}\label{non-pure}
If $M$ is a Banach manifold modeled on a set $\cE$
of Banach spaces, let $\cC$ be the set of connected
components of~$M$. Each $N\in \cC$
can be regarded as a pure Banach manifold.
If each $N\in \cC$ admits a local addition,
we obtain Banach manifolds $\BV([a,b],N)$.
If desired, $\BV([a,b],M)$ can be given
the topology making
it the topological sum of the topological spaces $\BV([a,b],N)$, and the smooth manifold structure
making each $\BV([a,b],N)$ an open smooth submanifold.
\end{rem}
\begin{rem}\label{rem-open-subman}
Let $a<b$ be real numbers, $M$ be a Banach
manifold admitting a local addition and $Q\sub M$
be an open subset. Then $\BV([a,b],Q)$
is a smooth $C^\infty$-submanifold of $\BV([a,b],M)$.
In fact, a local addition for~$M$
restricts to a local addition for~$Q$;
the standard charts for $\BV([a,b],Q)$
constructed using the latter are also
standard charts for $\BV([a,b],M)$.
\end{rem}
\section{Canonical mappings between manifolds of
BV-functions}\label{sec-canonical}
In this section, we investigate superposition operators between manifolds of manifold-valued $\BV$-functions. This will allow us to construct Lie groups of $\BV$-functions taking values in Banach--Lie groups. We shall first deal with the evaluation before moving on to the more challenging pushforward and pullback operations.

\begin{cor}\label{cor:smooth_eval}
    For each $x \in [a,b]$, the evaluation map $$\varepsilon_x \colon \BV([a,b],M) \rightarrow M,\quad f \mapsto f(x)$$ is smooth.
\end{cor}

\begin{proof}
It suffices to note that around each $f \in \BV([a,b],M)$, we can localize in the canonical chart $\varphi_f$ around $f$. This yields the identity
$\varepsilon_x \circ \phi_f = \Sigma \circ \epsilon_x|_{O_f} \colon \Gamma_f \supseteq O_f \rightarrow M$.
Now $\epsilon_x$ is the evaluation on the bundle sections $\Gamma_f$ which is smooth by Lemma \ref{lem:sect_eval} and $\Sigma$ is smooth, being the local addition. Hence $\varepsilon_x$ is smooth and this concludes the proof.
\end{proof}
\begin{prop}\label{prop:smooth_pushfwd}
Let $\varphi \colon M \rightarrow N$ be a $C^{k+2}$-mapping between Banach manifolds which admit local additions. Then the following map is $C^k$:
$$\BV([a,b],\varphi) \colon \BV([a,b],M) \rightarrow \BV([a,b],N), \quad f \mapsto \varphi \circ f.$$
Further, $T\BV([a,b],\varphi)=\Phi_N^{-1}\circ \BV([a,b],T\varphi)\circ \Phi_M$ and in particular, for every $f \in \BV([a,b],M)$ the map $T_f \BV([a,b],M)$ identifies as the continuous linear map $\Gamma_f (T\varphi)\colon \Gamma_f \rightarrow \Gamma_{\varphi \circ f}, \tau \mapsto T\varphi \circ \tau$.
\end{prop}

\begin{proof} We shall denote by $\Sigma_M \colon U_M \rightarrow M$ and $\Sigma_N \colon U_N \rightarrow N$ the local additions on $M$ and $N$. Further, let $M$ and $N$ be modeled on the Banach spaces $E$ and $F$, respectively. For continuity and differentiability of $\BV([a,b],\varphi)$ we exploit that both properties can be checked on a cover of open sets.

\textbf{Step 1: Localization on the BV-manifolds.} Pick $f \in \BV([a,b], M)$ and note that the domain $O_{\varphi \circ f}$ of the canonical chart $\phi_{\varphi \circ f}^{-1}$ is an open set in the compact-open topology. The pushforward with $\varphi$ is continuous in the compact-open topology, cf.\ \cite[Lemma B.8]{Schm}. Hence, as the topology of $\Gamma_f$ is finer than the compact open topology, there is an open neighborhood $\Omega_f \subseteq \Gamma_f$ of $0$ such that $\varphi \circ \phi_f (\Omega_f) \subseteq O_{\varphi \circ f}$. \smallskip

\textbf{Step 2: Localization in chart of {\boldmath$M$} and \boldmath{$N$}.} Subdivide $[a,b]$ into subintervals $I_1, \ldots, I_n$ such that on each subinterval $I_j$ there are charts $(\kappa_j,V_j)$ of $M$ and $(\theta_j, W_j)$ such that 
\begin{enumerate}
\item $W_j \times W_j \subseteq (\pi_N, \Sigma_N)(U_N)$
\item $f(I_j) \subseteq V_j$ and $\varphi(V_j) \subseteq W_j$. 
\end{enumerate}
Then $Z_j \coloneq (\pi_M, \Sigma_M)^{-1} (V_j \times \varphi^{-1}(W_j)$ is an open $0$-neighborhood in $TV_j$ and we can pick $r_j >0$ such that $\{v \in TV_j \mid \pi_M(v) \in f(I_j) \text{and} \lVert T\kappa (v)\rVert < r_j\} \subseteq Z_j$. 

As $f = \Sigma_M (0)$ for $0\in \Gamma_f$, we may shrink $\Omega_f$ such that for each $1\leq j\leq n$, we have $\Sigma_M \circ \tau (I_j) \subseteq \varphi^{-1}(W_j)$ and $\sup_{t\in I_j}\lVert \text{pr}_2 T\kappa_j (\tau(t))\rVert_E < r_j$ for all $\tau\in \Omega_f$. As in \eqref{thetamap}, we now construct two embeddings $\Psi_f\colon \Gamma_f \rightarrow \prod_j \BV (I_j,E)$ and $\Psi_{\varphi\circ f} \colon \Gamma_{\varphi \circ f} \rightarrow \prod_j \BV (I_j, F)$.
 \smallskip

\textbf{Step 3: {\boldmath$\BV([a,b],\varphi)$} is locally {\boldmath$C^k$}.} To simplify the notation, we now drop the indices from Step~2, writing $\kappa$, $\theta$ and so forth instead of $\kappa_j, \theta_j$ (thus mappings will be defined on $[a,b]$). We prove that on the open $0$-neighborhood $\Psi_f(\Omega_f)$, the map $\Psi_{\varphi \circ f}\phi_{\varphi \circ f}^{-1} \circ \BV([a,b],\varphi) \circ \phi_f \Psi_f^{-1}$ is $C^k$. Note that, inverting \eqref{thetamap}, one obtains $\Psi_f^{-1}(\tau)=T\kappa^{-1}(\kappa \circ f, \tau)$, where $\kappa$ is the chart from Step~2. For $\tau \in \Psi_f (\Omega)$ this yields
\begin{align}\notag
&\Psi_{\varphi \circ f}\circ \phi_{\varphi \circ f}^{-1} \circ \BV([a,b],\varphi) \circ \phi_f \circ \Psi_f^{-1}(\tau) \\
 =&\text{pr}_2 \circ T\theta \circ (\pi_N, \Sigma_N)^{-1}) (\varphi \circ f,\varphi \circ \Sigma_M \circ T\kappa^{-1}(\kappa \circ f,\tau)). \label{eq:loc_form}
\end{align}   
We shall now construct several differentiable mappings between spaces of $\BV$-mappings which allow us to rewrite \eqref{eq:loc_form} as a composition with the desired differentiability conditions. The following mappings are
well-defined by our choices in Step~2:
\begin{align*}
i_1 &\colon \BV([a,b], E) \rightarrow \BV([a,b], \kappa(V) \times E), \quad \tau \mapsto (\kappa \circ f, \tau)\\
m_1 &\colon \BV([a,b], T\kappa (Z))\rightarrow \BV([a,b], \theta(W)),\quad \gamma \mapsto \theta \circ \varphi \circ \Sigma_M \circ T\kappa^{-1} (\gamma)\\
i_2 &\colon \BV([a,b], \theta(W)) \rightarrow \BV([a,b], \theta(W) \times \theta(W)),\quad g \mapsto (\theta \circ \varphi \circ f,g)\\
m_2 &\colon \BV([a,b], \theta(W) \times \theta(W)) \rightarrow \BV ([a,b],F), \\ 
  & \hspace{3.21cm} (g,h) \mapsto \text{pr}_2 \circ T\theta \circ ( \pi_N,\Sigma_N)^{-1}\circ (\theta,\theta)^{-1} \circ (g,h). \notag
\end{align*}
Then $i_1,i_2$ are smooth by Lemma \ref{lem:smooth_insertion} and $m_2$ is smooth by Lemma \ref{lem:PF_smooth}. Now as $\varphi$ is $C^{k+2}$, $m_1$ implements a pushforward by a $C^{k+2}$-map, whence $m_1$ is $C^k$ by Proposition \ref{prop:star_smooth}. Recall from Step~2 that $\text{pr}_2 \circ T\kappa$ maps all sections in $\Omega_f$ to the ball of radius~$r$ in~$E$ and~$i_1$ maps this ball to $T\kappa (Z)$. Hence, it makes sense to compose $i_1$ and $m_1$ for sections in $\Omega_f$. By construction, we can then rewrite \eqref{eq:loc_form} as the composition $m_2 \circ i_2 \circ m_1\circ i_1$, whence the formula yields a $C^k$-map on $\Psi_f(\Omega_f)$. In particular, $\BV([a,b],\varphi)$ is $C^k$ on each open set $\phi_f (\Omega_f)$. As these sets cover $\BV([a,b],M)$, we deduce that $\BV([a,b],\varphi)$ is $C^k$. \smallskip

\textbf{Step 4: Formula for the derivative.} The assertions on the derivative can be proved by direct calculation using a normalized local addition. We refer to \cite[Corollary A.15]{AGS} for a detailed derivation.
\end{proof}

\noindent
Adapting the localization argument in the proof of Proposition \ref{prop:smooth_pushfwd}, we obtain a manifold version of Proposition \ref{prop:star_smooth}:

\begin{prop}\label{prop:mfd_star_smooth}
Let $k\in \N_0\cup\{\infty\}$
and $\gamma \colon [a,b] \times M \rightarrow N$ be a $C^{k+2}$-map,
where $M$ and $N$ are Banach manifolds which admit local additions. 
Then the following map is $C^k$:
\[
\gamma_\star \colon \BV([a,b],M) \rightarrow \BV([a,b],N),\quad f \mapsto \gamma \circ (\id_{[a,b]},f).
\]
Notably, $\gamma_\star$ is smooth whenever $\gamma$
is smooth.
\end{prop}
\begin{proof}
Note first that $\gamma_\star \colon C([a,b],M) \rightarrow C({a,b},N)$ makes sense and is continuous with respect to the compact-open topology, cf.\ \cite[Proposition 4.10]{GS22}. Since the topology on $\BV([a,b],M)$ is finer than the compact-open topology and for $f\in \BV([a,b],M)$ the domain $O_{\gamma_\star (f)}`$ of the canonical chart in $\BV([a,b],N)$ is an open subset with respect to the compact open topology, we deduce that $\gamma^{-1}_{\star}(O_{\gamma_\star (f)}`)$ is an open $f$-neighborhood. Hence, it suffices to prove that for each $f \in \BV([a,b],M)$ the map $\gamma_f \coloneq \phi_{\gamma_\star (f)} \circ \gamma_\star \circ  \phi_f$ is a $C^k$-map on some open neighborhood of $0 \in \Gamma_f$ (we shall suppress the open set in the notation).
 We can now proceed exactly as described in Step~2 and Step~3 of the proof of Proposition \ref{prop:smooth_pushfwd} and localize the sections in $\Gamma_f$ and $\Gamma_{\gamma_\star (f)}$ in suitable manifold charts. This results in a family of mappings which are $C^{k+2}$ and of the form $g_\star$ on $I \times U$ where $I$ is a compact interval and $U$ open in a Banach space. Since these $g_\star$ take values in a Banach space, we conclude with Proposition \ref{prop:star_smooth} that they are of class $C^k$. In conclusion also $\gamma_\star$ is of class $C^k$.
\end{proof}
\noindent
Lifting Lemma \ref{BV:aff_repara} to manifolds of $\BV$-functions, we obtain the following.

\begin{la}
Let $M$ be a Banach manifold with local addition and $\alpha \colon [c,d] \rightarrow [a,b], s \mapsto a+ b\cdot (s-c)/(b-a)$. Then the pullback
$$\BV(\alpha,M) \colon \BV([a,b],M) \rightarrow \BV ([c,d],M), \quad f \mapsto f \circ \alpha$$
is a diffeomorphism of Banach manifolds.
\end{la}

\begin{proof}
Localizing in manifold charts, Lemma \ref{BV:aff_repara} shows that the pullback is a well-defined bijection. It is
sufficient to prove for every $f \in \BV([a,b],M)$ that $\phi_{f \circ \alpha}^{-1} \circ \BV(\alpha,M) \circ \phi_f$ is smooth (since we can obtain smoothness of the inverse by applying the same argument to $\BV(\alpha^{-1},M)$). This reduces the problem to showing that the pullback induces an isomorphism of Banach spaces between $\Gamma_f$ and $\Gamma_{\alpha \circ f}$. The two
spaces correspond to closed vector subspaces of a product over~$k$ of the Banach spaces $\BV(I_k,E)$ and $\BV(J_k,E)$,
respectively,
where the intervals
$I_k$ subdivide $[c,d]$ and the $J_k := \alpha (I_k)$
subdivide $[a,b]$. The statement follows by a direct application of Lemma \ref{BV:aff_repara} to the factors.
\end{proof}

\subsection*{The Lie group {\boldmath$\BV([a,b],G)$}}
We shall now leverage the results on canonical mappings between manifolds of $\BV$-functions to construct the $\BV$-analogue of current groups. To this end, we
assume throughout the section that $G$ is a Banach--Lie group with associated Lie algebra $\cg:=\Lf (G)$. Recall from \cite[C.2]{Schm} that each Lie group admits a local addition.
Hence $\BV([a,b],G)$ can be made a smooth manifold using Proposition \ref{BVmfd-struct}.
For $x\in [a,b]$,
let $\ve_x\colon \BV([a,b],G)\to G$
be the evaluation at~$x$.
The discussion of the manifolds of $\BV$-functions in the previous sections immediately yields the following result.

\begin{prop}\label{prop:Lie_grp}
The pointwise group
operations turn $\BV([a,b],G)$ into a Banach--Lie group.
The pointwise
Lie-bracket $\BV([a,b],\cg)\times\BV([a,b],\cg)\to\BV([a,b],\cg)$ is continuous, and the map
\begin{equation}\label{iso-lieas}
\Lf(\BV([a,b],G))\to \BV([a,b],\cg),\quad
v\mto (\Lf(\ve_x)(v))_{x\in [a,b]}
\end{equation}
is an isomorphism of topological Lie algebras.
\end{prop}

\begin{proof}
By definition, the pointwise operations are given by composition with the group multiplication and group inversion on $\BV([a,b],G)\times \BV([a,b],G) \cong \BV([a,b],G\times G)$ (Lemma \ref{lem:product_mfd}) and $\BV([a,b],G)$,
respectively. As these operations are smooth maps on $G\times G$ and $G$, also the group operations of $\BV([a,b],G)$ are smooth by Proposition \ref{prop:smooth_pushfwd}.
Likewise, the pointwise Lie bracket on $\BV([a,b],\cg)$
is continuous by Lemmas \ref{prod:iso}
and \ref{improve-chn}.\smallskip

\noindent
Let $e\in G$ be the neutral element and
$\mathbf{1} \colon [a,b] \rightarrow G$ be the constant map
with value~$e$.
We deduce from Proposition \ref{prop:tangent_ident} that 
\begin{align}\label{ident:Liealg}
\Lf (\BV([a,b],G)= T_{\mathbf{1}} \BV([a,b],G) \cong \Gamma_{\mathbf{1}} (TG) =\BV([a,b],\Lf(G))
\end{align}holds as locally convex vector spaces,
by means of the map $h$ described in~(\ref{iso-lieas}).
It remains to note that $\Lf(\ve_x)$
is a Lie algebra homomorphism for each $x\in [a,b]$,
as $\ve_x$ is a smooth group homomorphism
(Corollary~\ref{cor:smooth_eval}).
Thus $h$ is a Lie algebra homomorphism
and hence an isomorphism of topological Lie algebras.
\end{proof}
\noindent
The next observation is useful
in the proof of Lemma~\ref{to-constant-smooth}, and can be reused.
\begin{numba}\label{identity-chart}
Let ${\bf 1}\colon [a,b]\to G$ be the constant function
with value~$e$. 
Then $\Gamma_{\bf 1}=\BV([a,b],\cg)$
with $\cg:=\Lf(G)$. Pick a local addition $\Sigma\colon \Omega\to G$ for~$G$.
Then $Q:=\Omega\cap \cg$ is an open $0$-neighborhood,
$P:=\Sigma(Q)$ is an open $e$-neighborhood in~$G$
and $\kappa:=\Sigma|_Q\colon Q\to P$
is a $C^\infty$-diffeomorphism.
We have $\cO_{\bf 1}=\BV([a,b],Q)$
and the inverse of the canonical chart is
$\varphi_{\bf 1}=\BV([a,b],\kappa)\colon \BV([a,b],Q)\to
\BV([a,b],P)$, $\tau\mto\kappa\circ\tau$.
\end{numba}
\begin{la}\label{to-constant-smooth}\label{rem:ev_subm}
Let $a<b$ be real numbers and $G$ be a Banach--Lie group,
with neutral element~$e$.
For $g\in G$, let $c_g\colon [a,b]\to G$
be the constant mp with value~$g$.
Then the following holds.
\begin{itemize}
\item[\rm(a)]
The map $c\colon G\to \BV([a,b],G)$, $g\mto c_g$
is an embedding of smooth manifolds
and a group homomorphism.
\item[\rm(b)]
The map $\ve_a\colon \BV([a,b],G)\to G$,
$\gamma\mto \gamma(a)$
is a $C^\infty$-submersion and a group homomorphism.
Its kernel
\[
\BV_*([a,b],G):=\{\gamma\in \BV([a,b],G)\colon \gamma(a)=e\}
\]
is a Lie subgroup of $\BV([a,b],G)$.
\item[\rm(c)]
The map $\omega_b\colon
\BV_*([a,b],G)\to G$, $\gamma\mto\gamma(b)$
is a smooth submersion and a group homomorphism.
\end{itemize}
\end{la}
\begin{proof}
Let $\|\cdot\|_\cg$ be a norm on $\cg$ defining its topology.\smallskip

(a) and (b):
For $v\in \cg$,
let $s_v\colon [a,b]\to\cg$ be the constant map
with value~$v$. Then $(s_v)'=0$,
whence $\|s_v\|_{BV}^{\st}=\|s_v\|_\infty=\|v\|_{\cg}$.
The linear map $\delta \colon \BV([a,b],\cg)\to\cg$,
$\tau\mto\tau(a)$ is continuous
and has the linear map
$s\colon \cg\to\BV([a,b],\cg)$, $v\mto s_v$
as a right inverse, which is continuous
with $\|s\|_{\op}\leq 1$.
Hence $\delta$ is a submersion. As a consequence,
$\ve_a|_{\BV([a,b],P)}=
\kappa\circ \delta|_{\BV([a,b],Q)}\circ \varphi_{\bf 1}^{-1}$
is a submersion and hence so is~$\ve_a$,
being a group homomorphism. The final assertion of~(b)
follows since level sets of submersions are
submanifolds.\smallskip

Since $c\circ \kappa=\varphi_{\bf 1}\circ s|_Q$
is an immersion, also $c|_P$ is an immersion
and hence also~$c$, being a group homomorphism.
Moreover, $c$ is a homeomorphism onto its image,
as $\ve_a\circ \, c=\id_G$. Hence~$c$ is an embedding of
$C^\infty$-manifolds.\smallskip

(c) $\BV_*([a,b],\cg)$ is a closed vector subspace
of $\BV([a,b],\cg)$ and $\varphi_{\bf 1}$
takes $\BV_*([a,b],Q):=\BV_*([a,b],\cg)\cap\BV([a,b],Q)$
onto the intersection $\BV_*([a,b],P)$
$:=\BV_*([a,b],G)\cap \BV([a,b],P)$.
The restriction $\psi$ of $\varphi_{\bf 1}$
to a bijective map $\BV_*([a,b],Q)\to\BV_*([a,b],P)$
therefore is the inverse of a chart for the submanifold
$\BV_*([a,b],G)$.
For $v\in \cg$, consider the function
\[
f_v\colon [a,b]\to\cg,\quad
t\mto \frac{t-a}{b-a}v
\]
with $\|f_v\|_\infty=\|v\|_{\cg}$.
For the restriction $\lambda$ of Lebesgue measure to
$[a,b]$, we have $f_v'=\frac{1}{b-a}v\, \lambda$
with variation norm $\|f_v'\|=\|v\|_\cg$,
whence $\|f_v\|_{BV}^{\st}=\|v\|_\cg$.
The mapping $R\colon \BV_*([a,b],\cg)\to \cg$,
$\tau\mto \tau(b)$ is continuous linear
and has $f\colon \cg\to \BV_*([a,b],\cg)$,
$v\mto f_v$ as a right inverse that is continuous
linear with $\|f\|_{\op}\leq 1$. So $R$
is a submersion. As $\omega_b$
is a group homomorphism and
coincides with the submersion
$\kappa\circ R\circ \psi^{-1}$
on some ${\bf 1}$-neighborhood, $\omega_b$
is a submersion.
\end{proof}
\begin{rem}\label{semidirect-lie}
Note that the smooth group homomorphism
$c\colon G\to \BV([a,b],G)$
is a right inverse for~$\ve_a$.
We identify $G$ with the Lie subgroup $c(G)$.
By the preceding, $\BV([a,b],G)$ is
the following semidirect product as a Lie group:
$$
\BV([a,b],G) = \BV_\ast ([a,b],G)\rtimes G .
$$
\end{rem}
\begin{rem}
As $\omega_b$ is a group homomorphism and a
smooth submersion, its kernel
$$
\BV_\ast^\ell([a,b],G) \coloneq \{f \in \BV_\ast ([a,b],G) \colon f(b)=e\}
$$
is a Lie subgroup of $\BV_*([a,b],G)$.
As a consequence,
\[
\ell^{\BV}_{[a,b]} (G) \coloneq \{f \in \BV([a,b],G)
\colon f(a)=f(b)\}
=\BV^\ell_*([a,b],G)\rtimes G
\]
is a Lie subgroup of $\BV_*([a,b],G)\rtimes G=\BV([a,b],G)$.
It may be viewed as a $\BV$-loop group $\BV(\mathbb{S}_1,G)$,
but we refrain from giving a formal definition of $\BV$-functions on the circle.
\end{rem}
\section{Differential equations for {\boldmath$\BV$}-functions}\label{sec-BV-ode}
We define differential equations
for $\BV$-functions and discuss local uniqueness,
local existence and parameter-dependence of
$\BV$-solutions. We begin with differential
equations on open subsets of Banach spaces;
in a second step, $\BV$-differential equations on Banach
manifolds are addressed.
The discussion may be of wider interest,
irrespective of Lie theory.
\subsection*{Local theory of {\boldmath$\BV$}-differential equations}
We consider the following situation.
\begin{numba}\label{situ-bvde}
Let $(E,\lVert\cdot\rVert_E)$ and
$(F,\lVert\cdot\rVert_F)$ be Banach spaces, $U\sub F$
be an open subset and $f\colon E\times U\to F$
be a mapping such that $\wt{f}(y):=f(\cdot,y)\colon E\to F$
is continuous linear for each $y\in U$,
and $\wt{f}\colon U\to(\cL(E,F),\|\cdot\|_{\op})$
is continuous.
\end{numba}
\begin{defn}\label{defn-sol}
Let $a<b$ be in $\R$ and $\mu\in \Mdna([a,b],E)$.
For $\alpha,\beta\in [a,b]$ with $\alpha<\beta$,
we say that a
$\BV$-function $\gamma\colon [\alpha,\beta]\to U$
is a \emph{$\BV$-solution} to 
\begin{equation}\label{diff-cara}
y^\prime=f_*(\mu,y)  
\end{equation}
if $\gamma'=f_*(\mu,\gamma)$ holds,
using $\gamma'\in \Mdna([\alpha,\beta],F)$
as in \ref{def-f-prime}
and $f_*(\mu,\gamma):=
f_*(\mu|_{[\alpha,\beta]},\gamma)$ (see Remark~\ref{relax}).
If $t_0\in [a,b]$ and $y_0\in U$
are given, we call a $\BV$-function
$\gamma\colon[\alpha,\beta]\to U$
a \emph{solution to the $\BV$-initial value problem}
\begin{equation}\label{bv-initi}
y'=f_*(\mu,y),\quad y(t_0)=y_0
\end{equation}
if $\gamma$ is a $\BV$-solution to (\ref{diff-cara})
such that $t_0\in [\alpha,\beta]$
and $\gamma(t_0)=y_0$.
\end{defn}
\begin{rem}
For $\alpha,\beta\in [a,b]$ with $\alpha<\beta$,
a function $\gamma\in \cL^\infty_{\rc}([a,b],U)$
(e.g., a continuous function $\gamma\colon [a,b]\to U$)
is a $\BV$-solution to (\ref{diff-cara})
if and only if,
for some $t_1\in[\alpha,\beta]$,
\[
\gamma(t)=\left\{
\begin{array}{cl}
\gamma(t_1) + f_*(\mu,\gamma)([t_1,t]) &\mbox{\,if $\,t\geq t_1$,}\\
\gamma(t_1) - f_*(\mu,\gamma)([t,t_1]) &\mbox{\,if $\,t\leq t_1$;}
\end{array}\right.
\]
this then holds for all $t_1\in [\alpha,\beta]$
(see Lemma~\ref{bv-basepoint}).
The continuous function $\gamma\colon [\alpha,\beta]\to U$
is a solution to the $\BV$-initial value problem (\ref{bv-initi}) if and only if $t_0\in [\alpha,\beta]$
and
\[
\gamma(t)=\left\{
\begin{array}{cl}
y_0 + f_*(\mu,\gamma)([t_0,t]) &\mbox{\,if $\,t\geq t_0$,}\\
y_0 - f_*(\mu,\gamma)([t,t_0]) &\mbox{\,if $\,t\leq t_0$.}
\end{array}\right.
\]
\end{rem}
\begin{defn} In the setting of~\ref{situ-bvde},
we say that the differential equation
\begin{equation*}
y'=f_*(\mu, y) 
\end{equation*}
satisfies \emph{local uniqueness of $\BV$-solutions} if for all $\BV$-solutions $\gamma_1\colon  I_1 \rightarrow U$ and $\gamma_2\colon I_2 \rightarrow U$ of \eqref{diff-cara} and $t_0 \in I_1 \cap I_2$ with $\gamma_1(t_0)=\gamma_2(t_0)$, there exists a $t_0$-neighborhood~$K$
in $I_1\cap I_2$
such that
$$
\gamma_1|_K=\gamma_2|_K.
$$
\end{defn}
\noindent
An observation
helps to study local uniqueness,
and can be re-used later.
\begin{la}\label{so-subdivide}
Let $(E,\lVert\cdot\rVert_E)$
a Banach space, $\mu\in \Mdna([a,b],E)$,
and $\ve>0$.
Then there exists $\delta>0$
such that $\Var(\mu)([\alpha,\beta])<\ve$
for all $\alpha,\beta\in [a,b]$
such that $\alpha<\beta$
and $|\beta-\alpha|\leq \delta$.
\end{la}
\begin{proof}
As $\Var(\mu)$ is non-atomic by Lemma~\ref{la:nonatomic},
the map $f\colon [a,b]\to\R$,
$x\mto \Var(\mu)([a,x])$ is continuous
(see Lemma~\ref{char-na})
and hence uniformly continuous. Thus,
we find $\delta>0$ such that
$|f(\beta)-f(\alpha)|<\ve$ for all $\alpha,\beta\in [a,b]$
with $|\beta-\alpha|<\delta$.
It remains to note that $|f(\beta)-f(\alpha)|=\Var(\mu)([\alpha,\beta])$ if $\alpha\leq\beta$.
\end{proof}
\begin{la}\label{lem-BV-loc-uniq}
In the situation of {\rm\ref{situ-bvde}},
assume that $\wt{f}\colon U\to (\cL(E,F),\|\cdot\|_{\op})$
is Lipschitz.
Then \eqref{diff-cara} satisfies local uniqueness of $\BV$-solutions.
\end{la}
\begin{proof}
Let $L$ be a Lipschitz constant for $\wt{f}$,
$a<b$ be real numbers, and $\mu\in \Mdna([a,b],E)$.
Let $\gamma_1\colon I_1\to U$ and $\gamma_2\colon I_2\to U$ be $\BV$-solutions of \eqref{diff-cara}
and $t_0\in I_1\cap I_2$
such that $\gamma_1(t_0)=\gamma_2(t_0)$.
Excluding a trivial case,
we may assume that $I_1\cap I_2$ is not a singleton.
Choose $\ve>0$ so small that $L\ve <\frac{1}{2}$.
Let $\delta>0$ be as in
Lemma~\ref{so-subdivide}. Then $K:=I_1\cap I_2\cap [t_0-\frac{\delta}{2},t_0+\frac{\delta}{2}]$
is a compact interval of length $\leq\delta$,
whence $\Var(\mu)(K)<\ve$
and thus $L\Var(\mu)(K)<\frac{1}{2}$.
For any $t\in K$ with $t\geq t_0$, we have
\begin{eqnarray*}
\|(\gamma_1-\gamma_2)(t)\|_F &=&  \|f_*(\mu,\gamma_1)([t_0,t])  -  f_*(\mu,\gamma_2) ([t_0,t]\|_F\\
&\leq& \|f_*(\mu,\gamma_1)|_K-f_*(\mu,\gamma_2)|_K\|\\
&=& \|f_*(\mu|_K,\gamma_1|_K)-f_*(\mu|_K,\gamma_2|_K)\|\\
&\leq & L\|\mu|_K\|\,\|\gamma_1|_K-\gamma_2|_K\|_\infty
=L\Var(\mu)(K)\|\gamma_1|_K-\gamma_2|_K\|_\infty;
\end{eqnarray*}
see\hspace*{-.3mm}
Lemma\hspace*{-.2mm}~\ref{star-Lipschitz}\hspace*{-.3mm}
for\hspace*{-.3mm} the\hspace*{-.3mm} last\hspace*{-.3mm} line.\hspace*{-.4mm}
Also,\hspace*{-.3mm}
$\|(\gamma_1\hspace*{-.2mm}\!-\!\hspace*{-.2mm}\gamma_2)(t)\|_F\!\leq\!
L\Var(\mu)(K)
\|\gamma_1|_K\hspace*{-.2mm}\!-\!\hspace*{-.2mm}\gamma_2|_K\|_\infty$\linebreak
for all $t\in K$ with $t\leq t_0$.
Taking the supremum over $t\in K$, we get
\[
\|\gamma_1|_K-\gamma_2|_K\|_\infty\leq L\Var(\mu)(K)\|\gamma_1|_K-\gamma_2|_K\|_\infty\leq \frac{1}{2}\|\gamma_1|_K-\gamma_2|_K\|_\infty.
\]
Hence $\|\gamma_1|_K-\gamma_2|_K\|_\infty=0$
and thus $\gamma_1|_K=\gamma_2|_K$.   
\end{proof}
\begin{la}\label{on-intersect}
Let the differential equation \eqref{diff-cara} satisfy local uniqueness of $\BV$-solutions. Let $\gamma_{1}\colon I_{1} \rightarrow E$ and $\gamma_{2}\colon I_{2} \rightarrow E$ be $\BV$-solutions to \eqref{diff-cara} with $\gamma_{1}\left(t_{0}\right)=\gamma_{2}\left(t_{0}\right)$ for some $t_{0} \in I_{1} \cap I_{2}$. Then
\[
\gamma_1 |_{I_1 \cap I_2} =\gamma_2|_{I_1 \cap I_2}.
\]
\end{la}
\begin{proof}
%Along the lines of \cite[Lemma 2.4.6]{GaN}, we
Let $A := \{t \in I_1 \cap I_2 \colon  \gamma_1(t) = \gamma_2(t)\}$. 
Since $\gamma_1$ and $\gamma_2$ are continuous and $F$ is Hausdorff, $A$ is closed in $I_1\cap I_2$.
By local uniqueness of $\BV$-solutions,
$A$ is also open in $I_1 \cap I_2$. By hypothesis, $A \neq \emptyset$.  Since $I_1 \cap I_2$ is an interval and hence connected, it follows that $A = I_1 \cap I_2$ and thus
$\gamma_1 |_{I_1 \cap I_2} =\gamma_2|_{I_1 \cap I_2}$.
\end{proof}
\begin{la}\label{twopiece}
Let $\gamma_1\colon[t_0,t_1]\to E$ and
$\gamma_2\colon [t_1,t_2]\to E $ be $\BV$-solutions of \eqref{diff-cara} and $\gamma_1(t_1)=\gamma_2(t_1)$. Then also
\[
\gamma:[t_0,t_2]\to E;\quad t\mto \left\{\begin{aligned}
\gamma_1(t)  &~\textit{if}~ t\in[t_0,t_1] \\
\gamma_2(t)  &~\textit{if}~ t\in[t_1,t_2]
\end{aligned}\right.
\]
is a solution of \eqref{diff-cara}.
\end{la}
\begin{proof}
This is immediate from Lemma~\ref{into-prod}.
\end{proof}
\begin{defn}
We say that the differential equation \eqref{diff-cara} satisfies \emph{local existence of $\BV$-solutions} if for any $t_0 \in [a, b]$ and $y_0 \in U$, there exists a $\BV$-solution $\gamma: I \rightarrow U\subseteq F$ for the initial value problem \eqref{bv-initi},
defined on a compact interval~$I$
which is a $t_0$-neighborhood in $[a,b]$.
\end{defn}
\noindent
We recall a result concerning parameter-dependence
of fixed points
(see Theorem~4.7 and Proposition~4.3\,(a)
in \cite{FXX}).
\begin{numba}\label{pardep-fp}
\emph{Let $E$ be a locally convex space,
$(F,\|\cdot\|_F)$
be a Banach space and $P\sub E$ as well
as $U \sub F$ be open subsets.
Let $k\in \N_0\cup\{\infty\}$
and $f \colon P \times U \to F$
be a $C^k$-map such that for some $\theta\in [0,1[$,
all of the mappings
$f_p := f(p, \cdot)\colon  U \to U$
are Lipschitz with Lipschitz constant~$\theta$.
Let $P_0$ be the set of all
$p\in P$ such that $f_p$ has a fixed point
$x_p$ $($which is necessarily unique$)$.
Then $P_0$ is open in~$P$ and the map
$P_0\to U$, $p\mto x_p$ is $C^k$.}
\end{numba}
\noindent
The next result guarantees local existence and local uniqueness of $\BV$-solutions, and addresses dependence
on parameters.
\begin{prop}\label{carath-ex}
Let $(E,\|\cdot\|_E)$
and $(F,\|\cdot\|_F)$
be Banach spaces, $y_0\in F$,
$R>0$ and
$f\colon E\times B^F_R(y_0)\to F$
be a map such that $\wt{f}(y):=f(\cdot,y)\colon E\to F$
is continuous linear for each $y\in B^E_R(y_0)$
and $\wt{f}\colon B^F_R(y_0)\to (\cL(E,F),\|\cdot\|_{\op})$ is Lipschitz. Let
$L\in [0,\infty[$ be a Lipschitz constant for $\wt{f}$.
Then
we have\emph{:}
%the following holds.
\begin{itemize}
\item[\rm(a)]
For all real numbers $a<b$ and each $\mu\in \Mdna([a,b],E)$,
the differential equation $y'=f_*(\mu,y)$
satisfies local uniqueness of $\BV$-solutions.
\item[\rm(b)]
For all $a<b$ and each $\mu\in \Mdna([a,b],E)$
such that $L\hspace*{.2mm}\|\mu\|<1$
holds for the variation norm and
\[
\|\wt{f}\|_\infty \|\mu\| \;< \; R
\]
with $\|\wt{f}\|_\infty:=
\sup\left\{\|\wt{f}(y)\|_{\op}\colon y\in
B^F_R(y_0)\right\}$,
we get that, for each $t_0\in [a,b]$,
the initial value problem
\begin{equation}\label{inivapro}
y'=f_*(\mu,y),\qquad
y(t_0)=y_0
\end{equation}
has a $\BV$-solution $\eta_\mu\colon [a,b]\to B^F_R(y_0)$
defined on all of $[a,b]$.
\item[\rm(c)]
The set $Q:=\{\mu\in \Mdna([a,b],E)\colon \|\wt{f}\|_\infty\|\mu\|<R\mbox{ and $L\|\mu\|<1$}\}$ is a convex
open $0$-neighborhood in $\Mdna([a,b],E)$.
If $\wt{f}$ is, moreover, $C^k$
with $k\in \N_0\cup\{\infty\}$
then also the following map is~$C^k$:
\begin{equation}\label{the-pardep}
Q\to \BV([a,b],E),\quad \mu\mto \eta_\mu .
\end{equation}
\end{itemize}
\end{prop}
\noindent
Note that $\|\wt{f}\|_\infty<\infty$ as $\wt{f}$ is Lipschitz
and $B^F_R(y_0)$ is bounded.

\noindent
\begin{proof}
(a) This is a special case of
Lemma~\ref{lem-BV-loc-uniq}.\smallskip

(b) Given $\mu$ as described in~(b),
there exists $r\in \;]0,R[$
such that $\|\wt{f}\|_\infty\|\mu\|\leq r$.
Write $y_0$ also for the constant function
$[a,b]\to F$, $t\mto y_0$.
Then $A:=\{\eta\in C([a,b],F)\colon
\|\eta-y_0\|_\infty\leq r\}$
is a closed subset
of the Banach space $C([a,b],F)$. For each $\eta\in A$, we have $\eta\in C([a,b],B_R^F(0)\}$,
whence $f_*(\mu,\eta)\in \Mdna([a,b],F)$ and
\begin{equation}\label{fix-is-BV}
\Phi(\eta)\colon [a,b]\to F,\quad
t\mto \left\{\begin{array}{cl}
y_0+f_*(\mu,\eta)([t_0,t])&\mbox{\,if $\,t\geq t_0$;}\\
y_0-f_*(\mu,\eta)([t,t_0])&\mbox{\,if $\,t\leq t_0$}
\end{array}\right.
\end{equation}
is in $\BV([a,b],F)$, by Lemma~\ref{bv-basepoint}.
Using variation norms,
%we
get
$\|\Phi(\eta)(t)-y_0\|_F\leq \|f_*(\mu,\eta)\|
\leq \|\wt{f}\|_\infty\|\ve\|_{\op}\|\mu\|\leq r$
due to (\ref{size-dense}) (where $\ve$ is the bilinear evaluation map
$\cL(E,F)\times E\to F$ with $\|\ve\|_{\op}\leq 1$).
Hence $\Phi(\eta)\in A$ for all $\eta\in A$.
For all $\eta,\zeta\in A$, we have for all $t\in [t_0,b]$
\begin{eqnarray*}
\|\Phi(\eta)(t)-\Phi(\zeta)(t)\|_F
&=&\|(f_*(\mu,\eta)-f_*(\mu,\zeta))([t_0,t])\|_F\\
&\leq & \Var(f_*(\mu,\eta)-f_*(\mu,\zeta))([t_0,t])\\
&\leq & \|f_*(\mu,\eta)-f_*(\mu,\zeta)\|
\leq L\|\mu\|\, \|\eta-\zeta\|_\infty,
\end{eqnarray*}
using Lemma~\ref{star-Lipschitz}.
A similar calculation establishes
$\|\Phi(\eta)(t)-\Phi(\zeta)(t)\|_F
\leq L\|\mu\|\, \|\eta-\zeta\|_\infty$
for all $t\in [a,t_0]$. Taking the supremum
over $t\in [a,b]$, we get
\[
\|\Phi(\eta)-\Phi(\zeta)\|_\infty\leq L\|\mu\|\, \|\eta-\zeta\|_\infty.
\]
Thus $\Phi\colon A\to A$ is a contraction,
as $L\|\mu\|<1$.
By Banach's Fixed Point Theorem,
$\Phi$ has a unique fixed point $\eta$.
Then $\eta=\Phi(\eta)\in \BV([a,b],F)$.
As $\eta(t)$ is given by the
right-hand side of~(\ref{fix-is-BV}),
$\eta$ solves the initial value
problem~(\ref{inivapro}).\smallskip

(c) Fix $a<b$. For all $\theta\in \;]0,1[$,
\[
Q_\theta:=\{\mu\in \Mdna([a,b],E)\colon L\|\mu\| <\theta\}
\]
is an open $0$-neighborhood in $\Mdna([a,b],E)$.
As $Q$ is the union of the sets
$Q_\theta$, it suffices to prove the conclusion
of~(c) with $Q_\theta$ in place of~$Q$.
For
$\mu\in Q_\theta$ and $\eta\in C([a,b],
B^F_R(y_0)$, we write
$\Psi_\theta(\mu,\eta)\colon [a,b]\to F$
for the function denoted by $\Phi(\eta)$ in~(\ref{fix-is-BV}).
As above, we see that $\Psi_\theta(\mu,\eta)\in \BV([a,b],F)$
and that
\[
\|\Psi_\theta(\mu,\eta)-\Psi_\theta(\mu,\zeta)\|_\infty\leq
L\|\mu\|\,\|\eta-\zeta\|_\infty\leq \theta\|\eta-\zeta\|_\infty
\]
for all $\mu\in Q_\theta$ and $\eta,\zeta\in C([a,b], B^F_R(y_0))$. As a consequence, the mapping\linebreak
$\Psi_\theta(\mu,\cdot)\colon
C([a,b],B^F_R(y_0))\to C([a,b],F)$
is Lipschitz with Lipschitz constant~$\theta$,
for each $\mu\in Q_\theta$.
By the proof of~(b), the map
has a fixed point $\eta_\mu$,
which is unique by~\cite[Proposition~4.3\,(a)]{FXX}.
We claim that the map
\[
\Psi_\theta\colon Q_\theta\times C([a,b],B^F_R(y_0))\to
\BV([a,b],F)
\]
is $C^k$. Then $\Psi_\theta$ is also $C^k$
as a map to $C([a,b],F)$, using that the inclusion map
$\BV([a,b],F)\to C([a,b],F)$ is continuous
linear (see Remark~\ref{rem:cont:inclusion}). Hence $h_\theta\colon Q_\theta\to C([a,b],F)$,
$\mu\mto\eta_\mu$
is a $C^k$-map, by~\ref{pardep-fp}.
As a consequence,
also
\[
Q_\theta\to \BV([a,b],F),\quad \mu\mto \eta_\mu=
\Psi_\theta(\mu,\eta_\mu)=\Psi_\theta(\mu,h_\theta(\mu))
\]
is a $C^k$-map (as required).
To prove the claim, we use that the linear map
\[
J\colon \Mdna([a,b],F)\to\BV([a,b],F),\quad
\mu\mto(t\mto \mu([a,t]))
\]
is continuous: $J(\mu)(a)=0$
and $J(\mu)'=\mu$ for
$\mu\in \Mdna([a,b],F)$
yields $\|J(\mu)\|_{BV}=\|J(\mu)\|_F+\|J(\mu)'\|=\|\mu\|$.
For all $(\mu,\eta)\in Q_\theta\times C([a,b], B^F_R(y_0))$,
\[
\Psi_\theta(\mu,\eta)(t)=y_0+f_*(\mu,\eta)([a,t])-f_*(\mu,\eta)([a,t_0])
\]
holds for all $t\in [a,b]$, whence
\[
\Psi_\theta(\mu,\eta)=y_0+J(f_*(\mu,\eta))-f_*(\mu,\eta)(S)
\]
with $S:=[a,t_0]$. Since $f_*$
is $C^k$ by Lemma~\ref{diff-pfwd}
and both the evaluation at~$S$
(see (\ref{bounded-op})) and~$J$ are continuous linear,
we deduce that $\Psi_\theta$ is $C^k$
as a map to $\BV([a,b],F)$.
\end{proof}
\begin{cor}\label{cor-loc-ex}
Let $(E,\|\cdot\|_E)$ and $(F,\|\cdot\|_F)$
be Banach spaces, $U\sub F$ be an open subset and
$f\colon E\times U\to F$ be a $C^2$-map
such that $\wt{f}(y):=f(\cdot,y)\colon E\to F$
is linear for each $y\in U$. Then
$y'=f_*(\mu,y)$ satisfies local existence
of $\BV$-solutions, for all real numbers
$a<b$ and $\mu\in \Mdna([a,b],E)$.
\end{cor}
\begin{proof}
Given $a$, $b$, $\mu$,
$t_0\in [a,b]$ and $y_0\in U$,
there exists $R>0$ such that $B^F_R(y_0)\sub U$
and $\wt{f}|_{B^F_R(y_0)}$
is Lipschitz with some constant~$L$
(see Remark~\ref{from-axioms}\,(a)).
There is $\ve>0$ such that $\|\wt{f}|_{B^F_R(y_0)}\|_\infty\ve<R$ and $L\ve<1$. By Lemma~\ref{so-subdivide},
there is $\delta>0$ such that $\Var(\mu)(K)<\ve$
for each compact subinterval $K\sub [a,b]$
of length $\leq 2\delta$.
Then $K:=[t_0-\delta,t_0+\delta]\cap [a,b]$
is a compact subinterval of $[a,b]$
of length $\leq 2\delta$
and a $t_0$-neighborhood in $[a,b]$.
By Proposition~\ref{carath-ex}\,(b),
the initial value problem $y'=f_*(\mu,y)$, $y(t_0)=y_0$
has a $\BV$ solution defined on~$K$.
\end{proof}
\begin{numba}
If $M$ and $N$ are Banach manifolds
of class $\text{FC}^1$,
$\phi\colon M\to N$ is an $\text{FC}^1$-map
and $X\colon M\to TM$ and $Y\colon N\to TN$
are continuous vector fields, we say that $X$
is \emph{$\phi$-related} to $Y$ if $Y\circ\phi=T\phi\circ X$.
\end{numba}
\begin{numba}
If also $L$ is an $\text{FC}^1$-Banach manifold,
$\psi\colon N\to L$ an $\text{FC}^1$-map
and $Z\colon L\to TL$ a continuous vector field
such that $Y$ is $\psi$-related to~$Z$,
then $X$ is $(\psi\circ\phi)$-related to~$Z$,
as $Z\circ \psi\circ\phi=T\psi\circ Y\circ \phi=
T\psi\circ T\phi\circ X=T(\psi\circ\phi)\circ X$.
\end{numba}
\begin{numba}
Consider an $FC^1$-map $\phi\colon U\to V$
between
open sets $U\sub E$ and $V\sub F$ in Banach spaces.
Consider continuous vector fields
$X\colon U\to TU=U\times E$ and $Y\colon V\to TV=V\times F$;
thus $X=(\id_U,\xi_X)$ and $Y=(\id_V,\xi_Y)$
with continuous maps $\xi_X\colon U\to E$
and $\xi_Y\colon V\to F$.
Then $X$ is $\phi$-related to~$Y$ if and~only~if
\[
\xi_Y\circ \phi=(d\phi)\circ (\id_U,\xi_X),
\]
in which case we shall say that
$\xi_X$ is \emph{$\phi$-related} to $\xi_Y$.
\end{numba}
\begin{numba}\label{use-related}
Let $f\colon E\times U\to F$
be as in~\ref{situ-bvde}. Let $Y$ be a Banach space,
$V\sub Y$ be an open subset and
$g\colon E\times V\to Y$ be a map
such that $\wt{g}(z):=g(\cdot,z)\colon E\to Y$
is continuous linear for all $z\in V$ and
$\wt{g}\colon V\to(\cL(E,Y),\|\cdot\|_{\op})$
is continuous.
Let $\tau\colon U\to V$ be an $FC^1$-map
and assume that
\begin{equation}\label{global-related}
g(v,\tau(x))=d\tau(x,f(v,x))\quad\mbox{for all $\,v\in E$
and $x\in U$.}
\end{equation}
Thus $f(v,\cdot)$ is $\tau$-related to $g(v,\cdot)$,
for each $v\in E$.
Let $a<b$ be real numbers, $\mu\in \Mdna([a,b],E)$
and $\gamma\colon [\alpha,\beta]\to U$
be a $\BV$-function, with
$a\leq \alpha<\beta\leq b$.\smallskip

\noindent
Then we have:
\emph{If $\gamma$ solves the $\BV$-differential equation}
\[
y '=f_*(\mu,y),
\]
\emph{then $\tau\circ\gamma$ solves the $\BV$-differential
equation}
\[
y'=g_*(\mu,y).
\]
In fact,
if we write $\gamma'=\rho\, {\rm d}\nu$
with $\nu\in \Mes([\alpha,\beta])_+$
and $\rho\in \cL^1(\nu,F)$, then
$\tau\circ\gamma\in \BV([a,b],V)$
and
\begin{eqnarray*}
(\tau\circ \gamma)' &=& (d\tau)_*(\gamma,\gamma')=
(d\tau)_*(\gamma,f_*(\mu,\gamma))\\
&=& (d\tau)_*(\gamma,f_*(\rho\, {\rm d}\nu,\gamma))
=(d\tau)_*(\gamma,(f\circ (\rho,\gamma))\, {\rm d}\nu))\\
&= & (d\tau)\circ (\gamma,f\circ(\gamma,\rho))\, {\rm d}\nu
=(g\circ (\rho,\tau\circ\gamma))\,{\rm d}\nu
=g_*(\mu,\tau\circ\gamma),
\end{eqnarray*}
by Remarks~\ref{rem-explicit-1} and~\ref{rem-explicit-2}.
\end{numba}
\subsection*{BV-differential equations on manifolds}
\noindent
So far, we studied $\BV$-differential
equations on open subsets of Banach spaces.
We now extend the approach to
differential equations on Banach manifolds.
\begin{numba}\label{situ-Mx}
Let $M$ be a Banach manifold modeled
on a Banach space~$F$.
Let $E$ be a Banach space
and
\[
f\colon E\times M\to TM
\]
be a continuous map such that $f(v,p)\in T_pM$
for all $(v,p)\in E\times M$,
the map $f(\cdot,p)\colon E\to T_pM$
is linear for each $p\in M$, 
and such that, for each chart
$\phi\colon U_\phi\to V_\phi\sub F$ of~$M$,
the map
\begin{equation}
f_\phi:=d\phi\circ f\circ (\id_E\times \phi^{-1})\colon
E\times V_\phi\to F,\;\;
(v,x)\mto d\phi (f(v,\phi^{-1}(x)))
\end{equation}
has the property that
\[
\wt{f}_\phi\colon V_\phi\to(\cL(E,F),\|\cdot\|_{\op}),
\quad
x\mto f_\phi(\cdot,x)
\]
is continuous.
\end{numba}
\begin{numba}\label{situ-m-2}
Let $a<b$ be in $\R$; let
$\mu \in \Mdna([a,b],E)$
and $\alpha,\beta\in [a,b]$ with $\alpha<\beta$.
We call a $\BV$-function $\gamma\colon [\alpha,\beta]\to M$
a \emph{solution to the $\BV$-differential equation}
\begin{equation}\label{undef}
\dot{y}=f_*(\mu,y)
\end{equation}
if $\phi\circ \gamma|_K$
solves the $\BV$-differential equation
\[
y'=(f_\phi)_*(\mu,y)
\]
for each chart $\phi\colon U_\phi\to  V_\phi\sub F$
of~$M$ and each non-degenerate
compact subinterval $K\sub [\alpha,\beta]$
such that $\gamma(K)\sub U_\phi$.
If $t_0\in [a,b]$ and $y_0\in M$, we say that $\gamma$
\emph{solves the $\BV$-initial value problem}
\begin{equation}\label{undef-2}
\dot{y}=f_*(\mu,y),\quad y(t_0)=y_0
\end{equation}
if $t_0\in [\alpha,\beta]$,
$\gamma(t_0)=y_0$ and
$\gamma$ solves (\ref{undef}).
\end{numba}
\begin{rem}
Note that neither do we define
$\dot{y}$ nor $f_*(\mu,y)$.
\end{rem}
\begin{la}\label{chara-sol-mfd}
In the situation of~{\rm\ref{situ-m-2}},
the following conditions are equivalent
for $\alpha,\beta\in [a,b]$
with $\alpha<\beta$
and a function
$\gamma\colon [\alpha,\beta]\to M$:
\begin{itemize}
\item[\rm(a)] $\gamma$ is $\BV$-function
which solves the $\BV$-differential equation
{\rm(\ref{undef})}.
\item[\rm(b)] There exists a
subdivision $a=t_0<t_1<\cdots <t_n=b$ of $[a,b]$
and charts $\phi_j\colon U_j\to V_j\sub F$
of~$M$ with $\gamma([t_{j-1},t_j])\subseteq U_j$
for $j\in \{1,\ldots, n\}$,
such that $\phi_j\circ \gamma|_{[t_{j-1},t_j]}$
is a $\BV$-solution of the $\BV$-differential
equation
\[
y'=(f_\phi)_*(\mu,y).
\]
\end{itemize}
\end{la}
\begin{proof}
(a)$\Rightarrow$(b) holds by definition.
For the reverse implication, let $t_0,\ldots, t_n$ and charts $\phi_j$ be as in (b). Let $\phi\colon U_\phi\to V_\phi\sub F$
be a chart for~$M$
and $K\sub [\alpha,\beta]$ be a non-degenerate
compact subinterval such that $\gamma(K)\sub U_\phi$.
Write $K=[s,t]$.
After replacing $\{t_0,\ldots,t_n\}$
with $\{t_0,\ldots,t_n\}\cup\{s,t\}$
and duplicating charts, we may assume that
$s=t_k$ and $t=t_\ell$
for certain $k<\ell$ in $\{1,\ldots,n\}$.
Then $t_k<\cdots< t_\ell$ is a subdivision
of $[t,s]$; by Lemma~\ref{twopiece},
we only need to show that
$(\phi\circ \gamma)|_{[t_{j-1},t_j]}$
solves $y'=(f_\phi)_*(\mu,y)$ for all
$j\in \{k+1,\ldots,\ell\}$.
We may therefore assume that $n=1$
and $K=[a,b]$. Then $\gamma([a,b])\sub U_1\cap U_\phi$.
After replacing the chart domains with their
intersection, we may assume
%that
$U:=U_1=U_\phi$. Then
\[
\tau:=\phi\circ (\phi_1)^{-1}\colon V_1\to V_\phi
\]
is a $C^\infty$-diffeomorphism.
By definition,
$f(v,\cdot)|_U$
is $\phi_1$-related to
$(\id_{V_1},f_{\phi_1}(v,\cdot))$.
By construction,
$(\id_{V_1},f_{\phi_1}(v,\cdot))$
is $\tau$-related to
\[
T\tau \circ (\id_{V_1},f_{\phi_1}(v,\cdot))
\circ\tau^{-1}=(\id_{V_\phi},f_\phi(v,\cdot)).
\]
Since $\phi_1\circ \gamma$ solves $y'=(f_{\phi_1})_*(\mu,y)$,
this implies that
$\phi\circ \gamma=\tau\circ (\phi_1\circ\gamma)$
solves $y'=(f_\phi)_*(\mu,y)$
(see \ref{use-related}).
\end{proof}
\begin{defn}
In the situation of \ref{situ-m-2},
we say that the $\BV$-differential equation
(\ref{undef}) satisfies
\emph{local uniqueness of $\BV$-solutions} if
for all $\BV$-solutions $\gamma_1\colon I_1\to M$
and $\gamma_2\colon I_2 \to M$ of \eqref{undef}
and $t_0 \in I_1 \cap I_2$ with $\gamma_1(t_0)
=\gamma_2(t_0)$, there exists a $t_0$-neighborhood
$K$ in $I_1\cap I_2$ such that
\[
\gamma_1|_K=\gamma_2|_K .
\]
\end{defn}
The following lemma can be proved like
Lemma~\ref{on-intersect},
replacing $F$ with $M$.
\begin{la}\label{on-intersect-2}
Let the differential equation \eqref{undef}
satisfy local uniqueness of $\BV$-solutions. Let $\gamma_1\colon I_1\to M$ and $\gamma_2\colon I_2 \to M$ be $\BV$-solutions to \eqref{undef} with
$\gamma_1(t_0)=\gamma_2(t_0)$
for some $t_0 \in I_1 \cap I_2$. Then
$\gamma_1 |_{I_1 \cap I_2} =\gamma_2|_{I_1 \cap I_2}$.
$\,\square$
\end{la}
\begin{prop}\label{main-global}
Let $(E,\|\cdot\|_E)$ be a Banach space,
$M$ be a manifold modeled on a Banach space~$F$
and $f\colon E\times M\to TM$
be a $C^2$-map
such that $f(v,p)\in T_pM$
for all $(v,p)\in E\times M$
and $f(\cdot,p)\colon E\to T_pM$
is linear for all $p\in M$.
Then the following holds.
\begin{itemize}
\item[\rm(a)]
For all real numbers $a<b$ and each $\mu\in \Mdna([a,b],E)$,
the differential equation $\dot{y}=f_*(\mu,y)$
satisfies local uniqueness of $\BV$-solutions.
\item[\rm(b)]
For each $y_0\in M$,
there exists $\ve>0$
such that, for all $a<b$ and each $\mu\in \Mdna([a,b],E)$
with variation norm $\|\mu\|<\ve$,
and each $t_0\in [a,b]$,
the $\BV$-initial value problem
\begin{equation}\label{inivapro-2}
\dot{y}=f_*(\mu,y),\qquad
y(t_0)=y_0
\end{equation}
has a $\BV$-solution $\eta_\mu\colon [a,b]\to M$
defined on all of $[a,b]$.
\item[\rm(c)]
For $\ve$ as in {\rm(b)} and all $a<b$ in $\R$,
the set $\{\mu\in \Mdna([a,b],E)\colon \|\mu\|<\ve\}$
$=: Q$
is open in $\Mdna([a,b],E)$.
If $f$ is, moreover, $C^{k+1}$ with $k\in \N_0\cup\{\infty\}$,
then the following map is $C^k$:
\[
Q\to\BV([a,b],M),\quad \mu\mto\eta_\mu .
\]
\end{itemize}
\end{prop}
\begin{proof}
(a) Since $f$ is $C^2$, also all of the maps
$f_\phi\colon E\times V_\phi\to F$
are $C^2$,
whence $\wt{f}_\phi
\colon V_\phi\to (\cL(E,F),\|\cdot\|_{\op})$
is locally Lipschitz, by Remark~\ref{from-axioms}\,(a).
As a consequence,
$y'=(f_\phi)_*(\mu,y)$
satisfies local uniqueness of $\BV$-solutions
for each chart~$\phi$
(cf.\ Lemma~\ref{lem-BV-loc-uniq}).
As local uniqueness can be checked in charts,
we deduce that $\dot{y}=f_*(\mu,y)$
satisfies local uniqueness of $\BV$-solutions.\smallskip

(b) and (c): Let $\phi\colon U_\phi\to V_\phi\sub F$ be a chart
of~$M$ such that $y_0\in U_\phi$.
After shrinking the domain,
we may assume that $V_\phi$
is a ball $B^F_R(x_0)$ for some $x_0\in F$
and $\wt{f}_\phi\colon V_\phi\to (\cL(E,F),
\|\cdot\|_{\op})$
is Lipschitz. Let $L$ be a Lipschitz
constant for $\wt{f}_\phi$.
Choose $\ve>0$ so small that $\ve L<1$
and $\ve \|\wt{f}_\phi\|_\infty<R$.
Then
\[
y'=(f_\phi)_*(\mu,y),\quad y(t_0)=x_0
\]
has a solution $\eta_\mu\colon [a,b]\to B^F_R(x_0)$
for all $a<b$ in $\R$, all
$\mu \in \Mdna([a,b],F)$ with variation norm
$\|\mu\|<\ve$, and all $t_0\in [a,b]$.
Moreover, the set $Q:=\{\mu\in \Mdna([a,b],F)\colon \|\mu\|<\ve\}$ is open in $\Mdna([a,b],F)$
and the map $h\colon Q\to \BV([a,b],B_R^F(x_0))$,
$\mu\mto \eta_\mu$ is $C^k$, by Proposition~\ref{carath-ex}
(b) and (c). We are using here that $\wt{f}_\phi$
is $C^k$ as $f_\phi$ is $C^{k+1}$
(see \cite[Proposition~1\,(b)]{Asp}).
Then $\phi^{-1}\circ\eta_\mu$
is a $\BV$-solution for $\dot{y}=f_*(\mu,y)$,
$y(t_0)=y_0$
(cf.\ Lemma~\ref{chara-sol-mfd}).
Since
\[
(\phi^{-1})_*\colon \BV([a,b],B^F_R(x_0))\to \BV([a,b],M),\;\;
\gamma\mto\phi^{-1}\circ\gamma
\]
is smooth (see Proposition~\ref{prop:smooth_pushfwd})
and $h$ is $C^k$, we deduce that
the composition
$(\phi^{-1})_*\circ h\colon
Q\to \BV([a,b],M)$, $\mto \phi^{-1}\circ \eta_\mu$
is~$C^k$.
\end{proof}
\noindent
Applying local charts,
the following assertion is immediate from
Lemma~\ref{chara-sol-mfd}
and Corollary~\ref{cor-loc-ex}.
\begin{la}
Let $M$ be a Banach manifold, $E$ be a Banach space
and $f\colon E\times M\to TM$
be a $C^2$-map such that $f(v,p)\in T_pM$
for all $(v,p)\in E\times M$
and $f(\cdot,p)\colon E\to T_pM$
is linear for each $p\in M$.
Then the differential equation $\dot{y}=f_*(\mu,y)$
satisfies local existence
of $\BV$-solution for all real numbers
$a<b$ and each $\mu\in \Mdna([a,b],E)$,
in the following sense: For each $t_0\in [a,b]$
and each $y_0\in M$,
there exists a compact subinterval $K\sub [a,b]$
which is a $t_0$-neighborhood in $[a,b]$
and a $\BV$-solution $\eta\colon K\to M$
for the initial value problem
$\dot{y}=f_*(\mu,y)$, $y(t_0)=y_0$. $\,\square$
\end{la}
\begin{la}\label{mfd-related-sol}
Let $f\colon E\times M\to TM$
be as in \emph{\ref{situ-Mx}} and $f_1\colon E\times N\to TN$
be an analogous map. Assume that
$\psi\colon M\to N$ is an $\text{FC}^1$-map
such that $f(v,\cdot)$ is $\psi$-related to $f_1(v,\cdot)$
for all $v\in E$.
Then $\psi\circ\eta$ is a $\BV$-solution
to $\dot{y}=(f_1)_*(\mu,y)$
for all $a<b$, each $\mu\in \Mdna([a,b],E)$
and each $\BV$-solution $\eta\colon [\alpha,\beta]\to M$ to $\dot{y}=f_*(\mu,y)$ with $\alpha<\beta$ in $[a,b]$.
\end{la}
\begin{proof}
Given $t_0\in [\alpha,\beta]$,
there exist charts $\kappa\colon U_\kappa\to V_\kappa$ of~$N$
and $\phi\colon U_\phi\to V_\phi$ of~$M$
such that $\psi(U_\phi)\sub U_\kappa$
and $\eta(t_0)\in U_\phi$. There exists a compact
subinterval $K\sub [\alpha,\beta]$
which is a $t_0$-neighborhood in $[\alpha,\beta]$
such that $\eta(K)\sub U_\phi$.
Now $f(v,\cdot)$ is $\phi$-related to
$f_\phi(v,\cdot)$ and
$f_1(v,\cdot)$ is $\kappa$-related to $(f_1)_\psi(v,\cdot)$
for all $v\in E$,
whence $f_\phi(v,\cdot)$ is
$\kappa\circ\psi\circ \phi^{-1}$-related
to $(f_1)_\kappa(v,\cdot)$.
Since $\phi\circ \eta|_K$ solves $y'=(f_\phi)_*(\mu,y)$,
\ref{use-related}
shows that $\kappa\circ \psi\circ\eta|_K=
(\kappa\circ\psi\circ \phi^{-1})\circ (\phi\circ\eta|_K)$
solves $y'=((f_1)_\kappa)_*(\mu,y)$.
Hence $\psi\circ\eta$ solves $\dot{y}=(f_1)_*(\mu,y)$.
\end{proof}
\begin{rem}
In the situation of Definition~\ref{defn-sol},
given a non-degenerate subinterval $I\sub [a,b]$
one can call a function $\gamma\colon I\to U$
a solution to the $\BV$-differential equation
$y'=f_*(\mu,y)$
if $\gamma|_K$ is a $\BV$-solution
for each non-degenerate compact subinterval
$K\sub I$. Likewise, in \ref{situ-m-2},
a function $\gamma\colon I\to M$
can be called a solution to $\dot{y}=f_*(\mu,y)$
if $\gamma|_K$
is a $\BV$-solution
for each non-degenerate compact subinterval
$K\sub I$. Such solutions are locally $\BV$.
In this work,
the added generality is not needed;
we only consider
$\BV$-solutions on compact intervals.
\end{rem}
\section[Banach--Lie groups are vector measure regular]{\!\!\!Banach-Lie groups are vector measure regular}\label{sec-vm-reg}

We now prove that each Banach--Lie group
is \emph{vector measure regular} in the sense
that the conclusions of Theorem~\ref{Banach-Lie-VM-reg}
are satisfied.
\begin{numba}
Throughout the section,
let $G$ be a Banach--Lie group with neutral element~$e$
and $\cg:=\Lf(G):=T_eG$. As explained in the introduction,
the map
\begin{equation}\label{liegp-f}
f\colon \cg\times G\to TG,\;\;  (v,g)\mto v.g:=T R_g(v)
\end{equation}
is smooth. We have $f(v,g)\in T_gG$ for all
$(v,g)\in \cg\times G$ and the map $f(\cdot,g)=T_eR_g\colon \cg\to T_gG$ is linear for all $g\in G$.
Thus, the $\BV$-initial value problem
\begin{equation}\label{evol-ini}
\dot{y}=f_*(\mu,y),\quad y(a)=e
\end{equation}
can be posed on $K:=[a,b]$ for all real numbers $a<b$,
using $\mu\in \Mdna(K,\cg)$ as a parameter.
If a solution $\eta\colon [a,b]\to G$
exists, it is unique by Proposition~\ref{main-global}\,(a) and Lemma~\ref{on-intersect-2}; we then
call $\eta$ the \emph{evolution} of~$\mu$
and write
\[
\Evol_K(\mu):=\eta.
\]
We let $\cD_K$ be the set of all
$\mu\in \Mdna(K,\cg)$
for which an evolution exist;
thus, we have an evolution map
\[
\Evol_K\colon \cD_K\to \BV(K,G).
\]
If $K=[0,1]$, we write $\Evol(\eta)$ in place of
$\Evol_K(\eta)$ and obtain an evolution map
on $\cD:=\cD_{[0,1]}$.
\end{numba}
\begin{la}\label{product-ints}
For each $g\in G$,
the vector field $f(v,\cdot)\colon G\to TG$
is $R_g$-related to itself
for all $v\in \cg$.
Hence, if $a<b$ are real numbers, $K:=[a,b]$,
$\mu\in \Mdna(K,\cg)$
and $\eta\colon [\alpha,\beta]\to G$
is a $\BV$-solution to $\dot{y}=f_*(\mu,y)$
with $\alpha<\beta$ in~$K$,
then also $R_g\circ \eta\colon [\alpha,\beta]\to G$,
$t\mto \eta(t)g$
is a $\BV$-solution to $\dot{y}=f_*(\mu,y)$.
Notably, $R_g\circ \Evol_K(\mu)$
is the unique $\BV$-solution to $\dot{y}=f_*(\mu,y)$, $y(a)=g$
defined on~$K$, for each $\mu\in \cD_K$.
\end{la}
\begin{proof}
For all $g,h\in G$, we have $R_g\circ R_h=R_{hg}$,
whence $TR_g f(v,h)=TR_g(v.h)=TR_gTR_h(v)=T(R_g\circ R_h)(v)=
TR_{hg}(v)=v.hg=f(v,hg)=f(v,R_g(h))$
for all $v\in \cg$. Thus $f(v,\cdot)$ is $R_g$-related
to $f(v,\cdot)$ for each $v$, whence
Lemma~\ref{mfd-related-sol} applies. Since $f$ is smooth,
$\dot{y}=f_*(\mu,y)$ satisfies local uniqueness
of solutions by Proposition~\ref{main-global}\,(a).
\end{proof}
\noindent
{\bf Proof of Theorem~\ref{Banach-Lie-VM-reg}.}
By Proposition~\ref{main-global} (b) and (c), there exists $\ve>0$
such that, for all real numbers $a<b$,
the evolution $\Evol_{[a,b]}(\mu)$
exists for all $\mu\in \Mdna([a,b],\cg)$
such that $\|\mu\|<\ve$, and the map
\[
\{\mu\in \Mdna([a,b],\cg)\colon \|\mu\|<\ve\}\to
\BV([a,b],G),\quad \mu\mto \Evol_{[a,b]}(\mu)
\]
is smooth.
Now consider an arbitrary
$\mu\in \Mdna([0,1],\cg)$. By Lemma~\ref{so-subdivide},
there exists $\delta>0$ such that $\Var(\mu)([\alpha,\beta])<\frac{\ve}{2}$ for all $\alpha<\beta$ in $[0,1]$
with $\beta-\alpha\leq \delta$.
Pick a subdivision $0=t_0<t_1<\cdots<t_n=1$
of $[0,1]$ such that $t_j-t_{j-1}\leq \delta$
for all $j\in \{1,\ldots, n\}$.
Then
\[
U:=\{\nu\in \Mdna([0,1],\cg)\colon \|\nu-\mu\|<\ve/2\}
\]
is an open $\mu$-neighborhood in $\Mdna([0,1],\cg)$.
Given $\nu\in U$, we have $\nu_j:=\nu|_{[t_{j-1},t_j]}\in
\Mdna(K_j,\cg)$ with $K_j:=[t_{j-1},t_j]$
for $j\in \{1,\ldots, n\}$.
Moreover, $\nu_j$ is contained in the open
$\ve$-ball $U_j\sub \Mdna(K_j,\cg)$
around~$0$;
the map
\[
U\to U_j,\quad \nu\mto \nu_j
\]
is smooth, being a restriction of the map $r_j\colon \Mdna([0,1],\cg)\to\Mdna(K_j,\cg)$,
$m\mto m|_{K_j}$ which is continuous and linear.
For $j\in \{1,\ldots,n\}$,
we already have a smooth evolution map
\[
\Evol_{K_j}|_{U_j}\colon U_j\to \BV(K_j,G).
\]
As point evaluations are smooth by Corollary~\ref{cor:smooth_eval}, also
\[
\evol_{K_j}\colon U_j\to G,\quad m \mto\Evol_{K_j}(m)(t_j)
\]
is smooth.
For each $\nu\in U$,
\[
\!\!\!\gamma_\nu(t) \!:=\!
     \begin{cases}
       \Evol_{K_1}(\nu_1)(t) &                           \mbox{if $t\in[0,t_1]$;}\\
       \Evol_{K_j}(\nu_j)(t)\evol_{K_{j-1}}(\nu_{j-1})\cdots\evol_{K_1}(\nu_1) &\mbox{if $t\in[t_{j-1},t_j]$, $j\geq 2$}\\
       \end{cases}
\]
for $t\in [0,1]$ defines a continuous
function which is piecewise of bounded variation;
thus $\gamma_\nu\in \BV([0,1],G)$.
Using Lemma~\ref{chara-sol-mfd}\,(b),
we see that $\gamma_\nu$ solves the $\BV$-initial
value problem~(\ref{evol-ini});
thus $\gamma_\nu=\Evol(\mu)$.
To complete the proof, we show
that $\Evol|_U$ is smooth.
For each $j\in \{1,\ldots, n\}$,
consider the map
\[
\rho_j\colon \BV([0,1],G)\to \BV(K_j,G),\quad \gamma\mto
\gamma|_{K_j}.
\]
Since
\[
\rho\colon\BV([0,1],G)\to\prod_{j=1}^n\BV(K_j,G),\;\; \gamma\mto(\rho_1(\gamma),\ldots, \rho_n(\gamma))
\]
is an embedding of smooth manifolds
by Lemma~\ref{mfd_splitting_intervall}, the map
$\Evol|_U$ will be smooth if $\rho\circ\Evol|_U$
is smooth, which holds if the components
$\rho_j\circ \Evol|_U$ are smooth for
all $j\in \{1,\ldots,n\}$.
Now
\[
\rho_1\circ\Evol|_U=\Evol_{K_1}\circ \, r_1|_U
\]
is smooth. The map
\[
c_j\colon G\to \BV(K_j,G)
\]
taking $g\in G$ to the constant function~$g$
is a smooth group homomorphism (see Lemma~\ref{to-constant-smooth}).
For $j\in \{2,\ldots,n\}$,
\[
\rho_j(\Evol(\nu))
=\Evol_{K_j}(\nu_j)
c_j(\evol_{K_{j-1}}(\nu_{j-1})\cdots \evol_{K_1}(\nu_1))
\in \BV(K_j,G)
\]
is $C^\infty$ in $\nu\in U$,
whence $\rho_j\circ \Evol|_U$ is smooth
also in this case. $\,\square$
\begin{numba}
Let $\kappa\colon Q\to P$
be a $C^\infty$-diffeomorphism from an open
$0$-neighborhood $Q\sub \cg$ onto an open
$e$-neighborhood $P\sub G$ such that $\kappa(0)=e$
and for all $a<b$ the set
$\BV([a,b],P)$ is open in $\BV([a,b],G)$
and the mapping $\BV([a,b],\kappa)\colon$
$\BV([a,b],Q)\to\BV([a,b],P)$
is a $C^\infty$-diffeomorphism
(see \ref{identity-chart}).
Let $\phi:=\kappa^{-1}\colon P\to Q$;
consider $f\colon \cg\times G\to TG$,
$(x,g)\mto x.g$ as in~(\ref{liegp-f}),
\[
f_\phi\colon \cg\times Q\to\cg,\quad (v,x)\mto
d\phi(f(v,\phi^{-1}(x)))
\]
and the smooth map
\[
h\colon \cg\times Q\to \cg,\quad (v,x)\mto
(f_\phi(\cdot,x))^{-1}(v).
\]
In view of Remark~\ref{rem-explicit-1}, we then have
\begin{equation}\label{hf-inv}
(f_\phi)_*(h_*(\mu,\gamma),\gamma)=\mu\quad\mbox{and}\quad
h_*((f_\phi)_*(\mu,\gamma),\gamma)=\mu
\end{equation}
for all $a<b$,
$\mu\in \Mdna([a,b],\cg)$
and $\gamma\in \cL^\infty_{\rc}([a,b],Q)$.
If $\eta\in \BV([a,b],P)$,
then
\[
\mu:=h_*((\phi\circ \eta)',\phi\circ\eta)
\in \Mdna([a,b],\cg)
\]
satisfies
\[
(f_\phi)_*(\mu,\phi\circ\eta)=(\phi\circ\eta)',
\]
whence $\eta$ is a $\BV$-solution to
\begin{equation}\label{smalleta}
\dot{y}=f_*(\mu,y)
\end{equation}
and hence
$\eta=\Evol_K(\mu)\eta(a)$
with $K:=[a,b]$, by Lemma~\ref{product-ints}.
\end{numba}
\begin{numba}
Let $\eta$ and $\mu$ be as before.
For $g\in G$, we know from Lemma~\ref{product-ints}
that also the map $\eta g\colon [a,b]\to G$,
$t\mto\eta(t)g$ is a solution to $\dot{y}=f_*(\mu,y)$.
If $\eta g\in \BV([a,b],P)$,
this implies that $\phi\circ (\eta g)$ solves
$y'=f_\phi(\mu,y)$; thus
\[
(\phi\circ (\eta g))'=(f_\phi)_*(\mu,\phi\circ (\eta g)).
\]
In view of (\ref{hf-inv}),
this implies
\begin{equation}\label{logader-unchanged}
h_*((\phi\circ (\eta g))',\phi\circ (\eta g))=\mu .
\end{equation}
\end{numba}
\begin{numba}\label{defn-logder}
Let $a<b$ in $\R$ and $K:=[a,b]$.
For each $\eta\in \BV(K,G)$,
there exists a
subdivision
$a=t_0<\cdots<t_n=b$ of~$K$
such that
\[
\eta(t)\eta(s)^{-1}\in P
\]
for all $j\in \{1,\ldots, n\}$ and $t,s\in
[t_{j-1}, t_j]=:K_j$.
Notably,
$\eta_j(t):=\eta(t)\eta(t_{j-1})^{-1}\in P$
for $j\in \{1,\ldots, n\}$ and $t\in K_j$.
We define $\delta^r_K(\eta)\in \Mdna([a,b],\cg)$
via
\[
\delta^r_K(\eta)(A):=\sum_{j=1}^n h_*((\phi\circ\eta_j)',
\phi\circ \eta_j)(A\cap K_j)\quad\mbox{for $\, A\in \cB([a,b])$.}
\]
If $K=[0,1]$,
we write $\delta^r$ in place of $\delta^r_K$.
\end{numba}
\begin{numba}
We mention that $\delta^r_K(\eta)$
is well defined, independent of the choice
of the subdivision $a=t_0<\cdots<t_n=b$.
It suffices to show that a refinement of
the subdivision does not change $\delta^r_K(\eta)$.
It suffices to pass to a refinement
of the form $\{t_0,\ldots, t_n\}\cup\{s\}$
with $s\in \;]t_{j-1},t_j[$
for some $j\in \{1,\ldots,n\}$.
Let $t'_i:=t_i$ for $i\leq j-1$, $t_j':=s$, $t_i':=t_{i-1}$
for $i\in \{j+1,\ldots, n+1\}$.
Let $\zeta_i(t):=\eta(t)\eta(t_i')^{-1}$
for $i\in \{1,\ldots, n+1\}$.
Then $\zeta_i=\eta_i$ for $i\in \{1,\ldots, j-1\}$,
$\zeta_j=\eta_j|_{[t_j,s]}$ and $\zeta_i=\eta_{i-1}$
for $i\in \{j+2,\ldots, n+1\}$.
Moreover,
\[
\zeta_{j+1}(t)=\eta(t)\eta(s)^{-1}=\eta(t)\eta(t_{j-1})^{-1}g
=\eta_j(t)g
\]
for all $t\in [s,t_{j+1}]$ with $g:=\eta(t_{j-1})\eta(s)^{-1}$,
where $\eta_j(t)\in P$ and $\zeta_{j+1}(t)\in P$.
Thus
\[
h_*((\phi\circ \zeta_{j+1})',\phi\circ \zeta_{j+1})=h_*((\phi\circ \eta_j)',\phi\circ \eta_j)|_{[s,t_{j-1}]}
\]
using (\ref{logader-unchanged}).
Hence
$\sum_{j=1}^nh_*((\phi\circ \eta_j)',\phi\circ \eta_j)(K_j\cap A)=\sum_{i=1}^{n+1}
h_*((\phi\circ \zeta_i)',\phi\circ \zeta_i)([t_{i-1}',t_i']\cap A)$
for all $A\in \cB([a,b])$.
\end{numba}
\begin{prop}\label{props-logder}
Let $a<b$ in $\R$ and $K:=[a,b]$.
\begin{itemize}
\item[\rm(a)]
The map $\delta^r_K\colon \BV(K,G)\to \Mdna(K,\cg)$,
$\eta\mto \delta^r(\eta)$
is smooth.
\item[\rm(b)]
Each $\eta\in \BV(K,G)$
solves the $\BV$-differential equation
$\dot{y}=f_*(\delta^r_K(\eta),y)$.
Thus $\eta=\Evol_K(\delta^r_K(\eta))\eta(a)$.
\item[\rm(c)]
For $\eta,\zeta\in \BV(K,G)$,
we have $\delta^r(\eta)=\delta^r(\zeta)$
if and only if there exists $g\in G$
such that $\eta(t)=\zeta(t) g$
for all $t\in K$.
\item[\rm(d)]
The map $\delta^r_K$ restricts to a $C^\infty$-diffeomorphism
$\BV_*(K,G)\to \Mdna(K,\cg)$
with inverse $\mu\mto\Evol_K(\mu)$.
\end{itemize}
\end{prop}
\begin{proof}
(a) Let $B\sub G$ be an open $e$-neighborhood
such that $BB\sub P$. Given $\gamma\in \BV([a,b],G)$,
its image $\gamma([a,b])$
is compact. Hence, there is an open $e$-neighborhood $A\sub G$
such that $gAg^{-1}\sub B$
for all $g\in \gamma([a,b])$.
Let $S\sub G$ be an open $e$-neighborhood such that
$S S^{-1}\sub A$. We show that $\delta^r_K$
is smooth on the open $\gamma$-neighborhood $\gamma \BV(K,S)$
in $\BV(K,G)$. Let $a=t_0<\cdots< t_n=b$
be a subdivision of~$K$ such that
\[
\gamma(t)\gamma(s)^{-1}\in B\quad\mbox{for all $j\in \{1,\ldots, n\}$ and $s,t\in [t_{j-1},t_j]=:K_j$.}
\]
Let $\eta\in \gamma \BV(K,S)$.
For $j$, $s$, and $t$ as before, we have
$\eta(t)=\gamma(t)x$ and $\eta(s)=\gamma(s)y$
for certain $x,y\in S$. Then $xy^{-1}\in A$,
whence
\[
\eta(t)\eta(s)^{-1}=\gamma(t)xy^{-1}\gamma(s)^{-1}
=\underbrace{\gamma(t)xy^{-1}\gamma(t)^{-1}}_{\in B}\underbrace{\gamma(t)\gamma(s)^{-1}}_{\in B}
\in BB\sub P.
\]
We can therefore use the previous subdivision
$t_0<\cdots<t_n$ to calculate the $\eta_j$ and $\delta^r_K(\eta)$
simultaneously for each $\eta\in \gamma \BV(K,S)$.
Note that the map
\[
\theta\colon \Mdna([a,b],\cg)\to \bigoplus_{j=1}^n\Mdna(K_j,\cg),\;\;
\mu\mto (\mu|_{K_j})_{j=1}^n
\]
is a surjective isometry if we define $\|(\mu_1,\ldots,\mu_n)\|:=\sum_{j=1}^n\|\mu_j\|$ for $\mu_j\in \Mdna(K_j,\cg)$,
with
\[
\theta^{-1}(\mu_1,\ldots, \mu_n)(A)=\sum_{j=1}^n\mu_j(A\cap K_j)
\quad\mbox{for $A\in \cB([a,b])$.}
\]
It therefore suffices to show that
$\theta\circ\delta^r_K$ is smooth on $\gamma\BV(K,S)$,
which holds if each component
\[
\eta\mto h_*((\phi\circ \eta_j)',\phi\circ \eta_j)
\]
is smooth there. Now
$\BV(K_j,\phi)\colon \BV(K_j,P)\to\BV(K_j,Q)$
is smooth for all $j\in \{1,\ldots, n\}$.
Moreover,
$h_*\colon \Mdna(K_j,\cg)\times C(K_j,Q)\to \Mdna(K_j,\cg)$
is smooth and $\BV(K_j,\cg)\to\Mdna(K_j,\cg)$, $\zeta\mto\zeta'$ is continuous and linear and thus smooth.
It only remains to note that $\eta_j=\eta|_{K_j}\eta(t_{j-1})^{-1}\in \BV(K_j,G)$ is a smooth function
of $\eta\in \BV(K,G)$
because $\BV(K_j,G)$ is a Lie group,
$G$ has a smooth inversion map,
the map $G\to \BV(K_j,G)$, $g\mto (t\mto g)$
is smooth (see Lemma~\ref{to-constant-smooth}\,(a)),
the restriction map
$\BV(K,G)\to \BV(K_j,G)$, $\zeta\mto\zeta|_{K_j}$
is smooth (cf.\ Lemma~\ref{mfd_splitting_intervall})
and the point evaluation $\BV(K,G)\to G$,
$\eta\mto\eta(t_{j-1})$ is smooth
(see Corollary~\ref{cor:smooth_eval}).\smallskip

(b) Given $\eta\in \BV(K,G)$,
let $t_1<\cdots<t_n$ and $\eta_1,\ldots,\eta_n$
be as in \ref{defn-logder}. For each $j\in \{1,\ldots,n\}$,
we know that $\eta_j$
solves $\dot{y}=f_*(\delta^r_K(\eta)|_{K_j},y)$
with $K_j:=[t_{j-1},t_j]$
(cf.\ (\ref{smalleta})),
whence also $\eta|_{K_j}$ solves this
differential equation, by Lemma~\ref{product-ints}.
Hence $\eta$ solves $\dot{y}=f_*(\delta^r_K(\eta),y)$,
whence $\eta=\Evol_K(\delta^r_K(\eta))\eta(a)$
(see Lemma~\ref{product-ints}).\smallskip

(c) If $\eta=\zeta g$,
then $\eta_j=\zeta_j$ in \ref{defn-logder}
and hence $\delta^r_K(\eta)=\delta^r_K(\zeta)$.
That $\eta=\zeta g$ for some $g\in G$
whenever $\delta^r_K(\eta)=\delta^r_K(\zeta)$
is immediate from~(b).\smallskip

(d) Since $\BV_*(K,G)$ is a submanifold of $\BV(K,G)$,
the restriction of $\delta^r_K$ to a map $D\colon \BV_*(K,G)\to \Mdna(K,\cg)$
is smooth. By~(c), $D$ is injective.
Conversely, for each $\mu\in \Mdna(K,\cg)$,
we have $\Evol_K(\mu)\in \BV_*(K,G)$.
As the latter is a submanifold and $\Evol_K$ is smooth,
also the corestriction
$E\colon \Mdna(K,\cg)\to\BV_*(K,G)$, $\mu\mto\Evol_K(\mu)$
is~$C^\infty$. If
$\mu\in \Mdna(K,\cg)$ is given,
let $\eta:=\Evol_K(\mu)$.
Let the subdivision
$a=t_0<\cdots< t_n=b$
as well as $\eta_j$ and $K_j:=[t_{j-1},t_j]$
for $j\in \{1,\ldots, n\}$
be as in \ref{defn-logder}. Then $\eta_j$
solves $\dot{y}=f_*(\delta^r_K(\eta)|_{K_j},y)$,
whence $(\phi\circ \eta_j)'=
(f_\phi)_*(\delta^r_K(\eta)|_{K_j},\phi\circ \eta_j)$.
But $\eta|_{K_j}$ solves $\dot{y}=f_*(\mu|_{K_j},y)$,
whence also $\eta_j$ solves $\dot{y}=f_*(\mu|_{K_j},y)$.
Hence $(\phi\circ \eta_j)'=(f_\phi)_*(\mu|_{K_j},\phi\circ\eta_j)$. Then
\begin{eqnarray*}
\mu|_{K_j}&=& h_*((f_\phi)_*(\mu|_{K_j},\phi\circ \eta_j),
\phi\circ\eta_j)
=
h_*((\phi\circ \eta_j)',\phi\circ\eta_j)\\
&=& h_*((f_\phi)_*((\delta^r_K(\eta))|_{K_j},\phi\circ\eta_j),
\phi\circ\eta_j)
=\delta^r_K(\eta)|_{K_j}
\end{eqnarray*}
for all $j\in \{1,\ldots, n\}$,
whence $\mu=\delta^r_K(\eta)$ and thus
$\delta^r_K(\Evol(\mu))=\mu$.
\end{proof}
\begin{rem}
As before, let $K:=[a,b]$.
Considering $\Evol_K$ as a bijection onto $\BV_*(K,G)$,
we have
\[
\delta^r_K(\eta)=(\Evol_K)^{-1}(\eta \eta(a)^{-1})
\]
for each $\eta\in \BV(K,G)$, by (c) and (d) in Proposition~\ref{props-logder}.
Hence $\delta^r_K(\eta)$ is independent of
the choices of~$\kappa$ and $S$
in the above definition.
\end{rem}
\noindent
For $g\in G$, let $c_g\colon [0,1]\to G$
be the constant function with value~$g$.
\begin{cor}\label{semid-again}
The mapping
\[
\Mdna([0,1],\cg)\times G\to \BV([0,1],G),\quad
(\mu,g)\mto \Evol(\mu)c_g
\]
is a $C^\infty$-diffeomorphism.
\end{cor}
\begin{proof}
In view of Remark~\ref{semidirect-lie},
the map $\pi\colon \BV_*([0,1],G)\times G\to
\BV([0,1],G)$, $(\gamma,g)\mto \gamma c_g$
is a $C^\infty$-diffeomorphism.
The assertion follows as also the mapping
$\Evol\colon \Mdna([0,1],\cg)\to \BV_*([0,1],G)$
is a $C^\infty$-diffeomorphism,
by Proposition~\ref{props-logder}\,(d).
\end{proof}
\begin{rem}
Since $\Mdna([0,1],\cg)$
is a Banach space and hence contractible,
the corollary entails that $\BV([0,1],G)$
and~$G$ are homotopy equivalent.
\end{rem}
\begin{appendix}
\section{Detailed proofs for Section \ref{sect:BV-mfd}}\label{app:details_tangent}
To identify the tangent manifold of the manifold $\BV([a,b],M)$ of $\BV$-functions, we first need a technical result for the identification of spaces of sections in the double tangent bundle $T^2M = T(TM)$.

\begin{la}\label{lem:T2_split}
Let $f \in \BV ([a,b],M)$ and $0 \colon M \rightarrow TM$ the zero-section. Denoting by $\kappa \colon T^2M \rightarrow T^2M$ the canonical flip of the double tangent bundle, we let $\Theta (v,w) \coloneq \kappa(T\lambda_p (v,w))$
for $v,w \in T_pM$. Then the following holds:
\begin{enumerate}
\item[\rm{(a)}] The map
\begin{align}\label{Theta_iso}
\Gamma_f \times \Gamma_f \rightarrow \Gamma_{0\circ f} (T^2M), \;\;
(\sigma , \tau) \mto \Theta\circ (\sigma, \tau)
\end{align}
is an isomorphism of Banach spaces.
\item[\rm{(b)}] Let $U \subseteq TM$ be an open neighborhood of the zero-section $0(M)$. Then \eqref{Theta_iso} restricts to a diffeomorphism
$$\Xi_f \colon \{\tau \in \Gamma_f \colon \tau([a,b])\subseteq U\} \times \Gamma_f \rightarrow \{\gamma \in \Gamma_{0\circ f} \colon \gamma([a,b]) \subseteq \kappa(TU)\}.$$
\end{enumerate}
\end{la}

\begin{proof}
(a) Note that $\Gamma_{0\circ f}(T^2M) =\Gamma_f (\pi^{-1}_{TM}(0(M)))$ as a set and as a topological space (as can be seen by direct inspection of the norms). A pointwise calculation shows that also the vector space structures coincide, whence the two spaces coincide as Banach spaces. As in Lemma \ref{lem:bundle-calc} (b), we identify $\Gamma_f \times \Gamma_f$ with $\Gamma_f (TM \oplus TM)$, whence by \cite[Lemma A.20 (a)]{AGS} and Lemma \ref{lem:bundle-calc} (a):
$$\Theta \circ (\sigma , \tau) = \Gamma_f (\Theta) (\sigma,\tau) \in \Gamma_f (\pi^{-1}_{TM}(0(M)))$$
for all $(\sigma,\tau) \in \Gamma_f \times \Gamma_f$. The map in \eqref{Theta_iso} therefore coincides up to identification with $\Gamma_f(\Theta)$ which is an isomorphism with inverse $\Gamma_f (\Theta^{-1})$ (combining \cite[Lemma A.20 (a)]{AGS} with Lemma \ref{lem:bundle-calc} (a)).

(b) can be established exactly as in \cite[Lemma A.20 (e)]{AGS}.
\end{proof}
\noindent
We shall use this result to establish the identification of the tangent manifold as a manifold of $\BV$-functions with values in the tangent manifold $TM$.

\begin{proof}[\bf Proof of Proposition \ref{prop:tangent_ident}]
Without loss of generality, we may assume that the local addition $\Sigma$ on $M$ is normalized. We use again for $f \in \BV ([a,b],M)$ the smooth parametrization $\phi_f \colon O_f \rightarrow O_f'\subseteq \BV([a,b],M)$ which maps $0 \in \Gamma_f$ to $f$. Thus $T\phi_f (0,\cdot) \colon \Gamma_f \rightarrow T_f \BV ([a,b],M)$ is an isomorphism of Banach spaces. For $\tau \in \Gamma_f$, since $\Sigma$ is normalized we have for each $x \in [a,b]$ 
\begin{align*}
T\varepsilon_x T \phi_f (0,\tau) = T\Sigma|_{T_{f(x)}M}(\tau(x))=\tau(x).
\end{align*}
We deduce that $\Phi (T\phi_f (0,\tau))=\tau \in \Gamma_f \subseteq \BV([a,b], TM)$, whence $\Phi (v)\in \Gamma_f$ for each $v \in T_f \BV([a,b],M)$ and $\Phi$ is a linear bijection of $T_f \BV([a,b],M)$ onto $\Gamma_f$.
This already implies that $\Phi$ is a bijection. Thus it suffices to prove that $\Phi$ is a smooth diffeomorphism. If this is true, $\BV([a,b],\pi_M) \circ \Phi = \pi_{\BV ([a,b],M)}$ implies that $\BV([a,b],\pi_M)$ is smooth and the projection of the vector bundle we wished to construct.

The sets $S_f \coloneq T\phi_f (O_f \times \Gamma_f)$ form an open cover of $T\BV([a,b],M)$ for $f \in \BV([a,b],M)$. Let $\Xi_f$ be the diffeomorphism from Lemma \ref{lem:T2_split} (b). Then the images 
$$\Phi(S_f) = (\phi_{0\circ f} \circ \Xi_f)(O_f\times \Gamma_f)=\phi_{0\circ f}(O_{0\circ f})$$ are open sets which cover $\BV([a,b],TM)$. To establish that $\Phi$ is a smooth diffeomorphism it therefore suffices to show for each $f \in \BV([a,b],M)$ that
$\Phi \circ T\phi_f = \phi_{0\circ f}\circ \Xi_f$.
Let $\lambda_p \colon T_pM \rightarrow TM$ be the inclusion of the fibre for $p \in M$. Now we use that tangent vectors can be represented as equivalence classes of curves passing through their base point to write
$$T\phi_f (\sigma, \tau) = [t \mapsto \Sigma \circ (\sigma +t \tau)], \quad (\sigma ,\tau) \in O_f \times \Gamma_f.$$
A immediate computation shows $\Phi(T\phi_f (\sigma,\tau))$ equals
\begin{align*}
([t\mapsto& \Sigma \circ \lambda_{f(x)}(\sigma(x)+t\tau(x))])_{x \in [a,b]} \hspace{3cm}\\
=& (T(\Sigma \circ \lambda_{f(x)})(\sigma(x),\tau(x)))_{x \in [a,b]}\\ 
=& (\Sigma_{TM} ((\kappa \circ T\lambda_{f(x))})(\sigma(x),\tau(x)))_{x \in [a,b]}\\
=&((\Sigma_{TM}\circ \Xi_f)(\sigma,\tau)(x))_{x\in [a,b]} = (\phi_{0\circ f}\circ \Xi_f)(\sigma,\tau)
\end{align*}
Hence $\Phi$ is a diffeomorphism.
\end{proof}
\noindent
Next, we identify the manifold of $\BV$-functions with values in a product as a product of manifolds of $\BV$-functions. 

\begin{proof}[\bf Proof of Lemma \ref{lem:product_mfd}]
The map sending a $\BV$-function in $\BV([a,b],M_1 \times M_2)$ to its components is the pushforward $\BV([a,b], (\pr_1,\pr_2))$, hence smooth by Proposition \ref{prop:smooth_pushfwd}. Clearly it is also a bijection. Hence, we only need to prove that its inverse $\theta \colon \BV([a,b],M_1) \times \BV([a,b],M_2) \rightarrow \BV([a,b],M_1\times M_2)$ is smooth. For this, we recall that if $\Sigma_i$ is the local addition on $M_i, i\in \{1,2\}$, then $\Sigma_1\times \Sigma_2$ is a local addition on $M_1\times M_2$. Let now $(f,g) \in \BV([a,b],M_1 \times M_2)$; then $\theta (O_f \times O_g) = O_{(f,g)}$, by construction. Thus,
it suffices to prove that $\theta_{f,g} \coloneq \varphi_{(f,g)}\circ \theta \circ  \varphi_f^{-1} \times \varphi_g^{-1}$ is smooth for every pair $(f,g)$. To verify the latter, we localize $\Gamma_f,\Gamma_g$ and $\Gamma_{(f,g)}$ using families of manifold charts $(\kappa_i)_i, (\lambda_i)_i$ and $(\kappa_i \times \lambda_i)_i$ via $\eqref{thetamap}$. Following Remark \ref{rem:cont:inclusion}, we see that the maps \eqref{thetamap} conjugate $\theta_{f,g}$ in each component to the restriction of the insertion which is smooth by Lemma \ref{prod:iso}. As the topology on $\Gamma_{(f,g)}$ is is initial with respect to the map \eqref{thetamap}, we deduce that $\theta_{f,g}$ is smooth. This concludes the proof.
\end{proof}
\end{appendix}


\begin{thebibliography}{99}
%
\bibitem{AGS}
Amiri, H., H. Gl\"{o}ckner, and A. Schmeding,
\emph{Lie groupoids of mappings taking values in a Lie groupoid},
Arch.\ Math.\ (Brno) {\bf 56} (2020),
307--356.
%
%
\bibitem{Bas}
Bastiani, A., \emph{Applications diff\'erentiables et vari\'et\'es diff\'erentiables 
de dimension infinie}, J. Anal.\ Math.\ \textbf{13} (1964), 1--114.
%
%
\bibitem{BGN}
Bertram, W., H. Gl\"{o}ckner, and
K.-H. Neeb,
\emph{Differential calculus over general base fields and rings},
Expo.\ Math.\ {\bf 22} (2004), 213--282.
%
%
\bibitem{Sig2}
 Chevyrev, I. and A. Kormilitzin, \emph{A primer on the signature method in machine learning}, preprint, arXiv:1603.03788, 2016.
%
%
\bibitem{Din}
Dinculeanu, N.,
``Vector Measures,''
Pergamon Press, Oxford, 1967.
%
%
\bibitem{EaM}
Ebin, D. G. and J. Marsden,
\emph{Groups of diffeomorphisms and the motion
of an incompressible fluid},
Ann.\ of Math.\ {\bf 92} (1970), 102--163.
%
%
\bibitem{Eel}
Eells, J. Jr., \emph{A setting for global analysis},
Bull.\ Amer.\ Math.\ Soc.\ {\bf 72} (1966), 751--807.
%
%
\bibitem{FaK}
Flaschel, P. and W. Klingenberg,
``Riemannsche Hilbertmannigfaltigkeiten,''
Springer, Berlin, 1972.
%
%
\bibitem{FaH14}
Friz, P. K. and M. Hairer, ``A Course on Rough Paths. With an Introduction to Regularity Structures,'' Springer, Cham, 2014.
%
%
\bibitem{FaV10}
Friz, P. K. and N. B. Victoir, ``Multidimensional Stochastic Processes as Rough Paths. Theory and Applications,''
Cambridge University Press, Cambridge, 2010.
%
%
\bibitem{Res}
Gl\"{o}ckner, H.,
\emph{Infinite-dimensional Lie groups without completeness restrictions},
pp.\ 43--59 in: Strasburger, A. et al.\ (eds.),
``Geometry and Analysis on Finite- and Infinite-Dimensional Lie Groups,''
Banach Center Publ.\ {\bf 55}, Warsaw, 2002.
%
%
\bibitem{CAN}
Gl\"{o}ckner, H.,
\emph{Lie groups of measurable mappings},
Canadian J. Math.\ {\bf 55} (2003),
%no. 5,
969--999.
%
%
\bibitem{Asp}
Gl\"{o}ckner, H.,
\emph{Aspects of differential calculus related
to infinite-dimensional vector bundles
and Poisson vector spaces},
Axioms 2022, 11, 221.
%
%
\bibitem{FXX}
Gl\"{o}ckner, H.,
\emph{Finite order differentiability properties,
fixed points and implicit functions over valued fields},
preprint, arXiv:math/0511218.
%
%
\bibitem{Mea}
Gl\"{o}ckner, H.,
\emph{Measurable regularity properties
of infinite-dimensional Lie groups},
preprint, arXiv:1601.02568.
%
%
\bibitem{Sem}
Gl\"{o}ckner, H.,
\emph{Regularity properties of infinite-dimensional Lie groups, and
semiregularity}, preprint, arXiv:1208.0715.
%
%
\bibitem{GaH}
Gl\"{o}ckner, H. and J. Hilgert,
\emph{Aspects of control theory on infinite-dimensional Lie
groups and $G$-manifolds},
J. Differential Equations {\bf 343} (2023), 186--232.
%
%
\bibitem{GaN} Gl\"{o}ckner, H. and K.-H. Neeb,
``Infinite-Dimensional Lie Groups,'' book in preparation.
%
%
\bibitem{GS22} Gl{\"o}ckner, H. and A. Schmeding, \emph{Manifolds of mappings on Cartesian products}, Ann.\ Global Anal.\ Geom.\ {\bf 61},
%No. 2,
(2022), 359--398.
%
%
\bibitem{GaNaS22}
Grong, E., T. Nilssen, and A. Schmeding, \emph{Geometric rough paths on infinite dimensional spaces}, J. Differential Equations {\bf 340} (2022) 151--178.
%
%
\bibitem{HaL10}
Hambly, B. and T. Lyons, \emph{Uniqueness for the signature of a path of bounded variation and the reduced path group} Ann.\ Math. {\bf 171}
%No. 1,
(2010), 109--167.
%
%
\bibitem{Ham}
Hamilton, R. S.,
\emph{The inverse function theorem of Nash and Moser},
Bull.\ Amer.\ Math.\ Soc.\ {\bf 7} (1982), 65--222.
%
%
\bibitem{Hn2}
Hanusch, M.,
\emph{The strong Trotter property for locally $\mu$-convex Lie groups},
J. Lie Theory {\bf 30} (2020),
%no. 1,
25--32.
%
%
\bibitem{Han}
Hanusch, M.,
\emph{A $C^k$-Seeley-extension-theorem for
Bastiani's differential calculus},
Can.\ J. Math.\ {\bf 75} (2023),
%No. 1,
170--201.
%
%
\bibitem{Sig3}
Kalsi, J., T. Lyons, and I. P. Arribas, \emph{Optimal execution with rough path signatures},
SIAM Journal on Financial Mathematics {\bf 11},
%No. 2,
(2020), 470--493.
%
%
\bibitem{Kel}
Keller, H.~H., ``Differential Calculus
in Locally Convex Spaces'', Springer,
Berlin,
1974.
%
%
\bibitem{Kli}
Klingenberg, W. P. A.,
``Riemannian Geometry,''
de Gruyter, Berlin, ${}^2$1995. 
%
%
\bibitem{KaM} Kriegl, A. and P. W. Michor,
``The Convenient Setting of Global Ana\-lysis,''
Amer.\ Math.\ Soc., Providence, 1997.
%
%
\bibitem{Lang} Lang, S.
``Fundamentals of Differential Geometry'', Springer, New York,
1999.
%
%
\bibitem{MRO}
Margalef-Roig, J. and E. Outerelo Dominguez,
``Differential Topology,''
North-Holland, Amsterdam, 1992.
%
%
\bibitem{Mic}
Michor, P. W., ``Manifolds of Differentiable Mappings'',
Shiva, Orpington, 1980.
%
%
\bibitem{Mil} Milnor, J., \emph{Remarks on infinite-dimensional Lie groups},
pp.\,1007--1057 in: B.\,S. DeWitt and R. Stora (eds.),
``Relativit\'{e}, groupes et topologie II,'' North-Holland,
Amsterdam, 1984.
%
%
\bibitem{MaV}
Moreau, J. J. and M. Valadier, \emph{A chain rule involving vector functions of bounded variation}, J. Funct.\ Anal.\ {\bf 74} (1987),
333--345.
%
%
\bibitem{Nee}
Neeb, K.-H.,
\emph{Towards a Lie theory of locally convex groups},
Jpn.\ J. Math.\ {\bf 1} (2006), 291--468. 
%
%
\bibitem{Nik}
Nikitin, N.,
``Regularity Properties of Infinite-Dimensional Lie Groups and Exponential Laws,''
doctoral dissertation, Paderborn University,
2021 (see nbn-resolving.de/urn:nbn:de:hbz:466:2-39133)
%
%
\bibitem{Pal}
Palais, R. S.,
``Foundations of Global Non-Linear Analysis,''
W. A. Benjamin, New York, 1968.
%
%
\bibitem{Pin}
Pinaud, M.,
\emph{Manifolds of absolutely continuous functions},
manuscript, Paderborn University, 2023.
%
%
\bibitem{Rud}
Rudin, W., ``Real and Complex Analysis,''
McGraw-Hill, New York, 1987.
%
%
\bibitem{Sche}
Schechter, E., ``Handbook of Analysis and its Foundations,'' Academic Press,
San Diego, 1997.
%
%
\bibitem{Sme}
Schmeding, A.,
\emph{Manifolds of absolutely continuous curves
and the square root velocity framework},
preprint, arXiv:1612.02604.
%
%
\bibitem{Schm}
Schmeding, A., ``An Introduction to Infinite-Dimensional Differential Geometry,''
Cambridge University Press, Cambridge,
2022.
%
%
\bibitem{Ste}
Stegemeyer, M.,
``On Global Properties of Geodesics.
The String Topology Coproduct
and Geodesic Complexity,''
doctoral dissertation,
Leipzig University,
2023.
%https://ul.qucosa.de/api/qucosa%3A82992/attachment/ATT-0/
%
%
\bibitem{Sig1}
Yang, W., T. Lyons, H. Ni, C, Schmid, and L. Jin, \emph{Developing the path signature methodology and its application to landmark-based human action recognition}, pp.\ 431--464 in:: G. Yin and T. Zariphopoulou (eds.),
``Stochastic Analysis, Filtering, and Stochastic Optimization,'' Springer, Cham, 2022.
%
\end{thebibliography}
\end{document}